\documentclass[11pt]{article}
\usepackage[numbers,square]{natbib}
\usepackage{amsmath,amsthm,amsfonts,amssymb,commath}
\usepackage{graphicx,subcaption,pgfplots}
\usepackage{xspace}
\usepackage{color}
\usepackage{tikz}
\usetikzlibrary{math,arrows,arrows.meta,calc,patterns,patterns.meta}
\pgfplotsset{compat=1.16}
\usepackage{stmaryrd}
\usepackage{url}
\usepackage{footnote}
\makesavenoteenv{tabular}
\makesavenoteenv{table}
\usepackage[hidelinks]{hyperref}
\hypersetup{colorlinks = true, urlcolor = blue,
  linkcolor = blue, citecolor = blue}
\usepackage{enumitem,cleveref}
\usepackage[margin=1.0in]{geometry}
\makeatletter
\newcommand{\tnorm}{\@ifstar\@tnorms\@tnorm}
\newcommand{\@tnorms}[1]{%
  \left|\mkern-1.5mu\left|\mkern-1.5mu\left|
   #1
  \right|\mkern-1.5mu\right|\mkern-1.5mu\right|
}
\newcommand{\@tnorm}[2][]{%
  \mathopen{#1|\mkern-1.5mu#1|\mkern-1.5mu#1|}
  #2
  \mathclose{#1|\mkern-1.5mu#1|\mkern-1.5mu#1|}
}
\makeatother
\newcommand{\jump}[1]{\llbracket #1 \rrbracket}
\newcommand{\ejump}[1]{\langle\!\langle #1 \rangle\!\rangle}
\newtheorem{theorem}{Theorem}[section]
\newtheorem{lemma}[theorem]{Lemma}
\newtheorem{remark}{Remark}
\newtheorem{definition}{Definition}
\numberwithin{equation}{section}
\title{A posteriori error analysis of a space-time hybridizable
discontinuous Galerkin method for the advection-diffusion
problem}
\author{Y. Wang\thanks{Department of Applied Mathematics, University of
    Waterloo, ON, Canada (\url{yuan.wang@uwaterloo.ca}),
    \url{http://orcid.org/0009-0006-8092-4378}
 } \and
  S. Rhebergen\thanks{Department of Applied Mathematics, University of
    Waterloo, ON, Canada (\url{srheberg@uwaterloo.ca}),
    \url{http://orcid.org/0000-0001-6036-0356}}}
\begin{document}
\maketitle
\begin{abstract}
  We present and analyze an a posteriori error estimator for a
  space-time hybridizable discontinuous Galerkin discretization
  of the time-dependent advection-diffusion problem. The
  residual-based error estimator is proven to be reliable and
  locally efficient. In the reliability analysis we combine a
  P\'eclet-robust coercivity type result and a saturation
  assumption, while local efficiency analysis is based on using
  bubble functions. The analysis considers both local space and
  time adaptivity and is verified by numerical simulations on
  problems which include boundary and interior layers.
\end{abstract}

\section{Introduction}
\label{s:introduction}

In this paper we present an a posteriori error analysis for a
space-time hybridizable discontinuous Galerkin (HDG) method for the
time-dependent advection-diffusion problem. When advection dominates,
the solution to such problems may admit sharp boundary and/or interior
layer(s) and a uniform refinement strategy may be inefficient to
reduce the numerical error when these layers are present. Special
layer-adapted meshes have been devised for layer problems
\cite{Roos:book}. However, when the location of the layer is unknown,
which is typically the case for time-dependent problems, adaptive mesh
refinement (AMR) may be a more desirable approach.

A posteriori error analysis has been studied for the stationary
advection-diffusion problem, in the advection-dominated limit, for
various finite element methods. Examples include a posteriori error
analysis for conforming finite element methods
\cite{Kunert:2003,Sangalli:2008,Verfurth:1998,Verfurth:2005},
discontinuous Galerkin methods (DG)
\cite{Ern:2008,Ern:2010,Schotzau:2014,Schotzau:2009,Schotzau:2011a},
and HDG methods \cite{Araya:2019,Chen:2016,Natasha:2021}. The focus
of these studies has been the robustness of the error estimator, i.e.,
that the ratio between the estimated error and the true error is
independent of the P\'eclet number.

Nonrobustness of the error estimator with respect to the standard
energy norm was first observed in \cite{Verfurth:1998}. To attain
robustness, a dual norm as a measurement of the error in the advective
gradient is devised \cite{Verfurth:2005b}. This approach was also used in
\cite{Ern:2010,Schotzau:2009,Verfurth:2005}. Specifically, a
P\'eclet-robust continuous inf-sup stability condition involving the
dual norm helps establish a reliability bound with a constant
independent of the P\'eclet number. However, the dual norm cannot be
evaluated locally and so is not suitable for practical
computations. To address this, \cite{Verfurth:2005b} computes a bound
for the dual norm, but at the expense of needing to solve an
auxiliary stationary reaction-diffusion problem. An alternative dual
norm, more suitable for advection-dominated problems, was presented in
\cite{Sangalli:2008}. Their residual-based estimator was shown to be
almost robust in one spatial dimension. An a posteriori error
analysis, not involving dual norms, was presented in
\cite{Chen:2016}. Using a weighted test function, \cite{Ayuso:2009}
proved a P\'eclet-robust discrete inf-sup condition for a DG
discretization of the advection-diffusion problem in the
advection-dominated regime (see \cite[Theorem 4.4]{Ayuso:2009}). The a
posteriori error analysis of \cite{Araya:2019,Chen:2016} is based on a
P\'eclet-robust coercivity type result (see \cite[Lemma
4.1]{Chen:2016}) similar to \cite[Lemma 4.1]{Ayuso:2009} (an
intermediate result used to obtain the P\'eclet-robust inf-sup
condition).

For the time-dependent advection-diffusion problem, various a
posteriori error estimators have been introduced
\cite{Araya:2014,Berrone:2004,Cangiani:2019,Cangiani:2014,Ern:2005,Verfurth:2005b}.
For example, a space-time version of the error estimator for the
stationary problem in \cite{Verfurth:1998} is derived and analyzed in
\cite{Araya:2014}, inheriting the nonrobustness. In contrast, based on
\cite{Verfurth:2005}, the dual norm technique is applied to the
time-dependent problem in \cite{Verfurth:2005b} in order to attain
P\'eclet robust space-time adaptivity. Furthermore, elliptic
reconstruction \cite{Georgoulis:2011,Lakkis:2006,Makridakis:2003}
provides a framework to extend error estimators for the stationary
problem to the time-dependent problem. This is used, for example, in
\cite{Cangiani:2019,Cangiani:2014} to extend the error estimator in
\cite{Schotzau:2009} for stationary problems to time-dependent
problems.

In this paper we consider a posteriori error analysis of a second
order accurate in time and arbitrary order accurate in space
space-time HDG discretization of the time-dependent
advection-diffusion problem \cite{Rhebergen:2013}. An a priori error
analysis of this discretization was presented in \cite{Kirk:2019} and
later extended to the advection-dominated regime in
\cite{Wang:2023}. In \cite[Theorem 4.1]{Wang:2023} we proved a
P\'eclet-robust discrete inf-sup condition. Analogous to the approach
taken in \cite{Chen:2016} for the stationary problem, the basis for
the a posteriori error analysis in this paper is the intermediate
P\'eclet-robust coercivity result \cite[Lemma
4.2]{Wang:2023}. Additionally, since our discretization is second
order accurate in time, we use a saturation assumption (inspired by
\cite{Burman:2009}) to bound the error in the time-derivative. We
remark that despite a nonrobust a posteriori error bound, as shown in
\cref{thm:reliability} and \cref{thm:efficiency}, the norm we use is locally
computable. Furthermore, the error estimator in this paper is fully
local hence it is an estimator for local space and time adaptivity in
the AMR procedure.

The paper is organized as follows. In \cref{s:advecdiffuprob}, we
introduce the space-time formulation of the time-dependent
advection-diffusion problem. In \cref{s:spacetimehdg} we describe the
finite element spaces and the space-time HDG discretization of the
advection-diffusion problem. The main results, the reliability and the
local efficiency of the error estimator, are presented and proved in
\cref{s:apos}. In \cref{s:numericalEx}, we present numerical examples,
particularly ones with boundary and interior layers, and compare the
results with the error analysis in \cref{s:apos}. Finally, conclusions
are drawn in \cref{s:conclusions}.

\section{The advection-diffusion problem}
\label{s:advecdiffuprob}

Let $I := (0,T)$ be a time interval of interest with $T \ge 1$, let
$\Omega \subset \mathbb{R}^d$ be a polygonal ($d=2$) or a polyhedral
($d=3$) domain, and let the boundary of $\Omega$ be partitioned as
$\partial\Omega = \Gamma_D \cup \Gamma_N$ where
$\Gamma_D \cap \Gamma_N = \emptyset$. Furthermore, let
$\mathcal{E} := I \times \Omega$ be a $(d+1)$-dimensional domain that
is Lipschitz and let $\Omega_0 = \cbr[0]{0} \times \Omega$ and
$\Omega_T = \cbr[0]{T} \times \Omega$. We partition the boundary of
$\mathcal{E}$ as
$\partial\mathcal{E}=\partial\mathcal{E}_D\cup\partial\mathcal{E}_N$
where $\partial\mathcal{E}_D \cap \partial \mathcal{E}_N = \emptyset$
and $\partial \mathcal{E}_D := \Gamma_D \times I$. We denote by
${n}:=\del{{n}_t,\overline{{n}}} \in \mathbb{R}^{(d+1)}$ the outward
space-time normal vector to $\partial\mathcal{E}$, where
${n}_t \in \mathbb{R}$ and $\overline{{n}} \in \mathbb{R}^d$ denote
the temporal and spatial components of $n$, respectively.

We consider the space-time formulation of the time-dependent
advection-diffusion equation:
\begin{subequations}
  \label{eq:advdif}
  \begin{align}
    \label{eq:st_adr}
    \nabla\cdot\del{{\beta}u}
    -\varepsilon\overline{\nabla}^2u
    &=f
    &&
       \text{in }
       \mathcal{E},
    \\
    \label{eq:st_adr_bcN}
    -\zeta^- u \beta \cdot{n}
    +\varepsilon\overline{\nabla}u\cdot\bar{n}
    &=g
    &&
       \text{on }
       \partial\mathcal{E}_N,
    \\
    \label{eq:st_adr_bcD}
    u&=0
    &&
       \text{on }
       \partial\mathcal{E}_D,
  \end{align}
\end{subequations}
where
$\overline{\nabla}=\del{\partial_{x_1},\partial_{x_2},\dots,\partial_{x_d}}$,
$\nabla:=\del[0]{\partial_t,\overline{\nabla}}$, $\varepsilon>0$ is
the constant diffusion coefficient, $\zeta^-=1$ on
$\partial\mathcal{E}_N^-:=\cbr[0]{(t,x) \in \mathcal{E}_N\,:\,
  \beta\cdot n < 0}$, $\zeta^-=0$ on
$\partial\mathcal{E}_N\backslash\partial\mathcal{E}_N^-$,
$f:=f(t,x) \in L^2(\mathcal{E})$ is a forcing term, and
$g:=g(t,x) \in L^2(\partial\mathcal{E}_N)$ is a given boundary
condition. Furthermore,
$\beta:=\del[0]{1,\overline{\beta}} \in \mathbb{R}^{d+1}$ with
$\overline{\beta} \in \sbr[0]{W^{1,\infty}(\mathcal{E})}^{d}$ a given
divergence-free function. We will assume that
$\norm[0]{\bar{\beta}}_{L^{\infty}(\mathcal{E})}\le 1$ and, following
\cite{Ayuso:2009},
$\norm[0]{\bar{\beta}}_{W^{1,\infty}(\mathcal{E})}\leq
c\norm[0]{\bar{\beta}}_{L^{\infty}(\mathcal{E})}\le c$, with $c>0$ a
constant. Furthermore, following \cite{Cangiani:2019,Schotzau:2009},
we assume that the size of $\Omega$ is order 1 so that
$\varepsilon^{-1}$ is the P\'eclet number of \cref{eq:advdif}.

\begin{remark}
  Note that by definition of $\nabla$, $\beta$, and $\mathcal{E}$,
  \cref{eq:st_adr} can be written as:
  \begin{equation*}
    \partial_t u
    +
    \overline{\nabla}\cdot\del[0]{\overline{\beta}u}
    -
    \varepsilon\overline{\nabla}^2u
    =f,
    \quad
    t \in I,\ x \in \Omega.
  \end{equation*}
  The space-time formulation \cref{eq:advdif}, however, is more
  convenient for the analysis in this paper.
\end{remark}

\section{The space-time hybridizable discontinuous Galerkin method}
\label{s:spacetimehdg}

\subsection{The space-time mesh}
\label{ss:description-stslabsfaceselements}

The space-time domain $\mathcal{E}$ is initially partitioned into $N$
space-time slabs $\mathcal{E}=\cup_{n=0}^{N-1}{\mathcal{E}^n}$. Each
space-time slab is defined as $\mathcal{E}^n := I_n \times \Omega$
where $I_n=\del[0]{t_n,t_{n+1}}$ denotes the $n$th time interval. The
space-time mesh results from dividing each space-time slab into
space-time elements, $\mathcal{E}^n:=\cup_j\mathcal{K}_j^n$. We denote
by $\mathcal{T}_h:=\cup_{n=0}^{N-1}\del[1]{\cup_j\mathcal{K}_j^n}$ the
set of all space-time elements.

Each space-time element $\mathcal{K}$ can be described with a fixed
reference element $\widehat{\mathcal{K}}=\del[0]{-1,1}^{d+1}$ and a
mapping
$\Phi_{\mathcal{K}} : \widehat{\mathcal{K}} \rightarrow \mathcal{K}$,
which is a composition of two mappings.  The first, denoted by
$G_{\mathcal{K}}:\widehat{\mathcal{K}}\rightarrow\widetilde{\mathcal{K}}$,
maps $\widehat{\mathcal{K}}$ to its affine domain
$\widetilde{\mathcal{K}}:=(0,\delta t_{\mathcal{K}}) \times
(0,h_K)^d$. Here, $h_K$ and $\delta t_{\mathcal{K}}$ denote the
spatial mesh size and the time-step, respectively. The affine mapping
is followed by the diffeomorphism $\phi_{\mathcal{K}}$ such that
$\phi_{\mathcal{K}}(\widetilde{\mathcal{K}}) = \mathcal{K}$. The
construction process can be also found in
\cite{Georgoulis:thesis,Sudirham:2006}. For notational purposes, we
will sometimes write $(t,x) = (x_0, x_1, ... x_d)$, with $t=x_0$. On
vectors/matrices, the $0$ index is associated with time.

The shape-regularity of each space-time element $\mathcal{K}$ is
described by the Jacobian of $\phi_{\mathcal{K}}$, denoted by
$J_{\phi_{\mathcal{K}}}$. Following
\cite{Georgoulis:thesis,Sudirham:2006}, we assume for all
$\mathcal{K}\in\mathcal{T}_h$:
\begin{equation}
  \label{eq:diffeom_jac}
  c^{-1}
  \leq
  \envert[0]{\det J_{\phi_{\mathcal{K}}}}
  \leq
  c
  ,
  \quad
  c^{-1}
  \leq
  \envert[0]{\det J^{-1}_{\phi_\mathcal{K}}}
  \leq
  c
  ,
  \quad
  \norm[0]{
    \del[0]{J_{\phi_{\mathcal{K}}}}_{ij}
  }_{L^\infty(\widetilde{\mathcal{K}})}
  \leq
  c,
  \quad
  1\leq i,j\leq d,
\end{equation}
where $c$ is a generic constant independent of $h_K$,
$\delta t_{\mathcal{K}}$, $\varepsilon$, and $T$. Since the spatial
domain $\Omega$ does not depend on time and since the mapping from
$(\tilde{t},\tilde{x})$ to $t$ is a linear translation, and only with
respect to $\tilde{t}$, we have
\begin{equation}
  \label{eq:diffeom_regular_special}
  \del{J_{\phi_{\mathcal{K}}}}_{00}=1,
  \quad
  \del{J_{\phi_{\mathcal{K}}}}_{0k}=0,
  \quad
  \del{J_{\phi_{\mathcal{K}}}}_{k0}=0,
  \quad
  1\le k\le d,\quad\forall\mathcal{K}\in\mathcal{T}_h.
\end{equation}
Let $F^j$ be a facet of $\mathcal{K}$ where $\widetilde{x}_j$
($\tilde{t}$ when $j=0$) is constant in its affine domain. The
parametrization of $F^j$, obtained from the restriction of
$\phi_{\mathcal{K}}$ to $\widetilde{F}^j$, is denoted by
$\phi_{F}$. We furthermore denote by
$J_{\phi_\mathcal{K}}^{j} \in \mathbb{R}^{(d+1)\times d}$ the matrix
obtained from $J_{\phi_{\mathcal{K}}}$ after removing the
$j^{\text{th}}$ column vector. By \cite[Theorem 21.3 and Definition on
page 189]{Munkres:book} it holds that
\begin{equation}
  \label{eq:k_surface}
  \int_{F^j} f(x)\dif s
  =
  \int_{\widetilde{F}^j}
  f\del{\phi_{F}(\widetilde{x})}
  \del[1]{
    \det\del[1]{
      \del[0]{
        J_{\phi_\mathcal{K}}^{j}
      }^\intercal
      J_{\phi_\mathcal{K}}^{j}
    }
  }^{1/2}
  \dif \widetilde{s}.
\end{equation}
We will assume that
\begin{equation}
  \label{eq:diffeom_regular_d}
  c^{-1}
  \leq
  \del[1]{
    \det\del[0]{
      \del[0]{
        J_{\phi_\mathcal{K}}^{i}
      }^\intercal
      J_{\phi_\mathcal{K}}^{i}
    }
  }^{1/2}
  \leq
  c ,
  \quad
  \norm[0]{
    J^i_{\phi_{\mathcal{K}}}
  }_{L^\infty(\widetilde{\mathcal{K}})}
  \leq
  c,
  \quad
  0\leq i\leq d.
\end{equation}

For a space-time element $\mathcal{K}$ with $t\in(t_*,t^*)$, we define
$K_*$ to be the part of the boundary on which $t=t_*$ and $K^*$ to be
the part of the boundary on which $t=t^*$.  We partition the boundary
of $\mathcal{K}$ as
$\partial\mathcal{K}=\mathcal{Q}_{\mathcal{K}} \cup
\mathcal{R}_{\mathcal{K}}$ where
$\mathcal{R}_{\mathcal{K}} := K_* \cup K^*$ and
$\mathcal{Q}_{\mathcal{K}}:=\partial\mathcal{K}\setminus\mathcal{R}_{\mathcal{K}}$.
We assume that boundary facets of $\mathcal{K}$ are flat. The analysis
in this paper also requires the $(d-1)$-dimensional edges of ${K}^*$
and ${K}_*$. We denote such an edge by $E_{\mathcal{K}}$.

We consider facets in the mesh that belong to any of the following
three cases: (1) boundary facets; (2) interior facets shared by two
elements at the same refinement level; (3) interior facets shared
between more than two elements. We denote the set of all facets by
$\mathcal{F}_h$. Within this set, the sets of all interior facets,
boundary facets, $\mathcal{Q}$-facets (facets on which ${n}_t = 0$),
and $\mathcal{R}$-facets (facets on which $\bar{n}=0$) are denoted by
$\mathcal{F}_h^i$, $\mathcal{F}_h^b$, $\mathcal{F}_{\mathcal{Q},h}$,
and $\mathcal{F}_{\mathcal{R},h}$, respectively. The union of all
facets in $\mathcal{F}_h$ is denoted by $\Gamma$. Furthermore, we
denote by $\partial\mathcal{T}_h$ the set of element boundaries, by
$\mathcal{Q}_h$ the set that consists of parts of an element boundary
on which $n_t=0$, by $\mathcal{R}_h$ the set that consists of parts of
an element boundary on which $\bar{n}=0$, and by
$\partial\mathcal{T}_h^i$ the set of element boundaries excluding the
part of the element boundary that lies on $\partial\mathcal{E}$.

We denote by $\omega_{\mathcal{K}}$ the union of elements
$\mathcal{K}'$ such that
$\partial\mathcal{K}\cap\partial\mathcal{K}'=F$, with $F$ a facet, and
denote by $\sigma_{\mathcal{K}}$ the union of elements that share at
least one vertex with $\mathcal{K}$. Consider now a facet $F$. Any
elements containing facets $F'$ such that $F \cap F'$ is itself a
facet belong to the set $\omega_F$. See \cref{fig:sets} for a
depiction of $\omega_{\mathcal{K}}$, $\sigma_{\mathcal{K}}$, and
$\omega_F$.

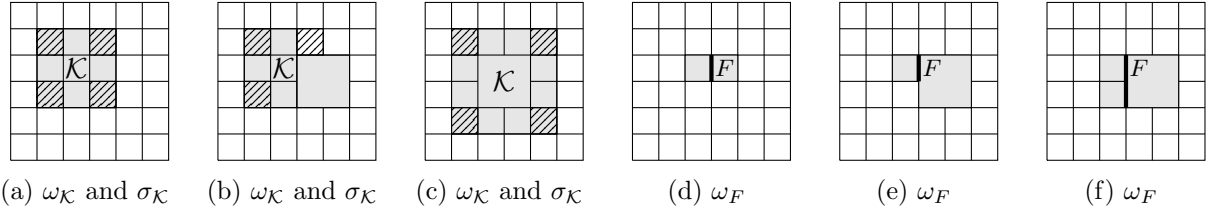
\begin{figure}[tbp]
  \centering
  \subfloat[$\omega_{\mathcal{K}}$ and $\sigma_{\mathcal{K}}$]{
    \begin{tikzpicture}[scale=0.35]
      \draw [color=black,fill=gray!20!white] (1,2) rectangle (4,5);
      \draw (1,4)--(2,4)--(2,5)--(1,5)--(1,4);
      \fill [pattern=north east lines] (1,4)--(2,4)--(2,5)--(1,5)--(1,4);
      \draw (3,4)--(4,4)--(4,5)--(3,5)--(3,4);
      \fill [pattern=north east lines] (3,4)--(4,4)--(4,5)--(3,5)--(3,4);
      \draw (1,2)--(2,2)--(2,3)--(1,3)--(1,2);
      \fill [pattern=north east lines] (1,2)--(2,2)--(2,3)--(1,3)--(1,2);
      \draw (3,2)--(4,2)--(4,3)--(3,3)--(3,2);
      \fill [pattern=north east lines] (3,2)--(4,2)--(4,3)--(3,3)--(3,2);
      \node at (2.5,3.5) {{$\mathcal{K}$}};
      \draw [color=black] (0,0)--(0,6)--(6,6)--(6,0)--(0,0);
      \draw [color=black] (1,0)--(1,6);
      \draw [color=black] (2,0)--(2,6);
      \draw [color=black] (3,0)--(3,6);
      \draw [color=black] (4,0)--(4,6);
      \draw [color=black] (5,0)--(5,6);
      \draw [color=black] (0,1)--(6,1);
      \draw [color=black] (0,2)--(6,2);
      \draw [color=black] (0,3)--(6,3);
      \draw [color=black] (0,4)--(6,4);
      \draw [color=black] (0,5)--(6,5);
    \end{tikzpicture}
  }
  \quad
  \subfloat[$\omega_{\mathcal{K}}$ and $\sigma_{\mathcal{K}}$]{
    \begin{tikzpicture}[scale=0.35]
      \draw [color=black,fill=gray!20!white] (1,2) rectangle (3,5);
      \draw [color=black,fill=gray!20!white] (3,2) rectangle (5,4);
      \draw (1,4)--(2,4)--(2,5)--(1,5)--(1,4);
      \fill [pattern=north east lines] (1,4)--(2,4)--(2,5)--(1,5)--(1,4);
      \draw (3,4)--(4,4)--(4,5)--(3,5)--(3,4);
      \fill [pattern=north east lines] (3,4)--(4,4)--(4,5)--(3,5)--(3,4);
      \draw (1,2)--(2,2)--(2,3)--(1,3)--(1,2);
      \fill [pattern=north east lines] (1,2)--(2,2)--(2,3)--(1,3)--(1,2);
      \node at (2.5,3.5) {{$\mathcal{K}$}};
      \draw [color=black] (0,0)--(0,6)--(6,6)--(6,0)--(0,0);
      \draw [color=black] (1,0)--(1,6);
      \draw [color=black] (2,0)--(2,6);
      \draw [color=black] (3,0)--(3,6);
      \draw [color=black] (4,0)--(4,2);
      \draw [color=black] (4,4)--(4,6);
      \draw [color=black] (5,0)--(5,6);
      \draw [color=black] (0,1)--(6,1);
      \draw [color=black] (0,2)--(6,2);
      \draw [color=black] (0,3)--(3,3);
      \draw [color=black] (5,3)--(6,3);
      \draw [color=black] (0,4)--(6,4);
      \draw [color=black] (0,5)--(6,5);
    \end{tikzpicture}
  }
  \quad
  \subfloat[$\omega_{\mathcal{K}}$ and $\sigma_{\mathcal{K}}$]{
    \begin{tikzpicture}[scale=0.35]
      \draw [color=black,fill=gray!20!white] (1,1) rectangle (5,5);
      \draw (1,1)--(2,1)--(2,2)--(1,2)--(1,1);
      \fill [pattern=north east lines] (1,1)--(2,1)--(2,2)--(1,2)--(1,1);
      \draw (4,1)--(5,1)--(5,2)--(4,2)--(4,1);
      \fill [pattern=north east lines] (4,1)--(5,1)--(5,2)--(4,2)--(4,1);
      \draw (1,4)--(2,4)--(2,5)--(1,5)--(1,4);
      \fill [pattern=north east lines] (1,4)--(2,4)--(2,5)--(1,5)--(1,4);
      \draw (4,4)--(5,4)--(5,5)--(4,5)--(4,4);
      \fill [pattern=north east lines] (4,4)--(5,4)--(5,5)--(4,5)--(4,4);
      \node at (3,3) {{$\mathcal{K}$}};
      \draw [color=black] (0,0)--(0,6)--(6,6)--(6,0)--(0,0);
      \draw [color=black] (1,0)--(1,6);
      \draw [color=black] (2,0)--(2,6);
      \draw [color=black] (3,0)--(3,2);
      \draw [color=black] (3,4)--(3,6);
      \draw [color=black] (4,0)--(4,6);
      \draw [color=black] (5,0)--(5,6);
      \draw [color=black] (0,1)--(6,1);
      \draw [color=black] (0,2)--(6,2);
      \draw [color=black] (0,3)--(2,3);
      \draw [color=black] (4,3)--(6,3);
      \draw [color=black] (0,4)--(6,4);
      \draw [color=black] (0,5)--(6,5);
    \end{tikzpicture}
  }
  \quad
  \subfloat[$\omega_F$]{
    \begin{tikzpicture}[scale=0.35]
      \draw [color=black] (0,0)--(0,6)--(6,6)--(6,0)--(0,0);
      \draw [color=black] (1,0)--(1,6);
      \draw [color=black] (2,0)--(2,6);
      \draw [color=black] (3,0)--(3,6);
      \draw [color=black] (4,0)--(4,6);
      \draw [color=black] (5,0)--(5,6);
      \draw [color=black] (0,1)--(6,1);
      \draw [color=black] (0,2)--(6,2);
      \draw [color=black] (0,3)--(6,3);
      \draw [color=black] (0,4)--(6,4);
      \draw [color=black] (0,5)--(6,5);
      \draw [color=black,fill=gray!20!white] (2,3) rectangle (3,4);
      \draw [color=black,fill=gray!20!white] (3,3) rectangle (4,4);
      \draw [black, ultra thick] (3,3)--(3,4);
      \node at (3.5,3.5) {\footnotesize{$F$}};
    \end{tikzpicture}
  }
  \quad
  \subfloat[$\omega_F$]{
    \begin{tikzpicture}[scale=0.35]
      \draw [color=black] (0,0)--(0,6)--(6,6)--(6,0)--(0,0);
      \draw [color=black] (1,0)--(1,6);
      \draw [color=black] (2,0)--(2,6);
      \draw [color=black] (3,0)--(3,6);
      \draw [color=black] (4,0)--(4,2);
      \draw [color=black] (4,4)--(4,6);
      \draw [color=black] (5,0)--(5,6);
      \draw [color=black] (0,1)--(6,1);
      \draw [color=black] (0,2)--(6,2);
      \draw [color=black] (0,3)--(3,3);
      \draw [color=black] (5,3)--(6,3);
      \draw [color=black] (0,4)--(6,4);
      \draw [color=black] (0,5)--(6,5);
      \draw [color=black,fill=gray!20!white] (2,3) rectangle (3,4);
      \draw [color=black,fill=gray!20!white] (3,2) rectangle (5,4);
      \draw [black, ultra thick] (3,3)--(3,4);
      \node at (3.5,3.5) {\footnotesize{$F$}};
    \end{tikzpicture}
  }
  \quad
  \subfloat[$\omega_F$]{
    \begin{tikzpicture}[scale=0.35]
      \draw [color=black] (0,0)--(0,6)--(6,6)--(6,0)--(0,0);
      \draw [color=black] (1,0)--(1,6);
      \draw [color=black] (2,0)--(2,6);
      \draw [color=black] (3,0)--(3,6);
      \draw [color=black] (4,0)--(4,2);
      \draw [color=black] (4,4)--(4,6);
      \draw [color=black] (5,0)--(5,6);
      \draw [color=black] (0,1)--(6,1);
      \draw [color=black] (0,2)--(6,2);
      \draw [color=black] (0,3)--(3,3);
      \draw [color=black] (5,3)--(6,3);
      \draw [color=black] (0,4)--(6,4);
      \draw [color=black] (0,5)--(6,5);
      \draw [color=black,fill=gray!20!white] (2,2) rectangle (3,3);
      \draw [color=black,fill=gray!20!white] (2,3) rectangle (3,4);
      \draw [color=black,fill=gray!20!white] (3,2) rectangle (5,4);
      \draw [black, ultra thick] (3,2)--(3,4);
      \node at (3.5,3.5) {\footnotesize{$F$}};
    \end{tikzpicture}
  }
  \caption{Depiction of sets $\omega_{\mathcal{K}}$,
	 $\sigma_{\mathcal{K}}$, and $\omega_F$ on conforming
	 and 1-irregularly refined meshs. Figures (A), (B), (C):
	 elements in the set $\omega_{\mathcal{K}}$ are the grey
	 colored elements excluding the hatched elements; elements in
	 the set $\sigma_{\mathcal{K}}$ are colored grey and include the
	hatched elements. Figures (D), (E), (F): elements in the set
	$\omega_F$ are colored grey. }
 \label{fig:sets}
\end{figure}

We will allow at most 1-irregularly refined space-time elements.
To account for local time-stepping, consider a space-time element
$\mathcal{K}$ in space-time slab $\mathcal{E}^n$. Then we
introduce, in addition to the local time-step $\delta
t_{\mathcal{K}}$ set by ${\Phi}_{\mathcal{K}}$, the slab
time-step $\Delta t_{\mathcal{K}} := t_{n+1}-t_n$, i.e., the
length of $I_n$. Note that $\delta t_{\mathcal{K}} \le \Delta
t_{\mathcal{K}}$ with $\delta t_{\mathcal{K}} < \Delta
t_{\mathcal{K}}$ when using local time-stepping. We will assume
that $\Delta t_{\mathcal{K}} / \delta t_{\mathcal{K}} \leq c$ for
all $\mathcal{K}\in\mathcal{T}_h$.

\subsection{The discretization}
\label{ss:approxspacesdisc}

We consider approximations to the advection-diffusion problem
\cref{eq:advdif} in the finite element space
$\boldsymbol{V}_h := V_h \times M_h$, where
\begin{align*}
  V_h
  &:=
    \cbr[0]{
    v_h
    \in
    L^2\del[0]{\mathcal{E}}:
    v_h|_{\mathcal{K}}
    \circ
    {\Phi}_{\mathcal{K}}
    \in
    Q^{\del[0]{p_t,p_s}}\del[0]{\widehat{\mathcal{K}}}
    \quad
    \forall\mathcal{K}\in\mathcal{T}_h
    },
  \\
  M_h
  &:=
    \cbr[0]{
    \mu_h
    \in
    L^2\del[0]{\Gamma}:
    \mu_h|_F
    \circ
    {\Phi}_{\mathcal{K}}
    \in
    Q^{\del{p_t,p_s}}\del[0]{\widehat{F}}
    \quad
    \forall F\in\mathcal{F}_h,\
    \mu_h=0 \text{ on } \partial\mathcal{E}_D}.
\end{align*}
Here $Q^{\del{p_t,p_s}}(U)$ is the set of all tensor product
polynomials of degree $p_t$ in the temporal direction and $p_s$ in
each spatial direction on a domain $U$. We will assume that
$p_s \ge 1$ and $p_t=1$, i.e., we consider a second order accurate
time stepping scheme. For simplicity of notation, we introduce
$\boldsymbol{v}_h:=(v_h, \mu_h) \in \boldsymbol{V}_h$ and
$\boldsymbol{u}_h := (u_h, \lambda_h) \in \boldsymbol{V}_h$. On an
element boundary we denote the HDG jump by
$\sbr{\boldsymbol{v}_h}:= \del{v_h-\mu_h}$ and on a facet
$F\in\mathcal{F}^i_h$, where
$F\subset\partial\mathcal{K}_1\cap\partial\mathcal{K}_2$, we denote
the usual DG jump by $\jump{v_h} := \del{v_{h1}{n}_1+v_{h2}{n}_2}$.
Next, consider two elements $\mathcal{K}_1$ and $\mathcal{K}_2$ such
that $K_1^* = K_{2,*}$. Denote the restriction of $\mu_h$ to
$\mathcal{Q}_{\mathcal{K}_1}$ and $\mathcal{Q}_{\mathcal{K}_2}$ by
$\mu_{h1}$ and $\mu_{h2}$, respectively. The jump of $\mu_h$ across
edges of $K_1^*$ is defined by $\ejump{\mu_h} := \mu_{h1} - \mu_{h2}$.
Note that for pairs of $K_1$ and $K_2$ such that
$K_1^*\subsetneq K_{2,*}$ or $K_{2,*}\subsetneq K_1^*$, we do not
define any edge jump.

Let $(u, v)_U$ and $\langle u, v \rangle_U$ be the $L^2$-inner products
on $U$ if $U \subset \mathbb{R}^{d+1}$ and $U \subset \mathbb{R}^d$,
respectively. We define
$(u, v)_{\mathcal{T}_h} := \sum_{\mathcal{K} \in \mathcal{T}_h}(u,
v)_{\mathcal{K}}$,
$\langle u, v \rangle_{\partial \mathcal{T}_h} := \sum_{\mathcal{K}
  \in \mathcal{T}_h} \langle u, v \rangle_{\partial \mathcal{K}}$,
$\langle u, v \rangle_{\mathcal{Q}_h} := \sum_{\mathcal{K} \in
  \mathcal{T}_h} \langle u, v \rangle_{\mathcal{Q}_{\mathcal{K}}}$,
and
$\langle u, v \rangle_{\partial\mathcal{E}_N} := \sum_{F \in
  \mathcal{F}_h^b \cap \partial\mathcal{E}_N} \langle u, v
\rangle_{F}$. The space-time IP-HDG method for \cref{eq:advdif} is
given by \cite{Kirk:2019,Wang:2023}: Find
$\boldsymbol{u}_h\in \boldsymbol{V}_h$ such that
\begin{equation}
  \label{eq:st_hdg_adr_compact}
  a_h\del[0]{
    \boldsymbol{u}_h
    ,
    \boldsymbol{v}_h
  }
  =
  \del[0]{f,v_h}_{\mathcal{T}_h}
  +
  \langle
  g,
  \mu_h
  \rangle_{\partial\mathcal{E}_N}
  \quad
  \forall\boldsymbol{v}_h\in \boldsymbol{V}_h,
\end{equation}
with
$a_h(\boldsymbol{u}_h, \boldsymbol{v}_h) := a_{h,d}(\boldsymbol{u}_h,
\boldsymbol{v}_h) + a_{h,c}(\boldsymbol{u}_h, \boldsymbol{v}_h)$ and
where
\begin{align*}
  a_{h,d}\del[0]{
  \boldsymbol{u} , \boldsymbol{v}
  }
  &
    :=
    \del[0]{\varepsilon\overline{\nabla}u,\overline{\nabla}v}_{\mathcal{T}_h}
    +
    \langle
    \varepsilon\alpha h_{K}^{-1}
    \sbr[0]{\boldsymbol{u}}
    ,
    \sbr[0]{\boldsymbol{v}}
    \rangle_{\mathcal{Q}_h}
    -
    \langle
    \varepsilon\sbr[0]{\boldsymbol{u}}
    ,
    \overline{\nabla}_{\bar{{n}}}v
    \rangle_{\mathcal{Q}_h}
    -
    \langle
    \varepsilon\overline{\nabla}_{\bar{{n}}}u,
    \sbr[0]{\boldsymbol{v}}
    \rangle_{\mathcal{Q}_h},
  \\
  a_{h,c}\del[0]{
  \boldsymbol{u} , \boldsymbol{v}
  }
  &
    :=
    -
    \del[0]{{\beta}u,\nabla v}_{\mathcal{T}_h}
    +
    \langle
    \zeta^+ \beta \cdot n \lambda, \mu
    \rangle_{\partial\mathcal{E}_N}
    +
    \langle
    \del[0]{{\beta}\cdot{n}}
    \lambda
    +
    \beta_s\sbr{\boldsymbol{u}}
    ,
    \sbr{\boldsymbol{v}}
    \rangle_{\partial\mathcal{T}_h}.
\end{align*}
Here $\alpha>0$ is a penalty parameter, $\zeta^+$ denotes the outflow
boundary indicator on a facet, and
$\beta_s := \sup_{(x,t)\in F}|\beta\cdot n|$, for
$F \subset \partial \mathcal{K}$. Note that
\begin{equation}
  \label{eq:betasinfmax}
  \inf_{(x,t) \in F}
  (\beta_s  - \tfrac{1}{2}\beta\cdot n)
  \ge
  \tfrac{1}{2}
  \max_{(x,t)\in F}|\beta\cdot n|
  \qquad
  \forall F \subset \partial\mathcal{K},
  \ \forall \mathcal{K} \in \mathcal{T}_h.
\end{equation}
Let $H^s({U})$ be the usual Sobolev space for $s \ge 0$, where we
remark that $H^0(U) = L^2(U)$. Furthermore, let
$V := \cbr[0]{v \in H^1(\mathcal{E})\,|\, v|_{\partial\mathcal{E}_D} =
  0} \cap H^2(\mathcal{E})$ and $M$ the trace space of $V$. Let
$V(h) := V_h+V$, $M(h) := M_h+M$, and
$\boldsymbol{V}(h) := V(h) \times M(h)$. We define the following norms
on $\boldsymbol{V}(h)$:
\begin{align*}
  \tnorm{\boldsymbol{v}}_{s,h}^2
  :=
  &
    \sum_{\mathcal{K}\in\mathcal{T}_h}
    \norm[0]{v}_\mathcal{K}^2
    +
    \sum_{\mathcal{K}\in\mathcal{T}_h}
    \norm[0]{ \envert[0]{ \beta_s - \tfrac{1}{2} \beta\cdot{n} }^{1/2} \sbr{\boldsymbol{v}}
    }_{\partial\mathcal{K}}^2
    +
    \sum_{F\in\partial\mathcal{E}_N}
    \norm[0]{
    \envert[0]{
    \tfrac{1}{2}
    \beta\cdot{n}
    }^{1/2}
    \mu
    }_F^2
  \\
  &
    +
    \sum_{\mathcal{K}\in\mathcal{T}_h}
    \varepsilon\norm[0]{\overline{\nabla}v}_\mathcal{K}^2
    +
    \sum_{\mathcal{K}\in\mathcal{T}_h}
    \varepsilon h_K^{-1}
    \norm[0]{\sbr[0]{\boldsymbol{v}}}_{\mathcal{Q}_\mathcal{K}}^2
    +
    \sum_{\mathcal{K}\in\mathcal{T}_h}
    \tau_\varepsilon
    \norm[0]{\partial_t v}^2_\mathcal{K},
    \nonumber
  \\
  \tnorm{\boldsymbol{v}}_{sT,h}^2
  :=
  &
    \sum_{\mathcal{K}\in\mathcal{T}_h}
    \norm[0]{v}_\mathcal{K}^2
    +
    \sum_{\mathcal{K}\in\mathcal{T}_h}
    \norm[0]{ \envert[0]{ \beta_s - \tfrac{1}{2} \beta\cdot{n} }^{1/2} \sbr{\boldsymbol{v}}
    }_{\partial\mathcal{K}}^2
    +
    T\sum_{F\in\partial\mathcal{E}_N}
    \norm[0]{
    \envert[0]{
    \tfrac{1}{2}
    \beta\cdot{n}
    }^{1/2}
    \mu
    }_F^2
  \\
  &
    +
    T\sum_{\mathcal{K}\in\mathcal{T}_h}
    \varepsilon\norm[0]{\overline{\nabla}v}_\mathcal{K}^2
    +
    \sum_{\mathcal{K}\in\mathcal{T}_h}
    \varepsilon h_K^{-1}
    \norm[0]{\sbr[0]{\boldsymbol{v}}}_{\mathcal{Q}_\mathcal{K}}^2
    +
    \sum_{\mathcal{K}\in\mathcal{T}_h}
    \tau_\varepsilon
    \norm[0]{\partial_t v}^2_\mathcal{K},
    \nonumber
\end{align*}
where, following \cite{Wang:2023}, the parameter
$\tau_{\varepsilon}$ depends on the size of the space-time
element compared to the diffusion parameter $\varepsilon$:
\begin{equation*}
  \tau_{\varepsilon} :=
  \Delta t_{\mathcal{K}}\tilde{\varepsilon}
  :=
  \begin{cases}
    \Delta t_{\mathcal{K}}
    &\text{if } \mathcal{K} \in \mathcal{T}_h^d
    :=
    \cbr[1]{
      \mathcal{K}\in\mathcal{T}_h
      |
      {\delta t_{\mathcal{K}}}\leq h_K\leq\varepsilon
    },
    \\
    \Delta t_{\mathcal{K}}\varepsilon^{1/2}
    &\text{if } \mathcal{K} \in \mathcal{T}_h^x
    :=
    \cbr[1]{
      \mathcal{K}\in\mathcal{T}_h
      |
      {\delta t_{\mathcal{K}}}\leq \varepsilon<h_K
    },
    \\
    \Delta t_{\mathcal{K}}\varepsilon
    &\text{if } \mathcal{K} \in \mathcal{T}_h^c
    :=
    \cbr[1]{
      \mathcal{K}\in\mathcal{T}_h
      |
      \varepsilon<{\delta t_{\mathcal{K}}}\leq h_K
    }.
  \end{cases}
\end{equation*}
The following inf-sup condition was proven in \cite[Section
4.2]{Wang:2023}. There exists a constant $c_T>0$ independent of $h_K$,
$\delta t_{\mathcal{K}}$, and $\varepsilon$, but linear in $T$, such
that
\begin{equation}
  \label{eq:inf_sup_s_norm}
  \tnorm{\boldsymbol{w}_h}_{s,h}
  \le c_T
  \sup_{\boldsymbol{v}_h\in \boldsymbol{V}_h}
  \frac{
    a_h(\boldsymbol{w}_h,\boldsymbol{v}_h)
  }{
	  \tnorm{\boldsymbol{v}_h}_{s,h}
  }
  \qquad \forall \boldsymbol{w}_h\in \boldsymbol{V}_h.
\end{equation}

\section{A posteriori error analysis}
\label{s:apos}

The residual on a space-time element $\mathcal{K}\in\mathcal{T}_h$ is
defined as
$R_h^{\mathcal{K}} := f + \varepsilon\overline{\nabla}^2u_h -
\nabla\cdot\del{\beta u_h}$, while the residual on a facet
$F \in \partial\mathcal{E}_N$ is defined as
$R_h^N := g - \varepsilon \overline{\nabla} u_h \cdot{\overline{{n}}}
+ \zeta^- u_h {\beta\cdot{n}}$. Furthermore, we define
$\eta_R^\mathcal{K} :=
\lambda_\mathcal{K}\norm[0]{R_h^{\mathcal{K}}}_\mathcal{K}$ where
$\lambda_{\mathcal{K}}:=\min\cbr[0]{1,h_K\varepsilon^{-1/2}}$,
${\eta_{J,1}^\mathcal{K}} := h_K^{1/2}\varepsilon^{1/2}
\norm[1]{\jump{\overline{\nabla}_{\overline{{n}}}u_h}}_{\mathcal{Q}_\mathcal{K}\setminus\partial\mathcal{E}}$,
$\eta_{J,2}^\mathcal{K} := \del[1]{
  \del[0]{\eta_{J,2,1}^\mathcal{K}}^2 +
  \del[0]{\eta_{J,2,2}^\mathcal{K}}^2 }^{1/2}$ where
$\eta_{J,2,1}^\mathcal{K} := h_K^{-1/2}{\varepsilon}^{1/2}
\norm{\sbr{\boldsymbol{u}_h}}_{\mathcal{Q}_{\mathcal{K}}}$ and
$\eta_{J,2,2}^\mathcal{K} := h_K^{1/4} {\varepsilon}^{-1/2}
\norm{\sbr{\boldsymbol{u}_h}}_{\mathcal{Q}_{\mathcal{K}}}$,
$\eta_{J,3}^\mathcal{K} := \del[1]{
  \del[0]{\eta_{J,3,\mathcal{Q}}^\mathcal{K}}^2 +
  \del[0]{\eta_{J,3,\mathcal{R}}^\mathcal{K}}^2 }^{1/2}$ where
$\eta_{J,3,\mathcal{Q}}^\mathcal{K} :=
\norm[0]{\envert[0]{\beta_s-\tfrac{1}{2}\beta\cdot
    n}^{1/2}\sbr{\boldsymbol{u}_h}}_{\mathcal{Q}_{\mathcal{K}}}$ and
$\eta_{J,3,\mathcal{R}}^\mathcal{K} :=
\norm[0]{\envert[0]{\beta_s-\tfrac{1}{2}\beta\cdot
    n}^{1/2}\sbr{\boldsymbol{u}_h}}_{\mathcal{R}_{\mathcal{K}}}$,
${\eta_{BC,1}^\mathcal{K}} := h_K^{1/2}\varepsilon^{-1/2}
\norm[0]{R_h^N}_{\mathcal{Q}_\mathcal{K}\cap\partial\mathcal{E}_N}$,
and
$\eta_{BC,2}^\mathcal{K} :=
\norm[0]{R_h^N}_{\partial\mathcal{K}\cap\Omega_0}
=\norm[0]{g-u_h}_{\partial\mathcal{K}\cap\Omega_0}$. We then introduce
the following a posteriori error estimator for the solution
$\boldsymbol{u}_h\in\boldsymbol{V}_h$ to \cref{eq:st_hdg_adr_compact}:
\begin{equation}
  \label{eq:apos_st_hdg_ests_total}
  \eta^2
  :=
  \sum_{\mathcal{K}\in\mathcal{T}_h}
  \del[0]{\eta^{\mathcal{K}}}^2,
\end{equation}
where
$\del[0]{\eta^{\mathcal{K}}}^2 := \del[0]{\eta_R^\mathcal{K}}^2 +
\sum_{i=1}^3 \del[0]{\eta_{J,i}^\mathcal{K}}^2 + \sum_{j=1}^2
\del[0]{\eta_{BC,j}^\mathcal{K}}^2$. The following theorems establish
reliability and efficiency of the a posteriori error estimator
\cref{eq:apos_st_hdg_ests_total}. Their proofs are given in
\cref{ss:rel,ss:eff}, respectively.

\begin{theorem}[Reliability]
  \label{thm:reliability}
  Let $u$ solve \cref{eq:advdif}, $\boldsymbol{u} = (u, u|_{\Gamma})$,
  and let $\boldsymbol{u}_h$ solve
  \cref{eq:st_hdg_adr_compact}. Assuming that
  $\delta t_{\mathcal{K}} = \mathcal{O}(h_K^2)$, we have the following
  reliability estimate
   \begin{equation}
     \label{eq:reliability}
     \tnorm{
       \boldsymbol{u}-\boldsymbol{u}_h
     }_{sT,h}
     \le cT
     \varepsilon^{-1/2}
     \eta.
   \end{equation}
\end{theorem}

\begin{theorem}[Efficiency]
  \label{thm:efficiency}
  Let $\boldsymbol{u}$ and $\boldsymbol{u}_h$ be as in
  \cref{thm:reliability} and assume that
  $\delta t_{\mathcal{K}} = \mathcal{O}(h_K^2)$. Furthermore, let
  $\mathrm{osc}_h^{\mathcal{K}} := \lambda_{\mathcal{K}}
  \norm[0]{(I-\Pi_h)R_h^{\mathcal{K}}}_{\mathcal{K}}$ and
  $\mathrm{osc}_h^{N} := h_K^{1/2} \varepsilon^{-1/2}
  \norm[0]{(I-\Pi_h^{\mathcal{F}})R_h^{N}}_{F}$, where
  $\Pi_h^{\mathcal{F}}$ denotes the $L^2$-projection onto $M_h$.
  Then, for all $\mathcal{K} \in \mathcal{T}_h$,
  \begin{equation}
    \label{eq:efficiency}
    \eta^\mathcal{K}
    \le c
	 \sum_{\mathcal{K}\subset\omega_{\mathcal{K}}}
    \varepsilon^{-1/2}
    \tilde{\varepsilon}^{-1/2}
    \tnorm{\boldsymbol{u}-\boldsymbol{u}_h}_{sT,h,\mathcal{K}}
    +c\,{\mathrm{osc}_h^{\mathcal{K}}}
    +c\,{\mathrm{osc}_h^{N}},
  \end{equation}
  where
  \begin{equation*}
    \begin{split}
      \tnorm{\boldsymbol{v}}_{sT,h,\mathcal{K}}^2
      :=
      &
      \norm[0]{v}_\mathcal{K}^2
      +
      \norm[0]{ \envert[0]{ \beta_s - \tfrac{1}{2} \beta\cdot{n} }^{1/2} \sbr{\boldsymbol{v}}
      }_{\partial\mathcal{K}}^2
      +
      T\sum_{F\in\partial\mathcal{E}_N\cap\partial\mathcal{K}}
      \norm[0]{
        \envert[0]{
          \tfrac{1}{2}
          \beta\cdot{n}
        }^{1/2}
        \mu
      }_F^2
      \\
      &
      +
      T\varepsilon\norm[0]{\overline{\nabla}v}_\mathcal{K}^2
      +
      \varepsilon h_K^{-1}
      \norm[0]{\sbr[0]{\boldsymbol{v}}}_{\mathcal{Q}_\mathcal{K}}^2
      +
      \tau_\varepsilon
      \norm[0]{\partial_t v}^2_\mathcal{K}.
    \end{split}
  \end{equation*}
\end{theorem}

\begin{remark}
  \label{rem:efficiencyrefined-new}
  From \cref{thm:efficiency}, and by definition of
  $\tilde{\varepsilon}$, we have that on sufficiently refined elements
  $\mathcal{K}\in\mathcal{T}_h^d$ the following estimate holds:
  \begin{equation*}
    \eta^\mathcal{K}
    \le c
	 \sum_{\mathcal{K}\subset\omega_{\mathcal{K}}}
    \varepsilon^{-1/2}
    \tnorm{\boldsymbol{u}-\boldsymbol{u}_h}_{sT,h,\mathcal{K}}
    +c\,{\mathrm{osc}_h^{\mathcal{K}}}
    +c\,{\mathrm{osc}_h^{N}}.
  \end{equation*}
\end{remark}

\subsection{Useful inequalities, approximations and projections}
\label{ss:ineqapproxbounds}

We have the following anisotropic inverse and trace inequalities:
\begin{subequations}
  \begin{align}
    \label{eq:eg_inv_1}
    \norm[0]{\partial_tv_h}_\mathcal{K}
    &\leq
      c
      \delta t_{\mathcal{K}}^{-1}
      \norm[0]{v_h}_\mathcal{K} && \forall v_h \in V_h,
    \\
    \label{eq:eg_inv_2}
    \norm[0]{\overline{\nabla}v_h}_{\mathcal{K}}
    &\leq
      c
      h_K^{-1}
      \norm[0]{v_h}_\mathcal{K} && \forall v_h \in V_h,
    \\
    \label{eq:eg_inv_3}
    \norm[0]{v_h}_{\mathcal{Q}_\mathcal{K}}
    &\leq
      c
      h_K^{-1/2}
      \norm[0]{v_h}_\mathcal{K} && \forall v_h \in V_h,
    \\
    \label{eq:eg_inv_4}
    \norm[0]{v_h}_{\mathcal{R}_\mathcal{K}}
    &\leq
      c
      \delta t_{\mathcal{K}}^{-1/2}
      \norm[0]{v_h}_\mathcal{K} && \forall v_h \in V_h,
    \\
    \label{eq:eg_inv_low_d_1_F}
    \norm[0]{\partial_t\mu_h}_{\mathcal{Q}_{\mathcal{K}}}
    &\leq
      c
      \delta t_{\mathcal{K}}^{-1}
      \norm[0]{\mu_h}_{\mathcal{Q}_{\mathcal{K}}} && \forall \mu_h \in M_h.
  \end{align}
\end{subequations}
Here \cref{eq:eg_inv_1,eq:eg_inv_2,eq:eg_inv_3,eq:eg_inv_4} have been
adapted from \cite[Corollaries 3.49, 3.54]{Georgoulis:thesis} to the
space-time context specifically taking into account the spatial mesh
size $h_K$ and time step $\delta t_{\mathcal{K}}$ of a space-time
element $\mathcal{K} \in \mathcal{T}_h$. \Cref{eq:eg_inv_low_d_1_F}
was proven in \cite[Lemma 3.1]{Wang:2023}. The following lemma
introduces additional inequalities.

\begin{lemma}
  \label{lem:eg_inv}
  Let $\mathcal{K} \in \mathcal{T}_h$ be a space-time element and
  $\mu_h\in M_h$. For any
  $F_{\mathcal{R}}\subset\mathcal{R}_{\mathcal{K}}$ and
  $F_{\mathcal{Q}}\subset\mathcal{Q}_{\mathcal{K}}$, we have
  \begin{subequations}
    \begin{align}
      \label{eq:eg_inv_low_d_3}
      \norm[0]{\overline{\nabla}\mu_h}_{F_{\mathcal{R}}}
      &\le c
        h_K^{-1}
        \norm[0]{\mu_h}_{F_{\mathcal{R}}},
      \\
      \label{eq:eg_inv_low_d_4}
      \norm[0]{\mu_h}_{{E}_{\mathcal{K}}}
      &\le c
        \delta t_{\mathcal{K}}^{-{1}/{2}}
        \norm[0]{\mu_h}_{F_{\mathcal{Q}}}
        \quad
        \forall E_{\mathcal{K}}\subset F_{\mathcal{Q}},
      \\
      \label{eq:eg_inv_low_d_5}
      \norm[0]{\mu_h}_{{E}_{\mathcal{K}}}
      &\le c
        h_K^{-{1}/{2}}
        \norm[0]{\mu_h}_{F_{\mathcal{R}}}
        \quad
        \forall E_{\mathcal{K}}\subset F_{\mathcal{R}}.
    \end{align}
  \end{subequations}
\end{lemma}
\begin{proof}
  \Cref{eq:eg_inv_low_d_3} is a result of applying \cref{eq:eg_inv_2}
  on $F_{\mathcal{R}}$ while \cref{eq:eg_inv_low_d_5} is a direct
  application of a standard isotropic trace inequality on
  $E_{\mathcal{K}}$ (see, for example, \cite[Lemma
  12.8]{Ern:book}). As for \cref{eq:eg_inv_low_d_4}, consider
  $\phi_{F_{\mathcal{Q}}}(\widetilde{E}_{\mathcal{K}})=E_{\mathcal{K}}$.
  Applying \cref{eq:eg_inv_4} on the affine domain gives us
  $\norm[0]{\widetilde{\mu}_h }_{\widetilde{E}_{\mathcal{K}}} \leq c
  \delta t_{\mathcal{K}}^{-1/2}
  \norm[0]{\widetilde{\mu}_h}_{\widetilde{F}_{\mathcal{Q}}}$. We
  conclude \cref{eq:eg_inv_low_d_4} via scaling arguments
  \cref{eq:scalingphiextra-1,eq:scalingequiv_2}.
\end{proof}

The following lemma presents local quasi-interpolation estimates that
will be useful in showing reliability of the error estimator
\cref{eq:apos_st_hdg_ests_total}.
\begin{lemma}
  [Local quasi-interpolation estimates]
  \label{lem:local_quasi_tnorm_st}
  Let $v\in H^1(\mathcal{K})$ and let $\Pi_h$ be the $L^2$-projection
  onto $V_h$. For any $\mathcal{K}\in\mathcal{T}_h$ and any
  $F_\mathcal{Q}\subset\mathcal{Q}_\mathcal{K}$,
  $F_\mathcal{R}\subset\mathcal{R}_\mathcal{K}$, assuming that
  $\delta t_{\mathcal{K}} = \mathcal{O}(h_K^2)$, the following
  quasi-interpolation estimates can be shown
  \begin{subequations}
    \label{eq:local_quasi_tnorm_st_step_ratio}
    \begin{align}
      \label{eq:local_quasi_tnorm_st_step_ratio_1}
      \norm[0]{v-\Pi_hv}_\mathcal{K}
      &
        \le c
        \lambda_\mathcal{K}
        \del[1]{
        h_{K}\varepsilon^{1/2}
        \norm[0]{\partial_t v}_\mathcal{K}
        +
        \varepsilon^{1/2}
        \norm[0]{\overline{\nabla} v}_\mathcal{K}
        +
        \norm[0]{v}_\mathcal{K}
        },
      \\
      \label{eq:local_quasi_tnorm_st_step_ratio_3}
      \norm[0]{v-\Pi_hv}_{F_\mathcal{Q}}
      &
        \le c
        h_K^{1/2}\varepsilon^{-1/2}
        \del[1]{
        h_{K}\varepsilon^{1/2}
        \norm[0]{\partial_t v}_\mathcal{K}
        +
        \varepsilon^{1/2}
        \norm[0]{\overline{\nabla} v}_\mathcal{K}
        },
      \\
      \label{eq:local_quasi_tnorm_st_step_ratio_4}
      \norm[0]{v-\Pi_hv}_{F_\mathcal{R}}
      &
        \le c
        \varepsilon^{-1/2}
        \del[1]{
        h_{K}\varepsilon^{1/2}
        \norm[0]{\partial_t v}_\mathcal{K}
        +
        \varepsilon^{1/2}
        \norm[0]{\overline{\nabla} v}_\mathcal{K}
        }.
    \end{align}
  \end{subequations}
\end{lemma}
\begin{proof}
  See \cref{s:localquasitnormst}.
\end{proof}

We also have the following projection-related bounds.
\begin{lemma}
  \label{lem:otherusefulbnds}
  Let $v\in H^1(\mathcal{K})$ and consider $\Pi_h$ and
  $\Pi_h^{\mathcal{F}}$, the latter being the $L^2$-projection onto
  $M_h$. For any $\mathcal{K}\in\mathcal{T}_h$ and any
  $F_\mathcal{Q}\subset\mathcal{Q}_\mathcal{K}$,
  $F_\mathcal{R}\subset\mathcal{R}_\mathcal{K}$, we have
  \begin{subequations}
    \label{eq:otherusefulbnds}
    \begin{align}
      \norm[0]{\del[0]{\Pi_h-\Pi_{h}^{\mathcal{F}}}v}_{F_{\mathcal{Q}}}
      &\le
        ch_K^{1/2}\norm[0]{\overline{\nabla}v}_{\mathcal{K}},
        \label{eq:otherusefulbnds_2}
      \\
      \norm[0]{\del[0]{\Pi_h-\Pi_{h}^{\mathcal{F}}}v}_{F_{\mathcal{R}}}
      &\le
        c\delta t_{\mathcal{K}}^{1/2}
        \norm[0]{\partial_t v}_{\mathcal{K}}.
        \label{eq:otherusefulbnds_3}
    \end{align}
  \end{subequations}
\end{lemma}
\begin{proof}
  \Cref{eq:otherusefulbnds_2} has been shown in
  \cite[Eq.(4.7c)]{Wang:2023}, while \cref{eq:otherusefulbnds_3} can
  be shown with similar steps.
\end{proof}

We define an averaging operator
$\mathcal{I}^c_h:V_h\rightarrow V_h\cap C^0(\mathcal{E})$. For any
$v_h\in V_h$, we first construct the coarsest conforming refinement
$\mathcal{T}_h^c$ of $\mathcal{T}_h$; the operator
$\mathcal{I}^c_hv_h$ is prescribed at vertices of $\mathcal{T}_h^c$ by
the average of the values of $v_h$ at the vertex (see \cite[section
22.2]{Ern:book}). For the Dirichlet boundary nodes, i.e. nodes on
$\partial\mathcal{E}_D$, $\mathcal{I}^c_hv_h$ is prescribed by
zero. Furthermore, given a space-time element
$\mathcal{K} \in \mathcal{T}_h$, we introduce
$\check{\mathcal{Q}}_{\mathcal{K}}^i$ to denote the union of
$\mathcal{Q}$-facets in $\mathcal{F}_{h}^i$ that have a non-empty
intersection with $\partial\mathcal{K}$.  Similarly we introduce
$\check{\mathcal{R}}_{\mathcal{K}}^i$.

\begin{lemma}
  \label{lem:oswald_local_st}
  For a space-time element $\mathcal{K}\in\mathcal{T}_h$, the
  averaging operator
  $\mathcal{I}_h^{c}:V_h\rightarrow {V_h\cap C^0(\mathcal{E})}$
  satisfies the following
  \begin{equation}
    \label{eq:oswald_local_st}
    \norm{v_h-\mathcal{I}^{c}_hv_h}_{\mathcal{K}}
    \le c
    \del[2]{
      \sum_{F\in\check{\mathcal{Q}}^i_\mathcal{K}}
      h_K^{{1}/{2}}
      \norm{\jump{v_h}}_F
      +
      \sum_{F\in\check{\mathcal{R}}^i_\mathcal{K}}
      \delta t_{\mathcal{K}}^{{1}/{2}}
      \norm{\jump{v_h}}_F
    }.
  \end{equation}
\end{lemma}
\begin{proof}
  See \cref{s:oswaldlocalst}.
\end{proof}

Let $\mathcal{T}_\mathfrak{h}$ be the subgrid obtained by halving the
time-step of each element in $\mathcal{T}_h$. For each
$\mathcal{K}\in\mathcal{T}_h$, we introduce $\mathring{\mathcal{K}}^*$
and $\mathring{\mathcal{K}}_*$ to denote the two resulting space-time
elements in $\mathcal{T}_{\mathfrak{h}}$, i.e.,
$\mathcal{K}=\mathring{\mathcal{K}}^*\cup\mathring{\mathcal{K}}_*$,
and write
$\mathcal{T}_{\mathcal{K}} := \cbr[0]{\mathring{\mathcal{K}}^*,
  \mathring{\mathcal{K}}_*}$. Furthermore, every $\mathcal{Q}$-facet
$F_{\mathcal{Q}}\subset\mathcal{Q}_{\mathcal{K}}$ is divided into
$F_{\mathcal{Q}}^*$ and $F_{\mathcal{Q},*}$. We define
$F_{\mathring{\mathcal{R}}} := \partial \mathring{\mathcal{K}}^* \cap
\partial \mathring{\mathcal{K}}_*$ and introduce
$E_{\mathring{\mathcal{K}}}$ to denote any edge of
$F_{\mathring{\mathcal{R}}}$. Finally, for any
$v_{\mathfrak{h}}\in V_{\mathfrak{h}}$, when considering a
$\mathcal{K} \in \mathcal{T}_h$ with
$\mathcal{K}=\mathring{\mathcal{K}}^*\cup\mathring{\mathcal{K}}_*$, we
let $v_{\mathfrak{h}}^*$ and $v_{\mathfrak{h},*}$ denote
$v_{\mathfrak{h}}|_{\mathring{\mathcal{K}}^*}$ and
$v_{\mathfrak{h}}|_{\mathring{\mathcal{K}}_*}$, respectively.  See
\cref{fig:st_subgrid} for illustrations in $(1+1)$ and
$(2+1)$-dimensional space-time domains.

\begin{figure}[tbp]
  \centering
  \begin{subfigure}{.5\textwidth}
    \tikzmath{
      \x0 = 0; \y0 = 0;
      \x1 = 0; \y1 = 1;
      \x2 = 0; \y2 = 2;
      \x3 = 0; \y3 = 3;
      \x4 = 0; \y4 = 4;
      \x5 = 4; \y5 = 4;
      \x6 = 4; \y6 = 3;
      \x7 = 4; \y7 = 2;
      \x8 = 4; \y8 = 1;
      \x9 = 4; \y9 = 0;
      \w0 = (\x1 + \x8)/2; \v0 = (\y1 + \y8)/2;
      \w1 = (\x2 + \x7)/2; \v1 = (\y2 + \y7)/2;
      \w2 = (\x3 + \x6)/2; \v2 = (\y3 + \y6)/2;
      \s0 = (\x1 + \x2)/2; \t0 = (\y1 + \y2)/2;
      \s1 = ((\x7 + \x8)/2 +\s0)/2; \t1 = ((\y7 + \y8)/2+\t0)/2;
      \s2 = (\w0 + \x1)/2; \t2 = (\v0 + \y1)/2;
      \s3 = (\w1 + \x2)/2; \t3 = (\v1 + \y2)/2;
      \s4 = (\w2 + \x6)/2; \t4 = (\v2 + \y6)/2;
      \s5 = (\w1 + \x7)/2; \t5 = (\v1 + \y7)/2;
      \s6 = (\x3 + \w2)/2; \t6 = (\y3 + \v2)/2;
      \a0 = (\w1 + \w2)/2; \b0 = (\v1 + \v2)/2;
      \a1 = (\s4 + \s5)/2; \b1 = (\t4 + \t5)/2;
      \d0 = 8;
      \X0 = \x0 +\d0; \Y0 = 0;
      \X1 = \x1 +\d0; \Y1 = 1;
      \X2 = \x2 +\d0; \Y2 = 2;
      \X3 = \x3 +\d0; \Y3 = 3;
      \X4 = \x4 +\d0; \Y4 = 4;
      \X5 = \x5 +\d0; \Y5 = 4;
      \X6 = \x6 +\d0; \Y6 = 3;
      \X7 = \x7 +\d0; \Y7 = 2;
      \X8 = \x8 +\d0; \Y8 = 1;
      \X9 = \x9 +\d0; \Y9 = 0;
      \W0 = (\X1 + \X8)/2; \V0 = (\Y1 + \Y8)/2;
      \W1 = (\X2 + \X7)/2; \V1 = (\Y2 + \Y7)/2;
      \W2 = (\X3 + \X6)/2; \V2 = (\Y3 + \Y6)/2;
      \S0 = (\X1 + \X2)/2; \T0 = (\Y1 + \Y2)/2;
      \S1 = ((\X7 + \X8)/2 +\S0)/2; \T1 = ((\Y7 + \Y8)/2+\T0)/2;
      \S2 = (\W0 + \X1)/2; \T2 = (\V0 + \Y1)/2;
      \S3 = (\W1 + \X2)/2; \T3 = (\V1 + \Y2)/2;
      \S4 = (\W2 + \X6)/2; \T4 = (\V2 + \Y6)/2;
      \S5 = (\W1 + \X7)/2; \T5 = (\V1 + \Y7)/2;
      \S6 = (\X3 + \W2)/2; \T6 = (\Y3 + \V2)/2;
      \S7 = (\X7 + \X8)/2; \T7 = (\Y7 + \Y8)/2;
      \S8 = (\X2 + \X3)/2; \T8 = (\Y2 + \Y3)/2;
      \S9 = (\X6 + \X7)/2; \T9 = (\Y6 + \Y7)/2;
      \A0 = (\W1 + \W2)/2; \B0 = (\V1 + \V2)/2;
      \A1 = (\S4 + \S5)/2; \B1 = (\T4 + \T5)/2;
      \A2 = (3*\W1 + \W2)/4; \B2 = (3*\V1 + \V2)/4;
      \A3 = (\W1 + 3*\W2)/4; \B3 = (\V1 + 3*\V2)/4;
      \A4 = (3*\S5 + \S4)/4; \B4 = (3*\T5 + \T4)/4;
      \A5 = (\S5 + 3*\S4)/4; \B5 = (\T5 + 3*\T4)/4;
      \A6 = (3*\X1 + \X2)/4; \B6 = (3*\Y1 + \Y2)/4;
      \A7 = (\X1 + 3*\X2)/4; \B7 = (\Y1 + 3*\Y2)/4;
      \A8 = (3*\W0 + \W1)/4; \B8 = (3*\V0 + \V1)/4;
      \A9 = (\W0 + 3*\W1)/4; \B9 = (\V0 + 3*\V1)/4;
      \e0 = \x9 + 1.5; \f0 = (\y0+\y5)/2;
      \E0 = \X0 - 0.5; \F0 = (\Y0+\Y4)/2;
    }
    \begin{tikzpicture}[scale=0.6]
      \draw[thick,->] (\x0,\y0) -- (\x0,\y4+0.5) node[above] {$t$};
      \draw[thick] (\x0,\y0) -- node[below] {$\Omega_0$} (\x9,\y0);
      \draw[thick,->] (\x9,\y0) -- (\x9+0.5,\y0) node[right] {$x$};
      \draw[thick] (\x0,\y0)
      -- (\x1,\y1)
      -- (\x2,\y2)
      -- (\x3,\y3)
      -- (\x4,\y4);
      \draw[thick] (\x4,\y4) -- node[above] {$\Omega_T$} (\x5,\y5);
      \draw[thick] (\x5,\y5)
      -- (\x6,\y6)
      -- (\x7,\y7)
      -- (\x8,\y8)
      -- (\x9,\y9)
      -- (\x0,\y0);
      \draw[thick] (\x1,\y1) -- (\x8,\y8);
      \draw[thick] (\x2,\y2) -- (\x7,\y7);
      \draw[thick] (\x3,\y3) -- (\x6,\y6);
      \draw[thick] (\w0,\v0) -- (\w1,\v1) -- (\w2,\v2);
      \draw[thick] (\s0,\t0) -- (\s1,\t1);
      \draw[thick] (\s2,\t2) -- (\s3,\t3);
      \draw[thick] (\s3,\t3) -- (\s6,\t6);
      \draw[thick] (\s4,\t4) -- (\s5,\t5);
      \draw[thick] (\a0,\b0) -- (\a1,\b1);
      \draw[thick,->] (\e0-1/2,\f0) -- node[above] {subgrid}
			(\E0-1/2,\F0);
      \draw[thick,->] (\X0,\Y0) -- (\X0,\Y4+0.5) node[above] {$t$};
      \draw[thick] (\X0,\Y0) -- node[below] {$\Omega_0$} (\X9,\Y0);
      \draw[thick,->] (\X9,\Y0) -- (\X9+0.5,\Y0) node[right] {$x$};
      \draw[thick] (\X0,\Y0)
      -- (\X1,\Y1)
      -- (\X2,\Y2)
      -- (\X3,\Y3)
      -- (\X4,\Y4);
      \draw[thick] (\X4,\Y4) -- node[above] {$\Omega_T$} (\X5,\Y5);
      \draw[thick] (\X5,\Y5)
      -- (\X6,\Y6)
      -- (\X7,\Y7)
      -- (\X8,\Y8)
      -- (\X9,\Y9)
      -- (\X0,\Y0);
      \draw[thick] (\X1,\Y1) -- (\X8,\Y8);
      \draw[thick] (\X2,\Y2) -- (\X7,\Y7);
      \draw[thick] (\X3,\Y3) -- (\X6,\Y6);
      \draw[thick] (\W0,\V0) -- (\W1,\V1) -- (\W2,\V2);
      \draw[thick] (\S0,\T0) -- (\S1,\T1);
      \draw[thick] (\S2,\T2) -- (\S3,\T3);
      \draw[thick] (\S3,\T3) -- (\S6,\T6);
      \draw[thick] (\S4,\T4) -- (\S5,\T5);
      \draw[thick] (\A0,\B0) -- (\A1,\B1);
      \draw[thick,dashed] (\S1,\T1) -- (\S7,\T7);
      \draw[thick,dashed] (\S8,\T8) -- (\A0,\B0);
      \draw[thick,dashed] (\S9,\T9) -- (\A1,\B1);
      \draw[thick,dashed] (\A2,\B2) -- (\A4,\B4);
      \draw[thick,dashed] (\A3,\B3) -- (\A5,\B5);
      \draw[thick,dashed] (\A6,\B6) -- (\A8,\B8);
      \draw[thick,dashed] (\A7,\B7) -- (\A9,\B9);
    \end{tikzpicture}
  \end{subfigure}
  %
  %
  \begin{subfigure}{.35\textwidth}
    \tikzmath{
      \x0 = -0.5; \y0 = -1.5;
      \x1 = \x0 + 1; \y1 =\y0;
      \x2 = \x0 + 0.5; \y2 =\y0 + 0.5;
      \x3 = \x0; \y3 =\y0 + 1;
      \h0 = 4.5; 	 \h1 = 2;
      \a0 = 2; 	 \b0 = -4;
      \a1 = \a0 + \h0; \b1 = \b0 - 0.2;
      \a2 = \a1 + \h1; \b2 = \b1 + \h1 + 0.2;
      \a3 = \a0 + \h1; \b3 = \b0 + \h1;
      %
      %
      \v0 = -0.5;		\v1 = 2.5;
      \h2 = 2;			\h3 = 0.8;
      \a4 = \a0 + \v0;	\b4 = \b0 + \v1;
      \a5 = \a4 + \h0;	\b5 = \b4 - 0.2;
      \a6 = \a5 + \h1;	\b6 = \b5 + \h1 + 0.2;
      \a7 = \a4 + \h1;	\b7 = \b4 + \h1;
      %
      %
      \m0 = (\a0+\a4)/2; \n0 = (\b0+\b4)/2;
      \m1 = (\a1+\a5)/2; \n1 = (\b1+\b5)/2;
      \m2 = (\a2+\a6)/2; \n2 = (\b2+\b6)/2;
      \m3 = (\a3+\a7)/2; \n3 = (\b3+\b7)/2;
      %
      %
      \p0 = (6*\m0+1*\m1)/7; \q0 = (6*\n0+1*\n1)/7;
      \p1 = (5*\m0+2*\m1)/7; \q1 = (5*\n0+2*\n1)/7;
      \p2 = (4*\m0+3*\m1)/7; \q2 = (4*\n0+3*\n1)/7;
      \p3 = (3*\m0+4*\m1)/7; \q3 = (3*\n0+4*\n1)/7;
      \p4 = (2*\m0+5*\m1)/7; \q4 = (2*\n0+5*\n1)/7;
      \p5 = (1*\m0+6*\m1)/7; \q5 = (1*\n0+6*\n1)/7;
      \e0 = (6*\m1+1*\m2)/7; \f0 = (6*\n1+1*\n2)/7;
      \e1 = (5*\m1+2*\m2)/7; \f1 = (5*\n1+2*\n2)/7;
      \e2 = (4*\m1+3*\m2)/7; \f2 = (4*\n1+3*\n2)/7;
      \e3 = (3*\m1+4*\m2)/7; \f3 = (3*\n1+4*\n2)/7;
      \e4 = (2*\m1+5*\m2)/7; \f4 = (2*\n1+5*\n2)/7;
      \e5 = (1*\m1+6*\m2)/7; \f5 = (1*\n1+6*\n2)/7;
      \k0 = (6*\m3+1*\m2)/7; \r0 = (6*\n3+1*\n2)/7;
      \k1 = (5*\m3+2*\m2)/7; \r1 = (5*\n3+2*\n2)/7;
      \k2 = (4*\m3+3*\m2)/7; \r2 = (4*\n3+3*\n2)/7;
      \k3 = (3*\m3+4*\m2)/7; \r3 = (3*\n3+4*\n2)/7;
      \k4 = (2*\m3+5*\m2)/7; \r4 = (2*\n3+5*\n2)/7;
      \k5 = (1*\m3+6*\m2)/7; \r5 = (1*\n3+6*\n2)/7;
      \s0 = (6*\m0+1*\m3)/7; \t0 = (6*\n0+1*\n3)/7;
      \s1 = (5*\m0+2*\m3)/7; \t1 = (5*\n0+2*\n3)/7;
      \s2 = (4*\m0+3*\m3)/7; \t2 = (4*\n0+3*\n3)/7;
      \s3 = (3*\m0+4*\m3)/7; \t3 = (3*\n0+4*\n3)/7;
      \s4 = (2*\m0+5*\m3)/7; \t4 = (2*\n0+5*\n3)/7;
      \s5 = (1*\m0+6*\m3)/7; \t5 = (1*\n0+6*\n3)/7;
      %
      %
      \e6 = (\e3+\k3)/2; \f6 = (\f3+\r3)/2;
      \e7 = (\m1+\m2)/2; \f7 = (\n1+\n2)/2;
    }
    \begin{tikzpicture}[scale=0.5]
      \draw[thick,->] (\x0,\y0-2) -- (\x1,\y1-2) node[right] {$x_1$};
      \draw[thick,->] (\x0,\y0-2) -- (\x2,\y2-2) node[right] {$x_2$};
      \draw[thick,->] (\x0,\y0-2) -- (\x3,\y3-2) node[above] {$t$};
      \draw[thick] (\a0,\b0) -- (\a1,\b1);
      \draw[thick] (\a1,\b1) -- (\a2,\b2);
      \draw[thick,dashed] (\a2,\b2) -- (\a3,\b3);
      \draw[thick,dashed] (\a3,\b3) -- (\a0,\b0);
      \draw[thick] (\a4,\b4) -- (\a5,\b5);
      \draw[thick] (\a5,\b5) -- (\a6,\b6);
      \draw[thick] (\a6,\b6) -- (\a7,\b7);
      \draw[thick] (\a7,\b7) -- (\a4,\b4);
      \draw[thick] (\a0,\b0) -- (\a4,\b4);
      \draw[thick] (\a1,\b1) -- (\a5,\b5);
      \draw[thick] (\a2,\b2) -- (\a6,\b6);
      \draw[thick,dashed] (\a3,\b3) -- (\a7,\b7);
      \draw[very thick] (\m0,\n0) -- (\m1,\n1);
      \draw[very thick] (\m1,\n1) -- (\m2,\n2);
      \draw[very thick] (\m2,\n2) -- (\m3,\n3);
      \draw[very thick] (\m3,\n3) -- (\m0,\n0);
      %
      %
      \draw[thick,dotted] (\p0,\q0) -- (\s0,\t0);
      \draw[thick,dotted] (\p1,\q1) -- (\s1,\t1);
      \draw[thick,dotted] (\p2,\q2) -- (\s2,\t2);
      \draw[thick,dotted] (\p3,\q3) -- (\s3,\t3);
      \draw[thick,dotted] (\p4,\q4) -- (\s4,\t4);
      \draw[thick,dotted] (\p5,\q5) -- (\s5,\t5);
      \draw[thick,dotted] (\m1,\n1) -- (\m3,\n3);
      \draw[thick,dotted] (\e0,\f0) -- (\k0,\r0);
      \draw[thick,dotted] (\e1,\f1) -- (\k1,\r1);
      \draw[thick,dotted] (\e2,\f2) -- (\k2,\r2);
      \draw[thick,dotted] (\e3,\f3) -- (\k3,\r3);
      \draw[thick,dotted] (\e4,\f4) -- (\k4,\r4);
      \draw[thick,dotted] (\e5,\f5) -- (\k5,\r5);
      \draw[>=triangle 45,->] (\m2, \n2+2) node[right]
      {\large$F_{\mathring{\mathcal{R}}}$} to [bend right=30]
      (\e6-1, \f6);
      \draw[>=triangle 45,->] (\a1+2, \b1+0.5) node[right]
      {\large${E}_{\mathring{\mathcal{K}}}$} to  (\e7, \f7);
    \end{tikzpicture}
  \end{subfigure}
  \caption{ In the left-hand side of the figure we show a
    $(1+1)$-dimensional example of constructing the subgrid
    while the right-hand side of the figure gives a
    $(2+1)$-dimensional illustration of the new facets and edges
    resulting from the subgrid refinement.}
  \label{fig:st_subgrid}
\end{figure}
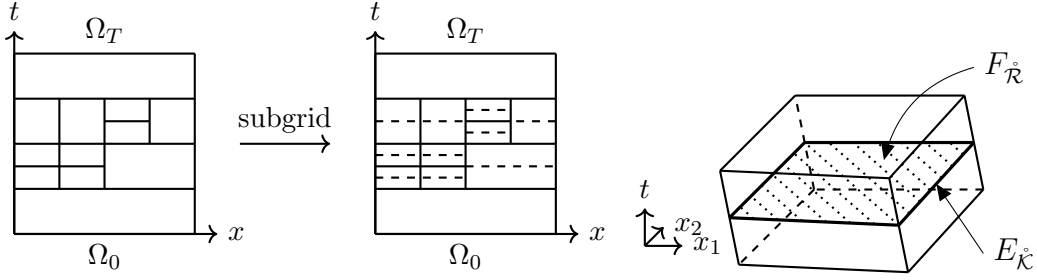

The following trace inequalities are obtained by applying
\cref{eq:eg_inv_low_d_4} and \cref{eq:eg_inv_low_d_5}.
\begin{lemma}
  \label{lem:trace_ineq_subgrid_edge}
  On the subgrid $\mathcal{T}_{\mathfrak{h}}$, the following trace
  inequalities hold (where, for each inequality, it is implicitly
  assumed that $E_{\mathring{\mathcal{K}}}$ is an edge of the
  facet on
  the right-hand side):
  \begin{equation}
    \label{eq:trace_ineq_subgrid_edge}
    \begin{aligned}
      \norm[0]{v_\mathfrak{h}^*}_{E_{\mathring{\mathcal{K}}}}
      &\le c
      {\delta t_{\mathcal{K}}^{-{1}/{2}}}
      \norm[0]{v_\mathfrak{h}^*}_{F_\mathcal{Q}^*},
      &
      \norm[0]{v_{\mathfrak{h},*}}_{E_{\mathring{\mathcal{K}}}}
      &\le c
      \delta t_{\mathcal{K}}^{-{1}/{2}}
      \norm[0]{v_{\mathfrak{h},*}}_{F_{\mathcal{Q},*}},
      \\
      \norm[0]{v_\mathfrak{h}^*}_{E_{\mathring{\mathcal{K}}}}
      &\le c
      h_K^{-{1}/{2}}
      \norm[0]{v_{\mathfrak{h}}^*}_{F_{\mathring{\mathcal{R}}}},
      &
      \norm[0]{v_{\mathfrak{h},*}}_{E_{\mathring{\mathcal{K}}}}
      &\le c
      h_K^{-{1}/{2}}
      \norm[0]{v_{\mathfrak{h},*}}_{F_{\mathring{\mathcal{R}}}}.
    \end{aligned}
  \end{equation}
\end{lemma}

\begin{definition}
  \label{def:subgrid_proj}
  We define the following restriction operator:
  \begin{equation}
    \label{eq:res_opt}
    \gamma_{\mathfrak{h}}:\boldsymbol{V}_h\rightarrow
    \boldsymbol{V}_{\mathfrak{h}}
    \ : \
    \del{{v}_h,\mu_h}
    \mapsto
    \del{
      v_h
      ,
      \gamma_{\mathcal{F},{\mathfrak{h}}}\del{\boldsymbol{v}_h}
    },
    \qquad
    \gamma_{\mathcal{F},{\mathfrak{h}}}\del{\boldsymbol{v}_h}
    :=
    \begin{cases}
      \mu_h,
      &\forall F\in
      \mathcal{F}_{\mathcal{Q},\mathfrak{h}}
      \cup\mathcal{F}_{\mathcal{R},h},
      \\
      v_h,
      &\forall
      F\in\mathcal{F}_{\mathcal{R},{\mathfrak{h}}}\setminus
      \mathcal{F}_{\mathcal{R},{h}}.
    \end{cases}
  \end{equation}
  Furthermore, let $i_h^{\mathcal{K}}(v_{\mathfrak{h}})$ denote the
  $L^2$-projection of $v_{\mathfrak{h}}$ onto ${V}_h$, and let
  $i^{\mathcal{F}}_h(\mu_{\mathfrak{h}})$ be defined as follows. For
  any facet $F \in \mathcal{F}_h$, if $F\in\mathcal{F}_{\mathfrak{h}}$,
  $\del[0]{i_h^{\mathcal{F}}(\mu_{\mathfrak{h}})}|_F:=\del[0]{\mu_{\mathfrak{h}}}|_F$;
  else, $\del[0]{i_h^{\mathcal{F}}(\mu_{\mathfrak{h}})}|_F$ is the
  $L^2$-projection of $\mu_{\mathfrak{h}}$ onto $M_{{h}}$. If
  $F \in\mathcal{F}_{\mathcal{R},\mathfrak{h}}
  \setminus\mathcal{F}_{\mathcal{R},h}$ we define
  $(i_h^{\mathcal{F}}(\mu_{\mathfrak{h}}))|_F :=
  (\mu_{\mathfrak{h}})|_F$. See \cref{fig:subgridprojection_facet}
  for an illustration of how $i_h^{\mathcal{F}}$ projects onto
  interior $\mathcal{Q}$-facets in $\mathcal{F}_{\mathcal{Q},{h}}$. We
  then define the projection operator:
  \begin{equation*}
    {i}_h:
    \boldsymbol{V}_{\mathfrak{h}} \rightarrow
    \boldsymbol{V}_{\mathfrak{h}} \,:\,
    \del[0]{{v}_\mathfrak{h},{\mu}_\mathfrak{h}} \mapsto
    \gamma_{\mathfrak{h}}\del[1]{ {i}_h^\mathcal{K}\del[0]{{v}_\mathfrak{h}} ,
      {i}_h^\mathcal{F}\del[0]{{\mu}_\mathfrak{h}} }.
  \end{equation*}
\end{definition}

\def\q{1cm}
\def\s{1.2cm}
\def\t{0.1cm}
\def\p{0.03cm}
\begin{figure}[tbp]
  \centering
  \subfloat\
  { \tikzmath{
      coordinate \x{01}, \x{02}, \x{03}, \x{04}, \x{05}, \x{06};
      \x{01} = (0cm,0cm);
      \x{02} = ([xshift=\q] \x{01});
      \x{03} = ([yshift=0.5\q] \x{02});
      \x{04} = ([yshift=0.5\q] \x{03});
      \x{05} = ([xshift=-\q] \x{04});
      \x{06} = ([yshift=-0.5\q] \x{05});
      coordinate \y{01}, \y{02}, \y{03}, \y{04}, \y{05}, \y{06};
      \y{01} = (0cm,-1.8cm);
      \y{02} = ([xshift=\q] \y{01});
      \y{03} = ([yshift=0.5\q] \y{02});
      \y{04} = ([yshift=0.5\q] \y{03});
      \y{05} = ([xshift=-\q] \y{04});
      \y{06} = ([yshift=-0.5\q] \y{05});
    }
    \begin{tikzpicture}[scale=1.8]
      \draw[] (\x{01})
      -- (\x{02})
      -- (\x{04})
      -- (\x{05})
      -- (\x{01});
      \draw[dashed] (\x{03}) -- (\x{06});
      \draw[] ([xshift=\s] \x{01})
      -- ([xshift=\s] \x{02})
      -- ([xshift=\s] \x{04})
      -- ([xshift=\s] \x{05})
      -- ([xshift=\s] \x{01});
      \draw[dashed] ([xshift=\s] \x{03})
      -- ([xshift=\s] \x{06});
      \draw[color=red,thick] ([xshift=\t,yshift=\p] \x{02})
      -- ([xshift=\t,yshift=-\p] \x{03});
      \draw[color=red,thick] ([xshift=\t,yshift=\p] \x{03})
      -- ([xshift=\t,yshift=-\p] \x{04});
      \draw ([xshift=\t] \x{02}) node[color=red,below]
      {$\mu_{\mathfrak{h}}$};
      \draw ([xshift=\t] \x{04}) node[color=red,above]
      {$\mu_{\mathfrak{h}}$};
      \draw[] (\y{01})
      -- (\y{02})
      -- (\y{04})
      -- (\y{05})
      -- (\y{01});
      \draw[dashed] (\y{03}) -- (\y{06});
      \draw[] ([xshift=\s] \y{01})
      -- ([xshift=\s] \y{02})
      -- ([xshift=\s] \y{04})
      -- ([xshift=\s] \y{05})
      -- ([xshift=\s] \y{01});
      \draw[dashed] ([xshift=\s] \y{03})
      -- ([xshift=\s] \y{06});
      \draw[color=red,thick] ([xshift=\t,yshift=\p] \y{02})
      -- ([xshift=\t,yshift=-\p] \y{04});
      \draw ([xshift=\t] \y{02})
      node[color=red,below]
      {$i_h^{\mathcal{F}}\mu_{\mathfrak{h}}$};
      \draw[->] ([xshift=\t,yshift=-8*\p] \x{02}) --
      node[right] {$i_h^{\mathcal{F}}$}
      ([xshift=\t,yshift=2*\p] \y{04});
    \end{tikzpicture}
  }
  \qquad
  \subfloat\
  { \tikzmath{
      coordinate \x{01}, \x{02}, \x{03}, \x{04}, \x{05}, \x{06};
      \x{01} = (0cm,0cm);
      \x{02} = ([xshift=\q] \x{01});
      \x{03} = ([yshift=0.5\q] \x{02});
      \x{04} = ([yshift=0.5\q] \x{03});
      \x{05} = ([xshift=-\q] \x{04});
      \x{06} = ([yshift=-0.5\q] \x{05});
      coordinate \x{07}, \x{08}, \x{09}, \x{10}, \x{11}, \x{12};
      \x{07} = ([xshift=0.5\q] \x{01});
      \x{08} = ([yshift=0.25\q] \x{02});
      \x{09} = ([yshift=0.25\q] \x{03});
      \x{10} = ([xshift=0.5\q] \x{05});
      \x{11} = ([yshift=0.25\q] \x{06});
      \x{12} = ([yshift=0.25\q] \x{01});
      coordinate \y{01}, \y{02}, \y{03}, \y{04}, \y{05}, \y{06};
      \y{01} = (0cm,-1.8cm);
      \y{02} = ([xshift=\q] \y{01});
      \y{03} = ([yshift=0.5\q] \y{02});
      \y{04} = ([yshift=0.5\q] \y{03});
      \y{05} = ([xshift=-\q] \y{04});
      \y{06} = ([yshift=-0.5\q] \y{05});
      coordinate \y{07}, \y{08}, \y{09}, \y{10}, \y{11}, \y{12};
      \y{07} = ([xshift=0.5\q] \y{01});
      \y{08} = ([yshift=0.25\q] \y{02});
      \y{09} = ([yshift=0.25\q] \y{03});
      \y{10} = ([xshift=0.5\q] \y{05});
      \y{11} = ([yshift=0.25\q] \y{06});
      \y{12} = ([yshift=0.25\q] \y{01});
    }
    \begin{tikzpicture}[scale=1.8]
      \draw[] (\x{01})
      -- (\x{02})
      -- (\x{04})
      -- (\x{05})
      -- (\x{01});
      \draw[dashed] (\x{03}) -- (\x{06});
      \draw[] ([xshift=\s] \x{01})
      -- ([xshift=\s] \x{02})
      -- ([xshift=\s] \x{04})
      -- ([xshift=\s] \x{05})
      -- ([xshift=\s] \x{01});
      \draw[] ([xshift=\s] \x{03}) -- ([xshift=\s] \x{06});
      \draw[] ([xshift=\s] \x{07}) -- ([xshift=\s] \x{10});
      \draw[dashed] ([xshift=\s] \x{08}) -- ([xshift=\s] \x{12});
      \draw[dashed] ([xshift=\s] \x{09}) -- ([xshift=\s] \x{11});
      \draw[color=red,thick] ([xshift=\t,yshift=\p] \x{02})
      -- ([xshift=\t,yshift=-\p] \x{03});
      \draw[color=red,thick] ([xshift=\t,yshift=\p] \x{03})
      -- ([xshift=\t,yshift=-\p] \x{04});
      \draw ([xshift=\t] \x{02}) node[color=red,below]
      {$\mu_{\mathfrak{h}}$};
      \draw ([xshift=\t] \x{04}) node[color=red,above]
      {$\mu_{\mathfrak{h}}$};
      \draw[] (\y{01})
      -- (\y{02})
      -- (\y{04})
      -- (\y{05})
      -- (\y{01});
      \draw[dashed] (\y{03}) -- (\y{06});
      \draw[] ([xshift=\s] \y{01})
      -- ([xshift=\s] \y{02})
      -- ([xshift=\s] \y{04})
      -- ([xshift=\s] \y{05})
      -- ([xshift=\s] \y{01});
      \draw[] ([xshift=\s] \y{03}) -- ([xshift=\s] \y{06});
      \draw[] ([xshift=\s] \y{07}) -- ([xshift=\s] \y{10});
      \draw[dashed] ([xshift=\s] \y{08}) -- ([xshift=\s] \y{12});
      \draw[dashed] ([xshift=\s] \y{09}) -- ([xshift=\s] \y{11});
      \draw[color=red,thick] ([xshift=\t,yshift=\p] \y{02})
      -- ([xshift=\t,yshift=-\p] \y{04});
      \draw ([xshift=\t] \y{02})
      node[color=red,below]
      {$i_h^{\mathcal{F}}\mu_{\mathfrak{h}}$};
      \draw[->] ([xshift=\t,yshift=-8*\p] \x{02}) --
      node[right] {$i_h^{\mathcal{F}}$}
      ([xshift=\t,yshift=2*\p] \y{04});
    \end{tikzpicture}
  }
  \caption{Illustration of subgrid projection $i_h^{\mathcal{F}}$ onto
    an interior $\mathcal{Q}$-facet in
    $\mathcal{F}_{\mathcal{Q},{h}}$. The neighboring elements of the
    $\mathcal{Q}$-facet are on the same refinement level in the left
    column and are on different refinement levels in the right
    column.}
  \label{fig:subgridprojection_facet}
\end{figure}
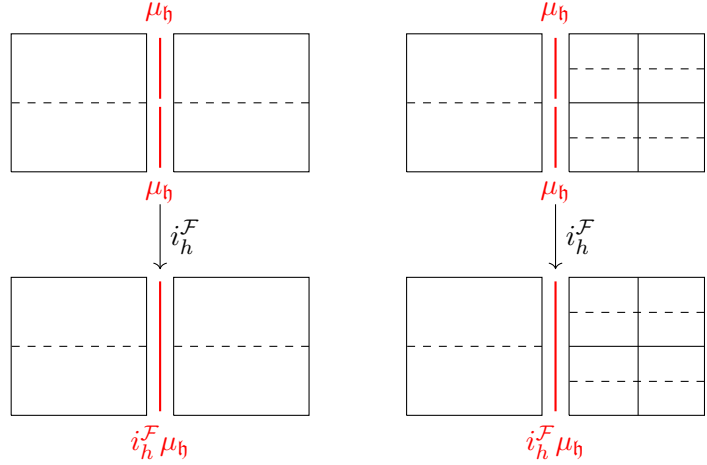

We have the following projection estimates.

\begin{lemma}
  \label{lem:subgrid_proj_est}
  Let ${\boldsymbol{v}_\mathfrak{h}}\in \boldsymbol{V}_\mathfrak{h}$,
  let the projection operator $i_h$ be defined as in
  \cref{def:subgrid_proj}. Then,
  \begin{subequations}
    \label{eq:subgrid_proj_est}
    \begin{align}
      \label{eq:subgrid_proj_est_1}
      \norm[0]{\del[0]{I-i_h^\mathcal{K}}v_\mathfrak{h}}_\mathcal{K}
      &\le c
        \del[1]{
        \delta t_{\mathcal{K}}^{{1}/{2}}\norm{\jump{v_\mathfrak{h}}}_{F_{\mathring{\mathcal{R}}}}
        +
        \delta t_{\mathcal{K}}^{{3}/{2}}\norm{\jump{\partial_tv_\mathfrak{h}}}_{F_{\mathring{\mathcal{R}}}}
        } && \text{for } \mathcal{K} \in \mathcal{T}_h,
      \\
      \label{eq:subgrid_proj_est_2}
      \norm[0]{\del[0]{I-i_h^\mathcal{F}}\mu_\mathfrak{h}}_{F_{\mathcal{Q}}}
      &\le c
        \del[1]{
        \delta t_{\mathcal{K}}^{{1}/{2}}
        \norm{\ejump{\mu_\mathfrak{h}}}_{E_{\mathring{\mathcal{K}}}}
        +
        \delta t_{\mathcal{K}}^{{3}/{2}}
        \norm{\ejump{\partial_t\mu_\mathfrak{h}}}_{E_{\mathring{\mathcal{K}}}}}
          && \text{for } F_{\mathcal{Q}}\in\mathcal{F}_{\mathcal{Q},h},
    \end{align}
  \end{subequations}
  where $\mathcal{K}$ on the right-hand side of
  \cref{eq:subgrid_proj_est_2} is chosen such that
  $F_{\mathcal{Q}}\subset\mathcal{Q}_{\mathcal{K}}$.
\end{lemma}
\begin{proof}
  See \cref{s:subgridprojestproof}.
\end{proof}

\begin{lemma}
  \label{lem:subgridhelperbnds}
  Let
  ${\boldsymbol{v}_{\mathfrak{h}}} \in \boldsymbol{V}_{\mathfrak{h}}$,
  there holds:
  \begin{subequations}
    \label{eq:subgridhelperbnds}
    \begin{align}
      \label{eq:subgridhelperbnds_1}
      \norm[0]{\jump{v_{\mathfrak{h}}}}_{F_{\mathring{\mathcal{R}}}}
      & \le c
        \sum_{\mathring{\mathcal{K}}\in \mathcal{T}_\mathcal{K}}
        \norm[0]{
        \envert[0]{\beta_s-\tfrac{1}{2}\beta\cdot n}^{1/2}
        \sbr[0]{\boldsymbol{v}_{\mathfrak{h}}}
        }_{ \partial\mathring{\mathcal{K}}
        \cap
        F_{\mathring{\mathcal{R}}}
        }
        \le c
        \tnorm{\boldsymbol{v}_{\mathfrak{h}}}_{s,\mathfrak{h}},
      \\
      \label{eq:subgridhelperbnds_2}
      \norm[0]{\jump{\partial_tv_{\mathfrak{h}}}}_{F_{\mathring{\mathcal{R}}}}
      & \le c
        \sum_{\mathring{\mathcal{K}}\in \mathcal{T}_\mathcal{K}}
        \norm[0]{
        \partial_tv_{\mathfrak{h}}
        }_{ \partial\mathring{\mathcal{K}}
        \cap
        F_{\mathring{\mathcal{R}}}
        }
        \le c
        \sum_{\mathring{\mathcal{K}}\in \mathcal{T}_\mathcal{K}}
        {\delta{t}_{\mathcal{K}}^{-1/2}}
        \norm[0]{
        \partial_tv_{\mathfrak{h}}
        }_{\mathring{\mathcal{K}}},
      \\
      \label{eq:subgridhelperbnds_4}
      \norm[0]{\ejump{\mu_{\mathfrak{h}}}}_{E_{\mathring{\mathcal{K}}}}
      &\le c
        h_K^{-1/2}
        \sum_{\mathring{\mathcal{K}}\in \mathcal{T}_\mathcal{K}}
        \norm[0]{
        \envert[0]{\beta_s-\tfrac{1}{2}\beta\cdot n}^{1/2}
        \sbr[0]{\boldsymbol{v}_{\mathfrak{h}}}
        }_{ \partial\mathring{\mathcal{K}}
        \cap
        F_{\mathring{\mathcal{R}}}
        }
        +
        c \delta{t}_{\mathcal{K}}^{-1/2}
		  \del[2]{
        \norm[0]{\sbr[0]{\boldsymbol{v}_{\mathfrak{h}}^*}}_{F_{\mathcal{Q}}^*}
        +
	  \norm[0]{\sbr[0]{\boldsymbol{v}_{\mathfrak{h},*}}}_{F_{\mathcal{Q},*}}},
      \\
      \label{eq:subgridhelperbnds_5}
      \norm[0]{\ejump{\partial_t\mu_{\mathfrak{h}}}}_{E_{\mathring{\mathcal{K}}}}
      &\le c
        h_K^{-1/2}\delta{t}_{\mathcal{K}}^{-1/2}
        \sum_{\mathring{\mathcal{K}}\in \mathcal{T}_\mathcal{K}}
        \norm[0]{
        \partial_tv_{\mathfrak{h}}
        }_{\mathring{\mathcal{K}}}
        + c
		  \delta{t}_{\mathcal{K}}^{-3/2}
		  \del[2]{
			  \norm[0]{\sbr[0]{\boldsymbol{v}_{\mathfrak{h}}^*}}_{F_{\mathcal{Q}}^*}
			  + 
			  \norm[0]{\sbr[0]{\boldsymbol{v}_{\mathfrak{h},*}}}_{F_{\mathcal{Q},*}}
		  }.
    \end{align}
  \end{subequations}
\end{lemma}
\begin{proof}
  See \cref{s:subgridhelperbndsproof}.
\end{proof}

\begin{lemma}
  \label{lem:subgridprojdiff}
  Let
  ${\boldsymbol{v}_{\mathfrak{h}}}\in \boldsymbol{V}_{\mathfrak{h}}$
  and let the projection operator $i_h$ be defined as in
  \cref{def:subgrid_proj}. Consider an element
  $\mathcal{K} \in \mathcal{T}_h$ such that it has a
  $\mathcal{Q}$-facet
  $F_{\mathcal{Q}} \subset \mathcal{Q}_{\mathcal{K}}$ such that
  $F_{\mathcal{Q}} \in \mathcal{F}_{\mathcal{Q},h}$. There holds that
  $\del[0]{{i_h^{\mathcal{K}}}v_{\mathfrak{h}}}|_{F_{\mathcal{Q}}}
  \equiv {i_h^{\mathcal{F}}} \del[0]{
    {v}_{\mathfrak{h}}|_{F_{\mathcal{Q}}} }$ on
  $\mathcal{F}_{\mathcal{Q},h}$.
\end{lemma}
\begin{proof}
  See \cref{s:subgridprojdiffproof}.
\end{proof}

\subsection{Saturation assumption and time derivative error
  estimation}
\label{ss:saturationassumption}

We pose problem \cref{eq:st_hdg_adr_compact} on the subgrid mesh
$\mathcal{T}_{\mathfrak{h}}$, i.e., find
$\boldsymbol{u}_{\mathfrak{h}}\in \boldsymbol{V}_{\mathfrak{h}}$ such
that
\begin{equation}
  \label{eq:st_hdg_adr_subgrid}
  a_{\mathfrak{h}}\del{
    \boldsymbol{u}_{\mathfrak{h}}
    ,
    \boldsymbol{v}_{\mathfrak{h}}
  }
  =
  \del{f,v_{\mathfrak{h}}}_{\mathcal{T}_{\mathfrak{h}}}
  +
  \langle{g,\mu_{\mathfrak{h}}}\rangle_{\partial\mathcal{E}_N}
  \quad
  \forall\boldsymbol{v}_{\mathfrak{h}}\in\boldsymbol{V}_{\mathfrak{h}}.
\end{equation}

\begin{lemma}[Galerkin orthogonality]
  \label{lem:glk_ort}
  Let $\boldsymbol{u}_h$ and $\boldsymbol{u}_\mathfrak{h}$ be the
  solutions of \cref{eq:st_hdg_adr_compact} and
  \cref{eq:st_hdg_adr_subgrid}, respectively. With the restriction
  operator defined in \cref{eq:res_opt}, we have the following
  Galerkin orthogonality result:
  \begin{equation}
    \label{eq:glk_ort}
    a_{\mathfrak{h}}
    \del{
      \boldsymbol{u}_{\mathfrak{h}}
      -
      \gamma_{\mathfrak{h}}\del{\boldsymbol{u}_h}
      ,
      \gamma_{\mathfrak{h}}\del{\boldsymbol{v}_h}
    }=0
    \quad
    \forall \boldsymbol{v}_h\in \boldsymbol{V}_h.
  \end{equation}
\end{lemma}
\begin{proof}
  For any $\boldsymbol{v}_h\in\boldsymbol{V}_h$,
  \begin{equation*}
    a_{\mathfrak{h}}
    \del{
      \boldsymbol{u}_\mathfrak{h}
      ,
      \gamma_\mathfrak{h}\del{\boldsymbol{v}_h}
    }
    =
    \del{f,v_h}_{\mathcal{T}_\mathfrak{h}}
    +
    \langle{g,\gamma_{\mathcal{F},\mathfrak{h}}\del{\boldsymbol{v}_h}}\rangle_{\partial\mathcal{E}_N}
    =
    \del{f,v_h}_{\mathcal{T}_\mathfrak{h}}
    +
    \langle{ g,\mu_h }\rangle_{\partial\mathcal{E}_N}
    =
    a_{h}
      \del{
      \boldsymbol{u}_h
      ,
      \boldsymbol{v}_h
      }
  \end{equation*}
  To show
  \cref{eq:glk_ort}, it remains to show that
  $a_{h} \del{ \boldsymbol{u}_h , \boldsymbol{v}_h }=a_{\mathfrak{h}}
  \del{ \gamma_\mathfrak{h}\del{\boldsymbol{u}_h},
    \gamma_\mathfrak{h}\del{\boldsymbol{v}_h}}$. First, note for the
  element integrals we have,
  \begin{equation*}
    \del[0]{\varepsilon\overline{\nabla}u_h,\overline{\nabla}v_h}_{\mathcal{T}_h}
    -\del[0]{\beta u_h,\nabla v_h}_{\mathcal{T}_h}
    =
    \del[0]{\varepsilon\overline{\nabla}u_h,\overline{\nabla}v_h}_{\mathcal{T}_\mathfrak{h}}
    -\del[0]{\beta u_h,\nabla v_h}_{\mathcal{T}_\mathfrak{h}},
  \end{equation*}
  and for the diffusion facet terms,
  \begin{equation*}
    \langle{\varepsilon\overline{\nabla}_{\overline{{n}}}u_h,\sbr{\boldsymbol{v}_h}}\rangle_{\mathcal{Q}_h}
    +
    \langle{\varepsilon\sbr{\boldsymbol{u}_h},\overline{\nabla}_{\overline{{n}}}v_h}\rangle_{\mathcal{Q}_h}
    =
    \langle{\varepsilon\overline{\nabla}_{\overline{{n}}}u_h,\sbr{\gamma_\mathfrak{h}\del{\boldsymbol{v}_h}}}\rangle_{\mathcal{Q}_\mathfrak{h}}
    +
    \langle{\varepsilon\sbr{\gamma_\mathfrak{h}\del{\boldsymbol{u}_h}},\overline{\nabla}_{\overline{{n}}}v_h}\rangle_{\mathcal{Q}_\mathfrak{h}}.
  \end{equation*}
  Next, since $\sbr{\gamma_\mathfrak{h}\del{\boldsymbol{v}_h}}=0$ on
  $\mathcal{R}_\mathfrak{h}\setminus\mathcal{R}_h$, we have
  \begin{equation*}
    \langle
    \del{\beta\cdot{n}}\gamma_{\mathcal{F},\mathfrak{h}}\del{\boldsymbol{u}_h}
    +
    \beta_s\sbr{\gamma_\mathfrak{h}\del{\boldsymbol{u}_h}}
    ,
    \sbr{\gamma_\mathfrak{h}\del{\boldsymbol{v}_h}}
    \rangle_{\partial\mathcal{T}_\mathfrak{h}}
    =
    \langle
    \del{\beta\cdot{n}}\lambda_h
    +
    \beta_s\sbr{\boldsymbol{u}_h}
    ,
    \sbr{\boldsymbol{v}_h}
    \rangle_{\partial\mathcal{T}_h},
  \end{equation*}
  and similarly, on the Neumann boundary, we have for the advective
  facet terms
  \begin{equation*}
    \langle
    \zeta^+\beta\cdot{n}
    \gamma_{\mathcal{F},\mathfrak{h}}\del{\boldsymbol{u}_h},
    \gamma_{\mathcal{F},\mathfrak{h}}\del{\boldsymbol{v}_h}
    \rangle_{\partial\mathcal{E}_N}
    =
    \langle
    \zeta^+\beta\cdot{n}
    \lambda_h,\mu_h
    \rangle_{\partial\mathcal{E}_N}.
  \end{equation*}
  Finally, for the penalty term,
  \begin{equation*}
   \langle{{\varepsilon\alpha}{h_K^{-1}}\sbr[0]{\gamma_\mathfrak{h}\del{\boldsymbol{u}_h}},
     \sbr[0]{\gamma_\mathfrak{h}\del{\boldsymbol{v}_h}}}\rangle_{\mathcal{Q}_\mathfrak{h}}
   =
   \langle{{\varepsilon\alpha}{h_K^{-1}}\sbr{\boldsymbol{u}_h},
     \sbr{\boldsymbol{v}_h}}\rangle_{\mathcal{Q}_h},
  \end{equation*}
  because the spatial element size parameter $h_K$ does not change
  from $\mathcal{K}$ to $\mathring{\mathcal{K}}$. Therefore,
  $a_h(\boldsymbol{u}_h,\boldsymbol{v}_h)=a_{\mathfrak{h}} \del{
    \gamma_\mathfrak{h}\del{\boldsymbol{u}_h},
    \gamma_\mathfrak{h}\del{\boldsymbol{v}_h}}$ for any
  $\boldsymbol{v}_h\in \boldsymbol{V}_h$ and hence \cref{eq:glk_ort}.
\end{proof}

Following \cite[Section 4, Remark 2]{Burman:2009}, we assume that the
following saturation assumption holds uniformly on the family of
meshes $\cbr{\mathcal{T}_h}_h$: There exists $\rho < 1$, independent
of $h_K$, $\delta t_{\mathcal{K}}$, and $\varepsilon$, such that:
\begin{equation}
  \label{eq:SA}
  \sum_{\mathcal{K}\in\mathcal{T}_h}
  \tau_\varepsilon
  \norm{\partial_t (u-u_\mathfrak{h})}_{\mathcal{K}}^2
  \leq
  \rho^2
  \sum_{\mathcal{K}\in\mathcal{T}_h}
  \tau_\varepsilon
  \norm{\partial_t (u-u_h)}_{\mathcal{K}}^2.
\end{equation}
With this saturation assumption we prove the following useful theorem.

\begin{theorem}[Time derivative estimation]
  \label{thm:st_time_apos}
  Let $u$ be the solution to \cref{eq:advdif} and let
  $\boldsymbol{u}_h=(u_h,\lambda_h)$ be the solution to
  \cref{eq:st_hdg_adr_compact}. If the saturation assumption
  \cref{eq:SA} holds, and if
  $\delta t_{\mathcal{K}} = \mathcal{O}(h_K^2)$, then
  \begin{equation}
    \label{eq:timederivativeest}
    \sum_{\mathcal{K}\in\mathcal{T}_h}
    \tau_\varepsilon
    \norm{\partial_t\del{u-u_h}}_{\mathcal{K}}^2
    \leq
    cT^2\varepsilon^{-1}
    \eta^2.
  \end{equation}
\end{theorem}
\begin{proof}
  By the triangle inequality and \cref{eq:SA} we find
  \begin{equation}
    \label{eq:tauepsdtuuh}
    \del[2]{
      \sum_{\mathcal{K}\in\mathcal{T}_h}
      \tau_{\varepsilon}
      \norm[0]{\partial_t\del[0]{u-u_h}}_{\mathcal{K}}^2
    }^{{1}/{2}}
    \le
    \frac{1}{1-\rho}
    \del[2]{
      \sum_{\mathcal{K}\in\mathcal{T}_h}
      \tau_{\varepsilon}
      \norm[0]{\partial_t\del[0]{u_{\mathfrak{h}}-u_h}}_{\mathcal{K}}^2
    }^{1/2}.
  \end{equation}
  By the inf-sup condition \cref{eq:inf_sup_s_norm}, we have
  \begin{equation}
    \label{eq:pretimederivativeest}
    \del[2]{
      \sum_{\mathcal{K}\in\mathcal{T}_h}
      \tau_{\varepsilon}
      \norm[0]{\partial_t\del[0]{u_{\mathfrak{h}}-u_h}}_{\mathcal{K}}^2
    }^{{1}/{2}}
    \le c
    \tnorm{\boldsymbol{u}_{\mathfrak{h}}-\gamma_{\mathfrak{h}}(\boldsymbol{u}_h)}_{s,\mathfrak{h}}
    \le c_T
    \sup_{\boldsymbol{v}_{\mathfrak{h}}\in \boldsymbol{V}_{\mathfrak{h}}}
    \frac{
      a_{\mathfrak{h}}\del{\boldsymbol{u}_{\mathfrak{h}}-\gamma_{\mathfrak{h}}(\boldsymbol{u}_h), \boldsymbol{v}_{\mathfrak{h}}}
    }{\tnorm{\boldsymbol{v}_{\mathfrak{h}}}_{s,\mathfrak{h}}}.
  \end{equation}
  Using Galerkin orthogonality \cref{eq:glk_ort}, that $\Omega_T$
  consists of $\mathcal{R}$-facets only and that
  $\del[0]{\mathrm{I}-i^\mathcal{F}_h}{\mu_\mathfrak{h}}$ vanishes on
  $\mathcal{R}$-facets, we have
  \begin{equation*}
    \begin{split}
      a_\mathfrak{h}
      (\boldsymbol{u}_\mathfrak{h} - \gamma_{\mathfrak{h}}(\boldsymbol{u}_h),
      \boldsymbol{v}_\mathfrak{h})
      &=
      a_\mathfrak{h}
      (\boldsymbol{u}_\mathfrak{h} -
      \gamma_{\mathfrak{h}}(\boldsymbol{u}_h) ,
      \del{\mathrm{I}-i_h}\boldsymbol{v}_\mathfrak{h})
      \\
      &=
      \del[0]{ f ,
        \del[0]{\mathrm{I}-i^\mathcal{K}_h}{v_\mathfrak{h}}
      }_{\mathcal{T}_\mathfrak{h}}
      +
      \langle{ g ,
        \del[0]{\mathrm{I}-i^\mathcal{F}_h}{\mu_\mathfrak{h}}
      }\rangle_{\partial\mathcal{E}_N\cap\mathcal{Q}_\mathfrak{h}}
      -
      a_\mathfrak{h} ( \gamma_{\mathfrak{h}}(\boldsymbol{u}_h) ,
      \del[0]{\mathrm{I}-i_h}\boldsymbol{v}_\mathfrak{h}).
    \end{split}
  \end{equation*}
  Using integration by parts on
  $\del[0]{\varepsilon\overline{\nabla}u_{\mathfrak{h}},
    \overline{\nabla}v_{\mathfrak{h}}}_{\mathcal{T}_{\mathfrak{h}}}$
  and
  $ \del{{\beta}u_{\mathfrak{h}}, \nabla
    v_{\mathfrak{h}}}_{\mathcal{T}_{\mathfrak{h}}} $, using the
  definition of the residual $R_{\mathfrak{h}}^{\mathcal{K}}$, and
  applying the Dirichlet and the Neumann boundary conditions, we have
  \begin{align*}
      a_{\mathfrak{h}}
      (\boldsymbol{u}_{\mathfrak{h}},\boldsymbol{v}_{\mathfrak{h}})
    &=
      \del[1]{f-R_{\mathfrak{h}}^{\mathcal{K}},v_{\mathfrak{h}}}_{\mathcal{T}_{\mathfrak{h}}}
      +
      \langle{\varepsilon\overline{\nabla}_{\overline{{n}}}u_{\mathfrak{h}},\mu_{\mathfrak{h}}}\rangle_{\mathcal{Q}_{\mathfrak{h}}\setminus\partial\mathcal{E}}
      -
      \langle{
      \beta\cdot nu_{\mathfrak{h}}
      ,
      \mu_{\mathfrak{h}}
      }\rangle_{\partial\mathcal{T}_{\mathfrak{h}}\setminus\partial\mathcal{E}}
    \\
    &\quad
      +
      \langle{
      \varepsilon\overline{\nabla}_{\overline{{n}}}u_{\mathfrak{h}}-\zeta^-\beta\cdot nu_\mathfrak{h},
      \mu_{\mathfrak{h}}
      }\rangle_{\partial\mathcal{E}_N}
      -
      \langle{
      \zeta^+
      \beta\cdot n\sbr[0]{\boldsymbol{u}_{\mathfrak{h}}}
      ,
      \mu_{\mathfrak{h}}
      }\rangle_{\partial\mathcal{E}_N}
    \\
    &\quad
      -
      \langle{\varepsilon\sbr[0]{\boldsymbol{u}_{\mathfrak{h}}},\overline{\nabla}_{\overline{{n}}}v_{\mathfrak{h}}
      }\rangle_{\mathcal{Q}_{\mathfrak{h}}}
      +
      \langle{{\varepsilon\alpha}{h_K^{-1}}\sbr[0]{\boldsymbol{u}_{\mathfrak{h}}},\sbr[0]{\boldsymbol{v}_{\mathfrak{h}}}}\rangle_{\mathcal{Q}_{\mathfrak{h}}}
      +
      \langle{
      \del[0]{\beta_s-\beta\cdot n}
      \sbr[0]{\boldsymbol{u}_{\mathfrak{h}}}
      ,
      \sbr[0]{\boldsymbol{v}_{\mathfrak{h}}}
      }
      \rangle_{\partial\mathcal{T}_{\mathfrak{h}}}.
  \end{align*}
  Replacing
  $(\boldsymbol{u}_{\mathfrak{h}},\boldsymbol{v}_{\mathfrak{h}})$ in
  the above by
  $(\gamma_{\mathfrak{h}}(\boldsymbol{u}_h),
  (I-i_h)\boldsymbol{v}_{\mathfrak{h}})$, we find
  \begin{equation}
    \label{eq:amathfrakhM1toM6}
    \begin{split}
      a_\mathfrak{h}
      (\boldsymbol{u}_\mathfrak{h} - \gamma_{\mathfrak{h}}(\boldsymbol{u}_h),
      \boldsymbol{v}_\mathfrak{h})
      &=
      \del[1]{R_{\mathfrak{h}}^{\mathcal{K}},(I-i_h^{\mathcal{K}})v_{\mathfrak{h}}}_{\mathcal{T}_{\mathfrak{h}}}
      \\ &\quad
      +\sbr[1]{
        -\langle{\varepsilon\overline{\nabla}_{\overline{{n}}}u_h, (I-i_h^{\mathcal{F}})\mu_{\mathfrak{h}}}\rangle_{\mathcal{Q}_{\mathfrak{h}}\setminus\partial\mathcal{E}}
        +\langle{\beta\cdot n u_h,
          (I - i_h^{\mathcal{F}})\mu_{\mathfrak{h}}
        }\rangle_{\partial\mathcal{T}_{\mathfrak{h}}\setminus\partial\mathcal{E}}}
      \\ &\quad
      + \langle{\varepsilon\sbr[0]{\gamma_{\mathfrak{h}}(\boldsymbol{u}_h)},\overline{\nabla}_{\overline{{n}}} ((I-i_h^{\mathcal{K}})v_{\mathfrak{h}})
      }\rangle_{\mathcal{Q}_{\mathfrak{h}}}
      \\ &\quad
      - \langle{{\varepsilon\alpha}{h_K^{-1}}\sbr[0]{\gamma_{\mathfrak{h}}(\boldsymbol{u}_h)}, \sbr[0]{(I-i_h)\boldsymbol{v}_{\mathfrak{h}}}  }\rangle_{\mathcal{Q}_{\mathfrak{h}}}
      \\ &\quad
      +\sbr[1]{
        - \langle{
          \del[0]{\beta_s-\beta\cdot n}
          \sbr[0]{\gamma_{\mathfrak{h}}(\boldsymbol{u}_h)}
          ,
          \sbr[0]{(I-i_h)\boldsymbol{v}_{\mathfrak{h}}}
        }
        \rangle_{\partial\mathcal{T}_{\mathfrak{h}}}
        + \langle{
          \zeta^+
          \beta\cdot n\sbr[0]{\gamma_{\mathfrak{h}}(\boldsymbol{u}_h)}
          ,
          (I-i_h^{\mathcal{F}})\mu_{\mathfrak{h}}
        }\rangle_{\partial\mathcal{E}_N}}
      \\ &\quad
      + \langle{
        g-\varepsilon\overline{\nabla}_{\overline{n}}u_h + \zeta^-(\beta\cdot n) u_h, (I-i_h^{\mathcal{F}})\mu_{\mathfrak{h}}
      }\rangle_{\mathcal{Q}_{\mathfrak{h}}\cap\partial\mathcal{E}_N}
      \\
      &=: M_1 + M_2 + M_3 + M_4 + M_5 + M_6.
    \end{split}
  \end{equation}
  We will bound the $M_i$'s separately.
  \\
  \textbf{Bound for $M_1$.} $M_1$ is bounded using the
  Cauchy--Schwarz inequality, \cref{eq:subgrid_proj_est_1},
  \cref{eq:subgridhelperbnds_1,eq:subgridhelperbnds_2}, and that
  $\delta t_{\mathcal{K}}=\mathcal{O}(h_K^2)$:
  \begin{equation*}
    \begin{split}
      &\del[0]{R_h^{\mathcal{K}},\del[0]{I-i_h^\mathcal{K}}v_\mathfrak{h}}_{\mathcal{T}_\mathfrak{h}}
      =\del[0]{R_h^{\mathcal{K}},\del[0]{I-i_h^\mathcal{K}}v_\mathfrak{h}}_{\mathcal{T}_h}
      \\
      \le& c
      \sum_{\mathcal{K}\in\mathcal{T}_h}
      \lambda_\mathcal{K}\norm[0]{R_h^{\mathcal{K}}}_\mathcal{K}
      \max\cbr[0]{h_K^{-1}\varepsilon^{{1}/{2}},1}
      \norm[0]{\del[0]{I-i_h^\mathcal{K}}v_\mathfrak{h}}_\mathcal{K}
      \\
      \le&c
      \sum_{\mathcal{K}\in\mathcal{T}_h}
      \lambda_\mathcal{K}\norm[0]{R_h^{\mathcal{K}}}_\mathcal{K}
      \max\cbr[0]{\varepsilon^{{1}/{2}},h_K}
      \del[2]{
        \sum_{\mathring{\mathcal{K}}\in\mathcal{T}_{\mathcal{K}}}
        \norm[0]{
          \envert[0]{\beta_s-\tfrac{1}{2}\beta\cdot n}^{1/2}
          \sbr[0]{\boldsymbol{v}_{\mathfrak{h}}}
        }_{ \partial\mathring{\mathcal{K}}
          \cap
          F_{\mathring{\mathcal{R}}}
        }
        +
        \sum_{\mathring{\mathcal{K}}\in\mathcal{T}_{\mathcal{K}}}
        {\delta t_{\mathcal{K}}}h_K^{-1}
        \norm{\partial_tv_\mathfrak{h}}_{\mathring{\mathcal{K}}}
      }.
    \end{split}
  \end{equation*}
  On elements where
  $\max\cbr[0]{\varepsilon^{1/2},h_K} = \varepsilon^{1/2}$, using that
  $\delta t_{\mathcal{K}} = \mathcal{O}(h_K^2)$, we find
  $\varepsilon^{1/2}{\delta
    t_{\mathcal{K}}}h_K^{-1}\norm[0]{\partial_tv_\mathfrak{h}}_{\mathring{\mathcal{K}}}
  \le c
  \tau_{\varepsilon}^{1/2}\norm[0]{\partial_tv_\mathfrak{h}}_{\mathring{\mathcal{K}}}$. On
  elements where $\max\cbr[0]{\varepsilon^{1/2},h_K} = h_K$ we have,
  by \cref{eq:eg_inv_1},
  $h_K{\delta
    t_{\mathcal{K}}}h_K^{-1}\norm[0]{\partial_tv_\mathfrak{h}}_{\mathring{\mathcal{K}}}\le
  c\norm[0]{v_\mathfrak{h}}_{\mathring{\mathcal{K}}}$. Therefore,
  \begin{equation*}
    \max\cbr[0]{\varepsilon^{1/2},h_K}
    \sum_{\mathring{\mathcal{K}}\in\mathcal{T}_{\mathcal{K}}}
    {\delta t_{\mathcal{K}}}h_K^{-1}
    \norm{\partial_tv_\mathfrak{h}}_{\mathring{\mathcal{K}}}
    \le \sum_{\mathring{\mathcal{K}}\in\mathcal{T}_{\mathcal{K}}}\tau_{\varepsilon}^{1/2}\norm[0]{\partial_tv_\mathfrak{h}}_{\mathring{\mathcal{K}}}
    + \sum_{\mathring{\mathcal{K}}\in\mathcal{T}_{\mathcal{K}}}\norm[0]{v_\mathfrak{h}}_{\mathring{\mathcal{K}}}.
  \end{equation*}
  Using furthermore that $\max\cbr[0]{\varepsilon^{1/2},h_K} \le 1$
  and H\"older's inequality for sums, we find
  \begin{equation*}
    \begin{split}
      M_1
      \le &
      c
      \del[1]{
        \sum_{\mathcal{K}\in\mathcal{T}_h}
        \del[0]{\eta_{R}^{\mathcal{K}}}^2
      }^{1/2}
      \sbr[3]{ \sum_{\mathcal{K} \in \mathcal{T}_h} \del[2]{
          \sum_{\mathring{\mathcal{K}}\in\mathcal{T}_{\mathcal{K}}}\del[1]{
          \norm[0]{
            \envert[0]{\beta_s-\tfrac{1}{2}\beta\cdot n}^{1/2}
            \sbr[0]{\boldsymbol{v}_{\mathfrak{h}}}
          }_{ \mathcal{R}_{\mathring{\mathcal{K}}}
          }
          +
          \tau_{\varepsilon}^{1/2}
          \norm[0]{\partial_tv_\mathfrak{h}}_{\mathring{\mathcal{K}}}
          +
          \norm[0]{v_\mathfrak{h}}_{\mathring{\mathcal{K}}}
          }
        }^2
      }^{1/2}
      \\
      \le& c
      \del[1]{
        \sum_{\mathcal{K}\in\mathcal{T}_h}
        \del[0]{\eta_{R}^{\mathcal{K}}}^2
      }^{1/2}
      \tnorm{\boldsymbol{v}_{\mathfrak{h}}}_{s,\mathfrak{h}}.
    \end{split}
  \end{equation*}
  \textbf{Bound for $M_2$.} We first write
  \begin{equation*}
    M_2 = \underbrace{-\langle{\varepsilon\overline{\nabla}_{\overline{{n}}}u_h,
        (I-i_h^{\mathcal{F}})\mu_{\mathfrak{h}}}\rangle_{\mathcal{Q}_{\mathfrak{h}}\setminus\partial\mathcal{E}}}_{M_{21}}
    +\underbrace{\langle{\beta\cdot n u_h,
        (I - i_h^{\mathcal{F}})\mu_{\mathfrak{h}}
      }\rangle_{\partial\mathcal{T}_{\mathfrak{h}}\setminus\partial\mathcal{E}}}_{M_{22}}.
  \end{equation*}
  For $M_{21}$, using that
  $\langle \cdot, \cdot \rangle_{\mathcal{Q}_{\mathfrak{h}}\setminus
    \partial \mathcal{E}} = \langle \cdot, \cdot
  \rangle_{\mathcal{Q}_h\setminus\partial\mathcal{E}}$, writing
  element-wise integrals as facet integrals on interior facets, using
  the Cauchy--Schwarz inequality and the projection estimate
  \cref{eq:subgrid_proj_est_2}, we find
  \begin{equation}
    \label{eq:some_eq_3}
    \begin{split}
      M_{21}
      &\leq
      c
      \sum_{F_{\mathcal{Q}}\in\mathcal{F}^i_{\mathcal{Q},{h}}}
      \envert[1]{
        \langle{
          \jump{\varepsilon\overline{\nabla}_{\overline{{n}}}u_h}
          ,
          \del[0]{I-i_h^\mathcal{F}}\mu_\mathfrak{h}
        }\rangle_{F_{\mathcal{Q}}}
      }
      \\
      & \leq c
      \sum_{F_{\mathcal{Q}}\in\mathcal{F}^i_{\mathcal{Q},{h}}}
      \delta t_{\mathcal{K}}^{{1}/{2}}
      \norm[0]{
        \jump{\varepsilon\overline{\nabla}_{\overline{{n}}}u_h}
      }_{F_{\mathcal{Q}}}
      \del[1]{
        \norm[0]{\ejump{\mu_\mathfrak{h}}}_{E_{\mathring{\mathcal{K}}}}
        +
        \delta t_{\mathcal{K}}
        \norm[0]{\ejump{\partial_t\mu_\mathfrak{h}}}_{E_{\mathring{\mathcal{K}}}}
      },
    \end{split}
  \end{equation}
  where $\mathcal{K}$ in the last step is chosen such that
  $F_{\mathcal{Q}}\subset\mathcal{Q}_{\mathcal{K}}$. Consider the two
  terms on the right-hand side of \cref{eq:some_eq_3}
  separately. First, using \cref{eq:subgridhelperbnds_4} and
  $\delta t_{\mathcal{K}} = \mathcal{O}(h_K^2)$, we have
  \begin{equation*}
    \begin{split}
      &
      \sum_{F_{\mathcal{Q}}\in\mathcal{F}^i_{\mathcal{Q},{h}}}
      \norm[0]{
        \jump{\varepsilon\overline{\nabla}_{\overline{{n}}}u_h}
      }_{F_{\mathcal{Q}}}
      \delta t_{\mathcal{K}}^{{1}/{2}}
      \norm[0]{\ejump{\mu_\mathfrak{h}}}_{E_{\mathring{\mathcal{K}}}}
      \\
      \le
      &
      c\sum_{F_{\mathcal{Q}}\in\mathcal{F}^i_{\mathcal{Q},{h}}}
      h_K^{1/2}\varepsilon^{-1/2}
      \norm[0]{
        \jump{\varepsilon\overline{\nabla}_{\overline{{n}}}u_h}
      }_{F_{\mathcal{Q}}}
      \\
      &\quad
      \cdot
      \del[2]{
        \delta t_{\mathcal{K}}^{{1}/{2}}
        h_K^{-1}
        \sum_{\mathring{\mathcal{K}}\in\mathcal{T}_{\mathcal{K}}}
        \norm[0]{
          \envert[0]{\beta_s-\tfrac{1}{2}\beta\cdot n}^{1/2}
          \sbr[0]{\boldsymbol{v}_{\mathfrak{h}}}
        }_{ \partial\mathring{\mathcal{K}}
          \cap
          F_{\mathring{\mathcal{R}}}
        }
        +
        \varepsilon^{1/2}h_K^{-1/2}\del[1]{
        \norm[0]{\sbr[0]{\boldsymbol{v}_{\mathfrak{h}}^*}}_{F_{\mathcal{Q}}^*}
        +
        \norm[0]{\sbr[0]{\boldsymbol{v}_{\mathfrak{h},*}}}_{F_{\mathcal{Q},*}}
        }
      }
      \\
      \le&c
      \del[1]{
        \sum_{\mathcal{K}\in\mathcal{T}_h}
        \del[0]{\eta_{J,1}^{\mathcal{K}}}^2
      }^{1/2}
      \tnorm{\boldsymbol{v}_{\mathfrak{h}}}_{s,\mathfrak{h}}.
    \end{split}
  \end{equation*}
  The second term on the right-hand side of \cref{eq:some_eq_3} can be
  bounded similarly by using \cref{eq:subgridhelperbnds_5} and
  $\delta t_{\mathcal{K}}h_K^{-1}\varepsilon^{1/2}\le
  c\tau_{\varepsilon}^{1/2}$: we have
  $\sum_{F_{\mathcal{Q}}\in\mathcal{F}^i_{\mathcal{Q},{h}}} \norm[0]{
    \jump{\varepsilon\overline{\nabla}_{\overline{{n}}}u_h}
  }_{F_{\mathcal{Q}}} \delta t_{\mathcal{K}}^{{3}/{2}}
  \norm[0]{\ejump{\partial_t\mu_\mathfrak{h}}}_{E_{\mathring{\mathcal{K}}}}
  \le c \del[1]{ \sum_{\mathcal{K}\in\mathcal{T}_h}
    \del[0]{\eta_{J,1}^{\mathcal{K}}}^2 }^{1/2}
  \tnorm{\boldsymbol{v}_{\mathfrak{h}}}_{s,\mathfrak{h}}$ and
  therefore
  $M_{21} \le c \del[1]{ \sum_{\mathcal{K}\in\mathcal{T}_h}
    \del[0]{\eta_{J,1}^{\mathcal{K}}}^2 }^{1/2}
  \tnorm{\boldsymbol{v}_{\mathfrak{h}}}_{s,\mathfrak{h}}$.

  For $M_{22}$, we first note that since
  $\del[0]{I-i_h^{\mathcal{F}}}\mu_{\mathfrak{h}}$ vanishes on
  $\mathcal{R}_{\mathfrak{h}}$ we have that
  $M_{22} = \langle{\beta\cdot n u_h, (I -
    i_h^{\mathcal{F}})\mu_{\mathfrak{h}}
  }\rangle_{\partial\mathcal{T}_{\mathfrak{h}}\setminus\partial\mathcal{E}}
  = \langle{\beta\cdot n u_h, (I -
    i_h^{\mathcal{F}})\mu_{\mathfrak{h}}
  }\rangle_{\mathcal{Q}_{\mathfrak{h}}\setminus\partial\mathcal{E}}$. Then,
  similar to \cref{eq:some_eq_3}, we have using
  \cref{eq:subgrid_proj_est_2,eq:subgridhelperbnds_4,eq:subgridhelperbnds_5},
  that $\delta t_{\mathcal{K}} = \mathcal{O}(h_K^2)$,
  $\varepsilon \le 1$, $h_K \le 1$, that
  $\norm[0]{\bar{\beta}}_{L^{\infty}(\mathcal{E})}\le 1$, and noting
  that
  $\langle{ \beta\cdot n\lambda_h,
    \del[0]{I-i_h^\mathcal{F}}\mu_\mathfrak{h}
  }\rangle_{\mathcal{Q}_{h}\setminus\partial\mathcal{E}} = 0$ by
  single-valuedness of $\lambda_h$, $\beta\cdot n$, and
  $\del[0]{I-i_h^\mathcal{F}}\mu_\mathfrak{h}$ on element boundaries,
  \begin{equation*}
    \begin{split}
      & M_{22}
      =
      \langle{
        \beta\cdot n \sbr[0]{\boldsymbol{u}_h},
        \del[0]{I-i_h^\mathcal{F}}\mu_\mathfrak{h}
      }\rangle_{\mathcal{Q}_{h}\setminus\partial\mathcal{E}}
      \\
      \leq&c
      \sum_{\mathcal{K} \in \mathcal{T}_h}
      \delta t_{\mathcal{K}}^{{1}/{2}}
      \norm[0]{
        \sbr[0]{\boldsymbol{u}_h}
      }_{\mathcal{Q}_{\mathcal{K}}}
      \del[1]{
        \norm[0]{\ejump{\mu_\mathfrak{h}}}_{E_{\mathring{\mathcal{K}}}}
        +
        \delta t_{\mathcal{K}}
        \norm[0]{\ejump{\partial_t\mu_\mathfrak{h}}}_{E_{\mathring{\mathcal{K}}}}
      }
      \\
      \leq&c
      \sum_{\mathcal{K} \in \mathcal{T}_h}
      h_K^{1/2}\varepsilon^{-1/2}
      \norm[0]{
        \sbr[0]{\boldsymbol{u}_h}
      }_{\mathcal{Q}_{\mathcal{K}}}
      \\
      &\quad
      \cdot
      \del[1]{
        \sum_{\mathring{\mathcal{K}}\in\mathcal{T}_{\mathcal{K}}}
        \norm[0]{
          \envert[0]{\beta_s-\tfrac{1}{2}\beta\cdot n}^{1/2}
          \sbr[0]{\boldsymbol{v}_{\mathfrak{h}}}
        }_{ \partial\mathring{\mathcal{K}}
          \cap
          F_{\mathring{\mathcal{R}}}
        }
        +
        \sum_{\mathring{\mathcal{K}}\in\mathcal{T}_{\mathcal{K}}}
        \tau_\varepsilon^{1/2}
        \norm[0]{\partial_tv_{\mathfrak{h}}}_{\mathring{\mathcal{K}}}
        +
        \varepsilon^{1/2}h_K^{-1/2} \del[1]{
        \norm[0]{\sbr[0]{\boldsymbol{v}_{\mathfrak{h}}^*}}_{F_{\mathcal{Q}}^*}
        +
        \norm[0]{\sbr[0]{\boldsymbol{v}_{\mathfrak{h},*}}}_{F_{\mathcal{Q},*}}
        }
      }
      \\
      \leq&c
      \del[1]{
        \sum_{\mathcal{K}\in\mathcal{T}_h}
        \del[0]{\eta_{J,2,2}^{\mathcal{K}}}^2
      }^{1/2}
      \tnorm{\boldsymbol{v}_{\mathfrak{h}}}_{s,\mathfrak{h}}.
    \end{split}
  \end{equation*}
  Combining the bounds for $M_{21}$ and $M_{22}$ we obtain:
  \begin{equation*}
    M_2 \le
    c \sbr[2]{\del[1]{
        \sum_{\mathcal{K}\in\mathcal{T}_h}
        \del[0]{\eta_{J,1}^{\mathcal{K}}}^2
      }^{1/2}
      +
      \del[1]{
        \sum_{\mathcal{K}\in\mathcal{T}_h}
        \del[0]{\eta_{J,2,2}^{\mathcal{K}}}^2
      }^{1/2}
    }
    \tnorm{\boldsymbol{v}_{\mathfrak{h}}}_{s,\mathfrak{h}}.
  \end{equation*}
  \textbf{Bound for $M_3$.} For $M_3$, using the Cauchy--Schwarz
  inequality, the trace inequality \cref{eq:eg_inv_3}, the inverse
  inequality \cref{eq:eg_inv_2}, the subgrid projection estimate
  \cref{eq:subgrid_proj_est_1},
  \cref{eq:subgridhelperbnds_1,eq:subgridhelperbnds_2}, that
  $\delta t_{\mathcal{K}} = \mathcal{O}(h_K^2)$, and H\"older's
  inequality for sums,
  \begin{equation*}
    \begin{split}
      M_3
      \le&
      c\sum_{\mathcal{K} \in \mathcal{T}_h}
      \varepsilon
      {h_K^{-{1}/{2}}}\norm[0]{\sbr[0]{\boldsymbol{u}_h}}_{\mathcal{Q}_{\mathcal{K}}}
      h_K^{{1}/{2}}
      \norm[0]{
        \overline{\nabla}
        \del[0]{
          \del[0]{I-i_h^\mathcal{K}}v_\mathfrak{h}
        }
      }_{\mathcal{Q}_{\mathcal{K}}}
      \\
      \le & c
      \sum_{\mathcal{K} \in \mathcal{T}_h}
      \varepsilon^{1/2}
      {h_K^{-{1}/{2}}}
      \norm[0]{\sbr[0]{\boldsymbol{u}_h}}_{\mathcal{Q}_{\mathcal{K}}}
      h_K^{-1}
      \delta t_{\mathcal{K}}^{1/2}
      \del[2]{
        \sum_{\mathring{\mathcal{K}}\in\mathcal{T}_{\mathcal{K}}}
        \norm[0]{
          \envert[0]{\beta_s-\tfrac{1}{2}\beta\cdot n}^{1/2}
          \sbr[0]{\boldsymbol{v}_{\mathfrak{h}}}
        }_{ \partial\mathring{\mathcal{K}}
          \cap
          F_{\mathring{\mathcal{R}}}
        }
        +
        \sum_{\mathring{\mathcal{K}}\in\mathcal{T}_{\mathcal{K}}}
        \tau_{\varepsilon}^{1/2}
        \norm{
          \partial_tv_\mathfrak{h}
        }_{\mathring{\mathcal{K}}}
      }
      \\
      \le & c
      \del[1]{
        \sum_{\mathcal{K}\in\mathcal{T}_h}
        \del[0]{\eta_{J,2,1}^{\mathcal{K}}}^2
      }^{1/2}
      \tnorm{\boldsymbol{v}_{\mathfrak{h}}}_{s,\mathfrak{h}}.
    \end{split}
  \end{equation*}
  \textbf{Bound for $M_4$.} Let
  $\mathring{\mathcal{K}}\in\mathcal{T}_{\mathcal{K}}$ and
  $F_{\mathcal{Q}}\subset\mathcal{Q}_{\mathring{\mathcal{K}}}$. We
  write $M_4 := M_{41} + M_{42}$ where $M_{41}$ is the sum of
  integrals over
  $F_{\mathcal{Q}} \in \mathcal{F}_{\mathcal{Q},\mathfrak{h}}$ and
  $M_{42}$ the sum of integrals over
  $F_{\mathcal{Q}} \notin \mathcal{F}_{\mathcal{Q},\mathfrak{h}}$. The
  latter case occurs when the neighboring element of
  $\mathring{\mathcal{K}}$ over $F_{\mathcal{Q}}$ is coarser than
  $\mathring{\mathcal{K}}$. To bound $M_{41}$, we first note that for
  $F_{\mathcal{Q}}\in\mathcal{F}_{\mathcal{Q},\mathfrak{h}}$, we have
  \begin{equation}
    \label{eq:qfacediffprojrewrite}
      \sbr[0]{
        \del[0]{I-i_h}
        \boldsymbol{v}_\mathfrak{h}
      }
      =
      \del[0]{I-i_h^\mathcal{K}}v_\mathfrak{h}
      -
      \del[0]{I-i_h^\mathcal{F}}\mu_\mathfrak{h}
      =
      \del[0]{I-i_h^\mathcal{F}}
      \del[0]{
        v_\mathfrak{h}
        -
        \mu_\mathfrak{h}
      },
  \end{equation}
  where the last step is by \cref{lem:subgridprojdiff}. Then,
  note that by the Cauchy--Schwarz inequality and boundedness of the
  projection $i_h^{\mathcal{F}}$, we have
  \begin{equation}
    \label{eq:mterms5}
    \begin{split}
      \langle{
        {\varepsilon\alpha}{h_K^{-1}}
        \sbr[0]{\gamma_{\mathfrak{h}}(\boldsymbol{u}_h)}
        ,
        \sbr[0]{
          \del[0]{I-i_h}
          \boldsymbol{v}_\mathfrak{h}
        }
      }\rangle_{F_{\mathcal{Q}}}
      \le
      c\del[1]{{\varepsilon^{1/2}}{h_K^{-1/2}}\norm[0]{\sbr[0]{\boldsymbol{u}_h}}_{\mathcal{Q}_{\mathcal{K}}}}
      \del[1]{
        {\varepsilon^{1/2}}{h_K^{-1/2}}\norm[0]{\sbr[0]{\boldsymbol{v}_\mathfrak{h}}}_{\mathcal{Q}_{\mathcal{K}}}
      },
    \end{split}
  \end{equation}
  so that
  $M_{41} \le c \del[2]{\sum_{\mathcal{K} \in \mathcal{T}_h}
    \del[1]{\eta_{J,2,1}^{\mathcal{K}}}^2}^{1/2}\tnorm{\boldsymbol{v}_{\mathfrak{h}}}_{s,\mathfrak{h}}$.

  We now consider $M_{42}$. Consider an
  $F_{\mathcal{Q}}\notin
  \mathcal{F}_{\mathcal{Q},\mathfrak{h}}$. Denote the coarser
  neighboring element of $\mathring{\mathcal{K}}$ over
  $F_{\mathcal{Q}}$ by $\mathring{\mathcal{K}}_{nb}$ and denote the
  restriction of $v_{\mathfrak{h}}$ to $\mathring{\mathcal{K}}_{nb}$
  by $v_{nb,\mathfrak{h}}$. We have
  \begin{equation}
    \label{eq:qfacediffprojrewrite-2}
    \begin{split}
      \sbr[0]{
        \del[0]{I-i_h}
        \boldsymbol{v}_\mathfrak{h}
      }
      &=
      \del[0]{I-i_h^\mathcal{K}}v_\mathfrak{h}
      + v_{nb,\mathfrak{h}}-\mu_{\mathfrak{h}}
      + i_h^\mathcal{F}\mu_{\mathfrak{h}}
      - v_{nb,\mathfrak{h}}
      \\
      &=
      \del[0]{I-i_h^\mathcal{K}}v_\mathfrak{h}
      +
      \del[0]{I-i_h^\mathcal{F}}\del[0]{v_{nb,\mathfrak{h}}-\mu_{\mathfrak{h}}}
      -
      \del[0]{I-i_h^\mathcal{K}}v_{nb,\mathfrak{h}}
      ,
    \end{split}
  \end{equation}
  where the last step is by \cref{lem:subgridprojdiff}. We have:
  \begin{multline}
    \label{eq:M42splitting3terms}
      \langle{
        {\varepsilon\alpha}{h_K^{-1}}
        \sbr[0]{\gamma_{\mathfrak{h}}(\boldsymbol{u}_h)}
        ,
        \sbr[0]{
          \del[0]{I-i_h}
          \boldsymbol{v}_\mathfrak{h}
        }
      }\rangle_{F_{\mathcal{Q}}}
      =
      \langle{
        {\varepsilon\alpha}{h_K^{-1}}
        \sbr[0]{\boldsymbol{u}_h}
        ,
        \del[0]{I-i_h^{\mathcal{K}}}
        v_{\mathfrak{h}}
      }\rangle_{F_{\mathcal{Q}}}
      +
      \langle{
        {\varepsilon\alpha}{h_K^{-1}}
        \sbr[0]{\boldsymbol{u}_h}
        ,
        \del[0]{I-i_h^{\mathcal{F}}}
        \sbr[0]{\boldsymbol{v}_{nb,\mathfrak{h}}}
      }\rangle_{F_{\mathcal{Q}}}
      \\
      -
      \langle{
        {\varepsilon\alpha}{h_K^{-1}}
        \sbr[0]{\boldsymbol{u}_h}
        ,
        \del[0]{I-i_h^{\mathcal{K}}}
        v_{nb,\mathfrak{h}}
      }\rangle_{F_{\mathcal{Q}}}.
  \end{multline}
  The second term on the right-hand side of
  \cref{eq:M42splitting3terms} is bounded in the same way as
  \cref{eq:mterms5}. For the first term on the right-hand side of
  \cref{eq:M42splitting3terms}, using the Cauchy--Schwarz inequality,
  the trace inequality \cref{eq:eg_inv_3}, the subgrid projection
  bound \cref{eq:subgrid_proj_est_1}, that
  $\delta t_{\mathcal{K}} = \mathcal{O}(h_K^2)$, and
  \cref{eq:subgridhelperbnds_1,eq:subgridhelperbnds_2}, we find
  \begin{multline*}
      \langle{
        {\varepsilon\alpha}{h_K^{-1}}
        \sbr[0]{\boldsymbol{u}_h},
        \del[0]{I-i_h^{\mathcal{K}}}{v}_\mathfrak{h}
      }\rangle_{F_{\mathcal{Q}}}
      \\
      \le
      c\del[1]{{\varepsilon^{1/2}}{h_K^{-1/2}}
        \norm[0]{\sbr[0]{\boldsymbol{u}_h}}_{\mathcal{Q}_{\mathcal{K}}}}
      \del[2]{
        \sum_{\mathring{\mathcal{K}}\in\mathcal{T}_{\mathcal{K}}}
        \del[1]{
          \norm[0]{
            \envert[0]{\beta_s-\tfrac{1}{2}\beta\cdot n}^{1/2}
            \sbr[0]{\boldsymbol{v}_{\mathfrak{h}}}
          }_{ \partial\mathring{\mathcal{K}}
            \cap
            F_{\mathring{\mathcal{R}}}
          }
          +
          \tau_{\varepsilon}^{1/2}
          \norm{\partial_tv_\mathfrak{h}}_{\mathring{\mathcal{K}}}
        }
      }.
  \end{multline*}
  The third term on the right-hand side of
  \cref{eq:M42splitting3terms} is bound in the same way. For $M_{42}$
  we therefore find that
  $M_{42} \le c \del[2]{\sum_{\mathcal{K} \in
      \mathcal{T}_h}\del[1]{\eta_{J,2,1}^K}^2}^{1/2}\tnorm{\boldsymbol{v}_{\mathfrak{h}}}_{s,\mathfrak{h}}$.

  Combining the bounds for $M_{41}$ and $M_{42}$,
  \begin{equation*}
    M_4
    \le c
    \del[1]{
      \sum_{\mathcal{K}\in\mathcal{T}_h}
      \del[0]{\eta_{J,2,1}^{\mathcal{K}}}^2
    }^{1/2}
    \tnorm{\boldsymbol{v}_{\mathfrak{h}}}_{s,\mathfrak{h}}.
  \end{equation*}
  \textbf{Bound for $M_5$.} For $M_5$ we first write
  \begin{equation*}
    M_5 = - \underbrace{\langle{
      \del[0]{\beta_s-\beta\cdot n}
      \sbr[0]{\gamma_{\mathfrak{h}}(\boldsymbol{u}_h)}
      ,
      \sbr[0]{(I-i_h)\boldsymbol{v}_{\mathfrak{h}}}
    }
    \rangle_{\partial\mathcal{T}_{\mathfrak{h}}}
    }_{M_{51}}
    + \underbrace{
      \langle{
      \zeta^+
      \beta\cdot n\sbr[0]{\gamma_{\mathfrak{h}}(\boldsymbol{u}_h)}
      ,
      (I-i_h^{\mathcal{F}})\mu_{\mathfrak{h}}
    }\rangle_{\partial\mathcal{E}_N}
    }_{M_{52}}.
  \end{equation*}
  To bound $M_{51}$ we consider the $\mathcal{Q}$-facets and
  $\mathcal{R}$-facets separately. For the $\mathcal{Q}$-facets we
  follow the same steps as used in bounding $M_4$. Let
  $\mathring{\mathcal{K}}\in\mathcal{T}_{\mathcal{K}}$ and
  $F_{\mathcal{Q}}\subset\mathcal{Q}_{\mathring{\mathcal{K}}}$. If
  $F_{\mathcal{Q}}\in\mathcal{F}_{\mathcal{Q},\mathfrak{h}}$, we use
  \cref{eq:qfacediffprojrewrite}, the Cauchy--Schwarz inequality,
  boundedness of the projection $i_h^{\mathcal{F}}$, and
  \cref{eq:betasinfmax}:
  \begin{equation}
    \label{eq:mterms6}
      \langle{
        \del[0]{
          \beta_s-\beta\cdot n
        }
        \sbr[0]{\gamma_{\mathfrak{h}}(\boldsymbol{u}_h)}
        ,
        \sbr[0]{
          \del[0]{I-i_h}
          \boldsymbol{v}_\mathfrak{h}
        }
      }\rangle_{F_{\mathcal{Q}}}
      \le c
      \norm[0]{
        \envert[0]{\beta_{s}-\tfrac{1}{2}\beta\cdot n}^{1/2}
        \sbr[0]{\boldsymbol{u}_{h}}
      }_{\mathcal{Q}_{\mathcal{K}}}
      \norm[0]{
        \envert[0]{\beta_{s}-\tfrac{1}{2}\beta\cdot n}^{1/2}
        \sbr[0]{\boldsymbol{v}_\mathfrak{h}}
      }_{\mathcal{Q}_{\mathcal{K}}}.
  \end{equation}
  If $F_{\mathcal{Q}}\notin \mathcal{F}_{\mathcal{Q},\mathfrak{h}}$,
  we have, using \cref{eq:qfacediffprojrewrite-2},
  \begin{equation}
    \label{eq:M51splittingterm2}
    \begin{split}
      \langle{
        \del[0]{
          \beta_s-\beta\cdot n
        }
        \sbr[0]{\gamma_{\mathfrak{h}}(\boldsymbol{u}_h)}
        ,
        \sbr[0]{
          \del[0]{I-i_h}
          \boldsymbol{v}_\mathfrak{h}
        }
      }\rangle_{F_{\mathcal{Q}}}
      =
      &
      \langle{
        \del[0]{
          \beta_s-\beta\cdot n
        }
        \sbr[0]{\gamma_{\mathfrak{h}}(\boldsymbol{u}_h)}
        ,
        \del[0]{I-i_h^{\mathcal{K}}}
        v_\mathfrak{h}
      }\rangle_{F_{\mathcal{Q}}}
      \\
      &+
      \langle{
        \del[0]{
          \beta_s-\beta\cdot n
        }
        \sbr[0]{\gamma_{\mathfrak{h}}(\boldsymbol{u}_h)}
        ,
        \del[0]{I-i_h^{\mathcal{F}}}
          \sbr[0]{\boldsymbol{v}_{nb,\mathfrak{h}}}
      }\rangle_{F_{\mathcal{Q}}}
      \\
      &+
      \langle{
        \del[0]{
          \beta_s-\beta\cdot n
        }
        \sbr[0]{\gamma_{\mathfrak{h}}(\boldsymbol{u}_h)}
        ,
        \del[0]{I-i_h^{\mathcal{K}}}
        v_{nb,\mathfrak{h}}
      }\rangle_{F_{\mathcal{Q}}}.
    \end{split}
  \end{equation}
  The second term on the right-hand side of
  \cref{eq:M51splittingterm2} is bounded in the same way as
  \cref{eq:mterms6}. For the first term on the right-hand side of
  \cref{eq:M51splittingterm2} we use the Cauchy--Schwarz inequality,
  the trace inequality \cref{eq:eg_inv_3}, the subgrid projection
  estimate \cref{eq:subgrid_proj_est_1}, the estimates
  \cref{eq:subgridhelperbnds_1,eq:subgridhelperbnds_2}, and
  \cref{eq:betasinfmax} to find:
  \begin{equation}
    \label{eq:mterms6-2}
    \begin{split}
      &
      \langle{
        \del[0]{
          \beta_s-\beta\cdot n
        }\sbr[0]{\boldsymbol{u}_h}
        ,
        \del[0]{I-i_h^{\mathcal{K}}}
        {v}_\mathfrak{h}
      }\rangle_{F_{\mathcal{Q}}}
      \\
      &\quad
      \le c \varepsilon^{-1/2}
      \norm[0]{\envert[0]{\beta_{s}-\tfrac{1}{2}\beta\cdot n}^{1/2}\sbr[0]{\boldsymbol{u}_h}}_{\mathcal{Q}_{\mathcal{K}}}
      \del[2]{
        \sum_{\mathring{\mathcal{K}} \in \mathcal{T}_{\mathcal{K}}}
        \del[1]{
          \norm[0]{|\beta_s - \tfrac{1}{2}\beta\cdot n|^{1/2}\sbr[0]{\boldsymbol{v}_{\mathfrak{h}}}}_{\partial\mathring{\mathcal{K}}\cap F_{\mathring{\mathcal{R}}}}
          +
          \tau_{\varepsilon}^{1/2}
          \norm[0]{\partial_t v_{\mathfrak{h}}}_{\mathring{\mathcal{K}}}
        }
      }.
    \end{split}
  \end{equation}
  The third term on the right-hand side of \cref{eq:M51splittingterm2}
  is bounded in the same way. Combining
  \cref{eq:mterms6,eq:mterms6-2}, we bound the contributions from the
  $\mathcal{Q}$-facets to $M_{51}$ as follows:
  \begin{equation}
    \label{eq:mterms5-4}
    \langle{
      \del[0]{\beta_s-\beta\cdot n}
      \sbr[0]{\gamma_{\mathfrak{h}}(\boldsymbol{u}_h)},
      \sbr[0]{\del[0]{I-i_h}\boldsymbol{v}_\mathfrak{h} }
    }\rangle_{\mathcal{Q}_\mathfrak{h}}
    \le c
    \varepsilon^{-1/2}
    \del[1]{
      \sum_{\mathcal{K}\in\mathcal{T}_h}
      \del[0]{\eta_{J,3,\mathcal{Q}}^{\mathcal{K}}}^2
    }^{1/2}
    \tnorm{\boldsymbol{v}_{\mathfrak{h}}}_{s,\mathfrak{h}}.
  \end{equation}
  Next we consider the contributions of the $\mathcal{R}$-facets
  to $M_{51}$. Using that
  $\del[0]{I-i_h^{\mathcal{F}}}\mu_{\mathfrak{h}}=0$ on $F \in
  \mathcal{F}_{\mathcal{R},h}$, and that
  $\sbr[0]{\gamma_{\mathfrak{h}}(\boldsymbol{u}_h)} = 0$ on $F
  \in \mathcal{F}_{\mathcal{R},\mathfrak{h}} \setminus
  \mathcal{F}_{\mathcal{R},h}$, using the Cauchy--Schwarz
  inequality, the trace inequality \cref{eq:eg_inv_4}, the
  subgrid projection estimate \cref{eq:subgrid_proj_est_1}, the
  estimates \cref{eq:subgridhelperbnds_1,eq:subgridhelperbnds_2}
  the inverse estimate \cref{eq:eg_inv_low_d_3}, we find
  \begin{equation}
    \label{eq:mterms7}
    \begin{split}
      &
      \langle{
        \del[0]{
          \beta_s-\beta\cdot n
        }
        \sbr[0]{\gamma_{\mathfrak{h}}(\boldsymbol{u}_h)}
        ,
        \sbr[0]{
          \del[0]{I-i_h}
          \boldsymbol{v}_\mathfrak{h}
        }
      }\rangle_{\mathcal{R}_\mathfrak{h}}
      =
      \langle{
        \del[0]{
          \beta_s-\beta\cdot n
        }
        \sbr[0]{\boldsymbol{u}_h},
        \del[0]{I-i_h^\mathcal{K}}
        {v_\mathfrak{h}}
      }\rangle_{\mathcal{R}_h}
      \\
      &\qquad
      \le c \sum_{\mathcal{K} \in \mathcal{T}_h}
      \varepsilon^{-1/2}
      \norm[0]{
        \envert[0]{\beta_{s}-\tfrac{1}{2}\beta\cdot n}^{1/2}
        \sbr[0]{\boldsymbol{u}_h}
      }_{\mathcal{R}_{\mathcal{K}}}
      \del[1]{
        \sum_{\mathring{\mathcal{K}}\in\mathcal{T}_{\mathcal{K}} }
        \norm[0]{
          \envert[0]{\beta_{s}-\tfrac{1}{2}\beta\cdot n}^{1/2}
          \sbr[0]{\boldsymbol{v}_\mathfrak{h}}
        }_{\partial\mathring{\mathcal{K}}}
        +
        \sum_{\mathring{\mathcal{K}}\in\mathcal{T}_{\mathcal{K}} }
        \tau_{\varepsilon}^{1/2}\norm{{\partial_tv_\mathfrak{h}}}_{\mathring{\mathcal{K}}}
      }
      \\
      &\qquad
      \le c
      \varepsilon^{-1/2}
      \del[1]{
        \sum_{\mathcal{K}\in\mathcal{T}_h}
        \del[0]{\eta_{J,3,\mathcal{R}}^{\mathcal{K}}}^2
      }^{1/2}
      \tnorm{\boldsymbol{v}_{\mathfrak{h}}}_{s,\mathfrak{h}}.
    \end{split}
  \end{equation}
  We can now bound $M_{51}$ by combining
  \cref{eq:mterms5-4,eq:mterms7}:
  \begin{equation*}
    M_{51} \le c \varepsilon^{-1/2}
    \sbr[2]{
      \del[1]{
        \sum_{\mathcal{K}\in\mathcal{T}_h}
        \del[0]{\eta_{J,3,\mathcal{Q}}^{\mathcal{K}}}^2
      }^{1/2}
      +
      \del[1]{
        \sum_{\mathcal{K}\in\mathcal{T}_h}
        \del[0]{\eta_{J,3,\mathcal{R}}^{\mathcal{K}}}^2
      }^{1/2}
    }
    \tnorm{\boldsymbol{v}_{\mathfrak{h}}}_{s,\mathfrak{h}}.
  \end{equation*}
  For $M_{52}$ we use that
  $\del[0]{I-i_h^{\mathcal{F}}}\mu_{\mathfrak{h}}=0$ on
  $F \in \mathcal{F}_{\mathcal{R},h}$, the Cauchy--Schwarz inequality,
  the boundedness of the projection $i_h^{\mathcal{F}}$, and
  \cref{eq:betasinfmax} to find that
  \begin{equation}
    \label{eq:mterms5-5}
      \langle{
        \zeta^+\beta\cdot n
        \sbr[0]{\gamma_{\mathfrak{h}}(\boldsymbol{u}_h)}
        ,
        \del[0]{I-i_h^{\mathcal{F}}}\mu_\mathfrak{h}
      }\rangle_{\partial\mathcal{E}_N}
      \le c
      \sum_{\mathcal{K} \in \mathcal{T}_h}
      \norm[0]{\envert[0]{\beta_s-\tfrac{1}{2}\beta\cdot{n}}^{1/2}\sbr[0]{\boldsymbol{u}_h}}_{\mathcal{Q}_{\mathcal{K}} \cap \partial\mathcal{E}_N}
      \beta_s^{1/2}\norm[0]{\mu_{\mathfrak{h}}}_{\mathcal{Q}_{\mathcal{K}} \cap \partial\mathcal{E}_N}.
  \end{equation}
  To bound
  $\beta_s^{1/2}\norm[0]{\mu_{\mathfrak{h}}}_{\mathcal{Q}_{\mathcal{K}}
    \cap \partial\mathcal{E}_N}$, consider a facet
  $F_{\mathcal{Q}} \subset \mathcal{Q}_{\mathcal{K}} \cap
  \partial\mathcal{E}_N$. By the mean value theorem for definite
  integrals (see, for example, \cite[Theorem 14.16]{Apostol:book}),
  there exists $(t_m,x_m)\in F_{\mathcal{Q}}$ such that
  \begin{equation}
    \label{eq:integralmvt}
    \norm[0]{\envert[0]{\beta\cdot n}^{1/2}\mu_{\mathfrak{h}}}_{F_{\mathcal{Q}}}^2
    =
    \int_{F_{\mathcal{Q}}}
    \envert[0]{\beta\cdot n}
    \mu_{\mathfrak{h}}^2\dif s
    =
    \envert[0]{\beta\cdot n}|_{(t_m,x_m)}
    \int_{F_{\mathcal{Q}}}
    \mu_{\mathfrak{h}}^2\dif s
    =
    \envert[0]{\beta\cdot n}|_{(t_m,x_m)}
    \norm[0]{\mu_{\mathfrak{h}}}_{F_{\mathcal{Q}}}^2.
  \end{equation}
  Let $(t_M,x_M)\in F_{\mathcal{Q}}$ be the point on which
  $\envert[0]{\beta\cdot n}$ attains its maximum $\beta_s$ on
  $F_{\mathcal{Q}}$. Since $\beta$ is Lipschitz continuous and $n$ is
  constant on $F_{\mathcal{Q}}$ (since $\mathcal{Q}$-facets are flat),
  we deduce that $\beta\cdot n$ is Lipschitz continuous on
  $F_{\mathcal{Q}}$. Thus, using that $\delta t_{\mathcal{K}}\le h_K$,
  we have
  \begin{equation}
    \label{eq:betadotnlipschitz}
    \envert[1]{\beta_s-\envert[0]{\beta\cdot n}|_{(t_m,x_m)}}
    \le c\envert[0]{
      (t_m,x_m) - (t_M,x_M)
    }\le ch_K.
  \end{equation}
  A consequence of \cref{eq:integralmvt}, \cref{eq:betadotnlipschitz},
  and \cref{eq:eg_inv_3} is the following bound:
  \begin{equation}
    \label{eq:m52helper}
    \begin{split}
      \beta_s
      \norm[0]{
        \mu_{\mathfrak{h}}
      }_{
        \mathcal{Q}_{\mathcal{K}}
        \cap
        \partial\mathcal{E}_N
      }^2
      & \le
      \envert[0]{\beta\cdot n}|_{(t_m,x_m)}
      \norm[0]{\mu_{\mathfrak{h}}}_{\mathcal{Q}_{\mathcal{K}} \cap
        \partial\mathcal{E}_N}^2
      +ch_K
      \norm[0]{\mu_{\mathfrak{h}}}_{\mathcal{Q}_{\mathcal{K}} \cap
        \partial\mathcal{E}_N}^2
      \\
      & \le
      c\varepsilon^{-1}\sbr[2]{
        \norm[0]{\envert[0]{\beta\cdot n}^{1/2}\mu_{\mathfrak{h}}}_{\mathcal{Q}_{\mathcal{K}} \cap
          \partial\mathcal{E}_N}^2
        + \norm[0]{v_{\mathfrak{h}}}_{\mathcal{K}}^2
        + \varepsilon h_K^{-1}
        \norm[0]{\sbr[0]{\boldsymbol{v}_{\mathfrak{h}}}}_{\mathcal{Q}_{\mathcal{K}} \cap
          \partial\mathcal{E}_N}^2}.
    \end{split}
  \end{equation}
  Combining \cref{eq:mterms5-5,eq:m52helper}, we find the following
  bound
  $M_{52} \le c\varepsilon^{-1/2} \del[1]{
    \sum_{\mathcal{K}\in\mathcal{T}_h}
    \del[0]{\eta_{J,3,\mathcal{Q}}^{\mathcal{K}}}^2 }^{1/2}
  \tnorm{\boldsymbol{v}_{\mathfrak{h}}}_{s,\mathfrak{h}}$.

  Combining the bounds for $M_{51}$ and $M_{52}$ we find that
  \begin{equation*}
    M_5 \le
    c \varepsilon^{-1/2}
    \sbr[2]{
      \del[1]{
        \sum_{\mathcal{K}\in\mathcal{T}_h}
        \del[0]{\eta_{J,3,\mathcal{Q}}^{\mathcal{K}}}^2
      }^{1/2}
      +
      \del[1]{
        \sum_{\mathcal{K}\in\mathcal{T}_h}
        \del[0]{\eta_{J,3,\mathcal{R}}^{\mathcal{K}}}^2
      }^{1/2}
    }
    \tnorm{\boldsymbol{v}_{\mathfrak{h}}}_{s,\mathfrak{h}}.
  \end{equation*}
  \textbf{Bound for $M_6$.} The derivation of a bound for $M_6$ is
  similar to that of the bound for $M_{22}$:
  \begin{equation*}
      M_6
      \le c
      \del[1]{
        \sum_{\mathcal{K}\in\mathcal{T}_h}
        \del[0]{\eta_{BC,1}^{\mathcal{K}}}^2
      }^{1/2}
      \tnorm{\boldsymbol{v}_{\mathfrak{h}}}_{s,\mathfrak{h}}.
  \end{equation*}
  Combining \cref{eq:pretimederivativeest,eq:amathfrakhM1toM6} with
  the bounds for $M_1$ to $M_6$ we find:
  \begin{multline*}
    \del[2]{
      \sum_{\mathcal{K}\in\mathcal{T}_h}
      \tau_{\varepsilon}
      \norm[0]{\partial_t\del[0]{u_{\mathfrak{h}}-u_h}}_{\mathcal{K}}^2
    }^{{1}/{2}}
    \le c_T
    \Big(
    \del[1]{
      \sum_{\mathcal{K}\in\mathcal{T}_h}
      \del[0]{
        \eta_R^{\mathcal{K}}
      }^2
    }^{1/2}
    +
    \del[1]{
      \sum_{\mathcal{K}\in\mathcal{T}_h}
      \del[0]{
        \eta_{J,1}^{\mathcal{K}}
      }^2
    }^{1/2}
    \\
    +
    \del[1]{
      \sum_{\mathcal{K}\in\mathcal{T}_h}
      \del[0]{
        \eta_{J,2,1}^{\mathcal{K}}
      }^2
    }^{1/2}
    +
    \del[1]{
      \sum_{\mathcal{K}\in\mathcal{T}_h}
      \del[0]{
        \eta_{J,2,2}^{\mathcal{K}}
      }^2
    }^{1/2}
    +
    \varepsilon^{-1/2}
    \del[1]{
      \sum_{\mathcal{K}\in\mathcal{T}_h}
      \del[0]{
        \eta_{J,3}^{\mathcal{K}}
      }^2
    }^{1/2}
    +
    \del[1]{
      \sum_{\mathcal{K}\in\mathcal{T}_h}
      \del[0]{
        \eta_{BC,1}^{\mathcal{K}}
      }^2
    }^{1/2}
    \Big).
  \end{multline*}
  \Cref{eq:timederivativeest} follows by using H\"older's inequality
  for sums and combining with \cref{eq:tauepsdtuuh}.
\end{proof}

\subsection{Reliability of the error estimator}
\label{ss:rel}

In this section we prove \cref{thm:reliability}. Let $e_u:=u-u_h$
denote the true error. To derive an upper bound for $e_u$ we follow
\cite{Schotzau:2014,Houston:2007,Schotzau:2009,Zhu:thesis} and
consider the following decomposition of
$u_h=\mathcal{I}^c_hu_h + u^r_h$. Here $\mathcal{I}_h^c$ is the
averaging operator defined in \cref{ss:ineqapproxbounds} and
$u^r_h:=u_h-\mathcal{I}^c_hu_h$.  We further introduce
$e^c_u:=u-\mathcal{I}^c_hu_h$.

In this section we will use the following weighting function:
\begin{equation}
  \label{eq:weight_func}
  \varphi := eT\exp(-t/T) + \chi,
\end{equation}
where the positive constant $\chi$ will be determined later. We
further introduce the forms
$k_h(\boldsymbol{u},\boldsymbol{v})
=-\langle\varepsilon\sbr[0]{\boldsymbol{u}},\overline{\nabla}_{\overline{n}}v\rangle_{\mathcal{Q}_h}
-\langle\varepsilon\overline{\nabla}_{\overline{n}}u,\sbr[0]{\boldsymbol{v}}\rangle_{\mathcal{Q}_h}$,
$b_h(\lambda,\mu) =\langle\zeta^+\beta\cdot
n\lambda,\mu\rangle_{\partial\mathcal{E}_N}$, and
$\widetilde{a}_h(\boldsymbol{u},\boldsymbol{v}) =
a_h(\boldsymbol{u},\boldsymbol{v})-k_h(\boldsymbol{u},\boldsymbol{v})-b_h(\lambda,
\mu)$ (see \cite{Schotzau:2014,Schotzau:2009,Schotzau:2011a}).

\begin{lemma}
  \label{lem:err_stronger_stab_full}
  Let $\varphi$ be as in \cref{eq:weight_func}. Then,
  \begin{equation}
    \label{eq:err_stronger_stab_full}
    \chi
    \sum_{\mathcal{K}\in\mathcal{T}_h}
    \varepsilon
    \norm[0]{\overline{\nabla} e_u^c}_{\mathcal{K}}^2
    +
    \tfrac{1}{2}
    \sum_{\mathcal{K}\in\mathcal{T}_h}
    \norm{e_u^c}_{\mathcal{K}}^2
    +
    \tfrac{1}{2}\chi
    \sum_{\mathcal{K}\in\mathcal{T}_h}
    \norm[0]{\envert[0]{\beta\cdot n}^{1/2}e_u^c}_{\partial\mathcal{E}_N}^2
    \le
    \sum_{i=1}^6
    T_i,
  \end{equation}
  where
  \begin{align*}
    T_1
    &=
      \del[0]{R_h^{\mathcal{K}},(I-\Pi_h)(\varphi
      e_u^c)}_{\mathcal{T}_h},
    \\
    T_2
    &=
      -
      \langle
      \varepsilon\overline{\nabla}_{\overline{n}}u_h,
      \del[0]{I-\Pi_h^{\mathcal{F}}}\del[0]{\varphi e_u^c}
      \rangle_{\mathcal{Q}_h^i}
      +
      \langle R_h^N,(I-\Pi_{h}^{\mathcal{F}})(\varphi
      e_u^c)\rangle_{\partial\mathcal{E}_N\setminus\Omega_T},
    \\
    T_3
    &=
      \langle{
      \varepsilon\alpha h_K^{-1}
      \sbr[0]{\boldsymbol{u}_h}
      ,
      \del[0]{\Pi_h-\Pi_h^{\mathcal{F}}}
      \del[0]{\varphi e_u^c}
      }\rangle_{\mathcal{Q}_h}
      -\langle
      \varepsilon\sbr[0]{\boldsymbol{u}_h},
      \overline{\nabla}_{\overline{n}}
      \del[0]{\Pi_h\del[0]{\varphi e_u^c}}
      \rangle_{\mathcal{Q}_h},
    \\
    T_4
    &=
      \langle
      \beta\cdot nu_h,\del[0]{I-\Pi_h^{\mathcal{F}}}\del[0]{\varphi e_u^c}
      \rangle_{\partial\mathcal{T}_h^i}
      +
      \langle
      \del[0]{\beta_s-\beta\cdot n}\sbr[0]{\boldsymbol{u}_h}
      ,
      \del[0]{\Pi_h-\Pi_h^{\mathcal{F}}}
      \del[0]{\varphi e_u^c}
      \rangle_{\partial\mathcal{T}_h^i},
    \\
    T_5 &=
          \del[0]{
          \varepsilon\overline{\nabla}
          \del[0]{I-\mathcal{I}_h^c}u_h,
          \overline{\nabla}\del[0]{\varphi e_u^c}
          }_{\mathcal{T}_h}
          -
          \del[0]{
          \beta\del[0]{I-\mathcal{I}_h^c}u_h
          ,\nabla\del[0]{\varphi e_u^c}
          }_{\mathcal{T}_h},
    \\
    T_6 &=
          \langle
          {\zeta^+}\beta\cdot n\del[0]{u_h-\mathcal{I}_h^cu_h},{\varphi e_u^c}
          \rangle_{\partial\mathcal{E}_N}
          -
          \langle \zeta^+\beta\cdot n\sbr[0]{\boldsymbol{u}_h},
          \Pi_h^{\mathcal{F}}\del[0]{\varphi e_u^c}
          \rangle_{\partial\mathcal{E}_N}
    \\
    &\qquad
      +
      \langle
      \del[0]{\beta_s-\beta\cdot n}\sbr[0]{\boldsymbol{u}_h}
      ,
      \del[0]{\Pi_h-\Pi_h^{\mathcal{F}}}
      \del[0]{\varphi e_u^c}
      \rangle_{\partial\mathcal{E}\setminus\Omega_T}.
  \end{align*}
\end{lemma}
\begin{proof}
  Using the definition of the weighting function
  \cref{eq:weight_func}, that $\varphi\ge\chi$, that
  $\beta\cdot\nabla\varphi =\partial_t\varphi = -e\exp(-t/T)$ and that
  $\overline{\nabla}\varphi = 0$, we have
  \begin{equation*}
      \del[0]{\varepsilon\overline{\nabla}e_u^c,\overline{\nabla}\del[0]{\varphi e_u^c}}_{\mathcal{T}_h}
      \ge
      \chi \varepsilon \del[0]{\overline{\nabla}e_u^c, \overline{\nabla} e_u^c }_{\mathcal{T}_h}
      \quad \text{and} \quad
      -\tfrac{1}{2}
      \del[0]{\del[0]{{\beta}\cdot\nabla\varphi}e_u^c,e_u^c}_{\mathcal{T}_h}
      \ge
      \tfrac{1}{2}
      \del[0]{ e_u^c,e_u^c}_{\mathcal{T}_h},
  \end{equation*}
  so that
  \begin{equation}
    \label{eq:err_stronger_stab}
    \chi
    \sum_{\mathcal{K}\in\mathcal{T}_h}
    \varepsilon
    \norm[0]{\overline{\nabla} e_u^c}_{\mathcal{K}}^2
    +
    \tfrac{1}{2}
    \sum_{\mathcal{K}\in\mathcal{T}_h}
    \norm{e_u^c}_{\mathcal{K}}^2
    \leq
    \del[0]{\varepsilon\overline{\nabla}e_u^c,\overline{\nabla}\del[0]{\varphi e_u^c}}_{\mathcal{T}_h}
    -\tfrac{1}{2}
    \del[0]{\del[0]{{\beta}\cdot\nabla\varphi}
      e_u^c,e_u^c}_{\mathcal{T}_h}.
  \end{equation}
  For the right-hand side of \cref{eq:err_stronger_stab}, using that
  $-\tfrac{1}{2}(\beta\cdot\nabla\varphi)(e_u^c)^2 = \varphi e_u^c
  \nabla \cdot(\beta e_u^c) -
  \tfrac{1}{2}\nabla\cdot(\beta\varphi(e_u^c)^2)$ because
  $\nabla \cdot \beta = 0$, integration by parts, that $\beta\cdot n$,
  $e_u^c$, and $\varphi$ are single-valued on element boundaries, that
  $e_u^c$ vanishes on $\partial\mathcal{E}_D$, the divergence
  theorem, and \cref{eq:st_adr}, we find:
  \begin{equation}
    \label{eq:upperboundbreakdown1}
    \begin{split}
      &\del[0]{\varepsilon\overline{\nabla}e_u^c,\overline{\nabla}\del[0]{\varphi e_u^c}}_{\mathcal{T}_h}
      -\tfrac{1}{2}
      \del[0]{\del[0]{{\beta}\cdot\nabla\varphi} e_u^c,e_u^c}_{\mathcal{T}_h}
      \\
      =
      &
      -\del[0]{\varepsilon\overline{\nabla}^2u,{\varphi e_u^c}}_{\mathcal{T}_h}
      +
      \langle\varepsilon\overline{\nabla}_{\overline{n}}u,\varphi
      e_u^c\rangle_{\mathcal{Q}_h\cap\partial\mathcal{E}_N}
      -
      \del[0]{\varepsilon\overline{\nabla}\mathcal{I}_h^cu_h,\overline{\nabla}\del[0]{\varphi e_u^c}}_{\mathcal{T}_h}
      \\
      &\quad
      +
      \del[0]{\nabla\cdot\del[0]{{\beta}u},\varphi e_u^c}_{\mathcal{T}_h}
      -
      \del[0]{\nabla\cdot\del[0]{{\beta}\mathcal{I}_h^cu_h},\varphi e_u^c}_{\mathcal{T}_h}
      -\tfrac{1}{2}
      \langle{
        {\beta}\cdot n e_u^c,\varphi e_u^c
      }\rangle_{\partial\mathcal{E}_N}
      \\
      =&
      \del[0]{f,{\varphi e_u^c}}_{\mathcal{T}_h}
      +
      \langle\varepsilon\overline{\nabla}_{\overline{n}}u,\varphi
      e_u^c\rangle_{\mathcal{Q}_h\cap\partial\mathcal{E}_N}
      -
      \del[0]{\varepsilon\overline{\nabla}\mathcal{I}_h^cu_h,\overline{\nabla}\del[0]{\varphi e_u^c}}_{\mathcal{T}_h}
      \\
      &\qquad
      +\del[0]{
        {\beta}\mathcal{I}_h^cu_h,
        \nabla\del[0]{\varphi e_u^c}
      }_{\mathcal{T}_h}
      -\langle
      \beta\cdot n\mathcal{I}_h^cu_h,
      \varphi e_u^c
      \rangle_{\partial\mathcal{E}}
      -\tfrac{1}{2}
      \langle{
        {\beta}\cdot n e_u^c,\varphi e_u^c
      }\rangle_{\partial\mathcal{E}_N}
      \\
      =&
      \del[0]{f,{\varphi e_u^c}}_{\mathcal{T}_h}
      -\widetilde{a}_h((\mathcal{I}_h^cu_h,\mathcal{I}_h^cu_h),(\varphi e_h^c,\varphi e_h^c))
      +
      \langle\varepsilon\overline{\nabla}_{\overline{n}}u,\varphi
      e_u^c\rangle_{\mathcal{Q}_h\cap\partial\mathcal{E}_N}
      \\
      &\qquad
      -\langle
      \beta\cdot n\mathcal{I}_h^cu_h,
      \varphi e_u^c
      \rangle_{\partial\mathcal{E}_N}
      -\tfrac{1}{2}
      \langle{
        {\beta}\cdot n e_u^c,\varphi e_u^c
      }\rangle_{\partial\mathcal{E}_N}.
    \end{split}
  \end{equation}
  Using
  $\zeta^-\beta\cdot n=\tfrac{1}{2}\del[0]{\beta\cdot
    n-\envert[0]{\beta\cdot n}}$, the last term above, excluding
  $\Omega_T\subset\partial\mathcal{E}_N$, is rewritten as follows
  \begin{multline*}
    -\tfrac{1}{2}
    \langle{
      {\beta}\cdot n e_u^c,\varphi e_u^c
    }\rangle_{\partial\mathcal{E}_N\setminus\Omega_T}
    \\
    =
    -
    \langle{
      \zeta^-\beta\cdot nu,\varphi e_u^c
    }\rangle_{\partial\mathcal{E}_N\setminus\Omega_T}
    -\tfrac{1}{2}
    \langle{
      \envert[0]{{\beta}\cdot n} u,\varphi e_u^c
    }\rangle_{\partial\mathcal{E}_N\setminus\Omega_T}
    +\tfrac{1}{2}
    \langle{
      {\beta}\cdot n \mathcal{I}_h^cu_h,\varphi e_u^c
    }\rangle_{\partial\mathcal{E}_N\setminus\Omega_T}.
  \end{multline*}
  Therefore, using the Neumann boundary condition
  \cref{eq:st_adr_bcN}, the right-hand side of
  \cref{eq:upperboundbreakdown1} becomes
  \begin{equation}
    \label{eq:upperboundbreakdown2}
    \begin{split}
      &
      \del[0]{f,{\varphi e_u^c}}_{\mathcal{T}_h}
      +
      \langle g,\varphi e_u^c
      \rangle_{\partial\mathcal{E}_N\setminus\Omega_T}
      -\widetilde{a}_h((\mathcal{I}_h^cu_h,\mathcal{I}_h^cu_h),(\varphi e_h^c,\varphi e_h^c))
      -\langle
      \beta\cdot n\mathcal{I}_h^cu_h,
      \varphi e_u^c
      \rangle_{\partial\mathcal{E}_N}
      \\
      &\qquad
      -\tfrac{1}{2}
      \langle{
        {\beta}\cdot n e_u^c,\varphi e_u^c
      }\rangle_{\Omega_T}
      -\tfrac{1}{2}
      \langle{
        \envert[0]{{\beta}\cdot n} u,\varphi e_u^c
      }\rangle_{\partial\mathcal{E}_N\setminus\Omega_T}
      +\tfrac{1}{2}
      \langle{
        {\beta}\cdot n \mathcal{I}_h^cu_h,\varphi e_u^c
      }\rangle_{\partial\mathcal{E}_N\setminus\Omega_T}
      \\
      =&
      \del[0]{f,{\varphi e_u^c}}_{\mathcal{T}_h}
      +
      \langle g,\varphi e_u^c
      \rangle_{\partial\mathcal{E}_N\setminus\Omega_T}
      -\widetilde{a}_h((\mathcal{I}_h^cu_h,\mathcal{I}_h^cu_h),(\varphi e_h^c,\varphi e_h^c))
      -\tfrac{1}{2}\langle
      \beta\cdot n\mathcal{I}_h^cu_h,
      \varphi e_u^c
      \rangle_{\partial\mathcal{E}_N\setminus\Omega_T}
      \\
      &\qquad-\langle
      \beta\cdot n\mathcal{I}_h^cu_h,
      \varphi e_u^c
      \rangle_{\Omega_T}
      -\tfrac{1}{2}
      \langle{
        {\beta}\cdot n e_u^c,\varphi e_u^c
      }\rangle_{\Omega_T}
      -\tfrac{1}{2}
      \langle{
        \envert[0]{{\beta}\cdot n} u,\varphi e_u^c
      }\rangle_{\partial\mathcal{E}_N\setminus\Omega_T}
      \\
      =&
      \del[0]{f,{\varphi e_u^c}}_{\mathcal{T}_h}
      +
      \langle g,\varphi e_u^c
      \rangle_{\partial\mathcal{E}_N\setminus\Omega_T}
      -\widetilde{a}_h((\mathcal{I}_h^cu_h,\mathcal{I}_h^cu_h),(\varphi e_h^c,\varphi e_h^c))
      -\mathfrak{B}_h,
    \end{split}
  \end{equation}
  where in the last step we collect remaining boundary terms in
  $\mathfrak{B}_h$:
  \begin{equation*}
    \mathfrak{B}_h
    =
    \tfrac{1}{2}\langle
    \beta\cdot n\mathcal{I}_h^cu_h,
    \varphi e_u^c
    \rangle_{\partial\mathcal{E}_N\setminus\Omega_T}
    +\tfrac{1}{2}\langle
    \beta\cdot n\mathcal{I}_h^cu_h,
    \varphi e_u^c
    \rangle_{\Omega_T}
    +\tfrac{1}{2}
    \langle{
      {\beta}\cdot n u,\varphi e_u^c
    }\rangle_{\Omega_T}
    +\tfrac{1}{2}
    \langle{
      \envert[0]{{\beta}\cdot n} u,\varphi e_u^c
    }\rangle_{\partial\mathcal{E}_N\setminus\Omega_T}.
  \end{equation*}
  Next, the HDG method \cref{eq:st_hdg_adr_compact} with test
  functions
  $\boldsymbol{\Pi}_h(\varphi e_u^c,\varphi e_u^c):=(\Pi_h(\varphi
  e_u^c),\Pi_h^{\mathcal{F}}(\varphi e_u^c))$, and noting that
  $\Pi_h^{\mathcal{F}}(\varphi e_u^c)=0$ on $\partial\mathcal{E}_D$,
  becomes:
  \begin{multline}
    \label{eq:upperboundbreakdown3}
    0
    =
    -\del[0]{f,\Pi_h(\varphi e_u^c)}_{\mathcal{T}_h}
    -
    \langle g,\Pi_{h}^{\mathcal{F}}(\varphi e_u^c)\rangle_{\partial\mathcal{E}_N\setminus\Omega_T}
    \\
    +
    \widetilde{a}_h(\boldsymbol{u}_h,\boldsymbol{\Pi}_h(\varphi e_u^c,\varphi e_u^c))
    +
    {k}_h(\boldsymbol{u}_h,\boldsymbol{\Pi}_h(\varphi e_u^c,\varphi e_u^c))
    +
    {b}_h(\lambda_h,\Pi_h^{\mathcal{F}}(\varphi e_u^c)).
  \end{multline}
  Adding \cref{eq:upperboundbreakdown3} to
  \cref{eq:upperboundbreakdown2}, the right-hand side of
  \cref{eq:err_stronger_stab} becomes
  \begin{equation}
    \label{eq:upperboundbreakdown4}
    \begin{split}
      \del[0]{\varepsilon\overline{\nabla}e_u^c,\overline{\nabla}\del[0]{\varphi e_u^c}}_{\mathcal{T}_h}
      -\tfrac{1}{2}
      \del[0]{\del[0]{{\beta}\cdot\nabla\varphi}
        e_u^c,e_u^c}_{\mathcal{T}_h}
      =
      &
      \del[0]{f,(I-\Pi_h)(\varphi e_u^c)}_{\mathcal{T}_h}
      +
      \langle g,(I-\Pi_{h}^{\mathcal{F}})(\varphi e_u^c)\rangle_{\partial\mathcal{E}_N\setminus\Omega_T}
      \\
      &
      -
      \widetilde{a}_h(\boldsymbol{u}_h,(\boldsymbol{I}-\boldsymbol{\Pi}_h)(\varphi e_u^c,\varphi e_u^c))
      -\mathfrak{R}_h
      \\
      &
      +
      {k}_h(\boldsymbol{u}_h,\boldsymbol{\Pi}_h(\varphi e_u^c,\varphi e_u^c))
      +
      {b}_h(\lambda_h,\Pi_h^{\mathcal{F}}(\varphi e_u^c))
      -\mathfrak{B}_h,
    \end{split}
  \end{equation}
  where
  $\mathfrak{R}_h:=\widetilde{a}_h((\mathcal{I}_h^cu_h,\mathcal{I}_h^cu_h),(\varphi
  e_h^c,\varphi e_h^c))-\widetilde{a}_h(\boldsymbol{u}_h,(\varphi
  e_u^c,\varphi e_u^c))$. By definition of $\widetilde{a}_h$, the
  first three terms on the right-hand side of
  \cref{eq:upperboundbreakdown4} become
  \begin{equation}
    \label{eq:upperboundbreakdown5}
    \begin{split}
      &
      \del[0]{f,(I-\Pi_h)(\varphi e_u^c)}_{\mathcal{T}_h}
      +
      \langle g,(I-\Pi_{h}^{\mathcal{F}})(\varphi e_u^c)\rangle_{\partial\mathcal{E}_N\setminus\Omega_T}
      \\
      &\quad
      -
      \del[0]{
        \varepsilon\overline{\nabla}u_h,
        \overline{\nabla}\del[0]{I-\Pi_h}\del[0]{\varphi e_u^c}
      }_{\mathcal{T}_h}
      -
      \langle{
        \varepsilon\alpha h_K^{-1}
        \sbr[0]{\boldsymbol{u}_h}
        ,
        \sbr[0]{
          \del{\boldsymbol{I}-\boldsymbol{\Pi}_h}
          \del{\varphi e_u^c,\varphi e_u^c}
        }
      }\rangle_{\mathcal{Q}_h}
      \\
      &\quad
      +
      \del[0]{
        {\beta}u_h, \nabla{\del[0]{I-\Pi_h}\del[0]{\varphi e_u^c}}
      }_{\mathcal{T}_h}
      -
      \langle
      {\beta\cdot n}\lambda_h+\beta_s\sbr[0]{\boldsymbol{u}_h}
      ,
      \sbr[0]{
        \del{\boldsymbol{I}-\boldsymbol{\Pi}_h}
        \del{\varphi e_u^c,\varphi e_u^c}
      }
      \rangle_{\partial\mathcal{T}_h}
      \\
      =
      &
      \del[0]{R_h^{\mathcal{K}},(I-\Pi_h)(\varphi e_u^c)}_{\mathcal{T}_h}
      +
      \langle R_h^N,(I-\Pi_{h}^{\mathcal{F}})(\varphi e_u^c)\rangle_{\partial\mathcal{E}_N\setminus\Omega_T}
      \\
      &\quad
      -
      \langle
      \varepsilon\overline{\nabla}_{\overline{n}}u_h,
      \del[0]{I-\Pi_h^{\mathcal{F}}}\del[0]{\varphi e_u^c}
      \rangle_{\mathcal{Q}_h^i}
      +
      \langle{
        \varepsilon\alpha h_K^{-1}
        \sbr[0]{\boldsymbol{u}_h}
        ,
        \del[0]{\Pi_h-\Pi_h^{\mathcal{F}}}
        \del[0]{\varphi e_u^c}
      }\rangle_{\mathcal{Q}_h}
      +
      \langle
      \varepsilon\overline{\nabla}_{\overline{n}}u_h,
      \del[0]{\Pi_h-\Pi_h^{\mathcal{F}}}\del[0]{\varphi e_u^c}
      \rangle_{\mathcal{Q}_h}
      \\
      &\quad
      +
      \langle
      \del[0]{1-\zeta^-}\beta\cdot nu_h,\del[0]{I-\Pi_h^{\mathcal{F}}}\del[0]{\varphi e_u^c}
      \rangle_{\partial\mathcal{E}_N\setminus\Omega_T}
      +
      \langle
      \beta\cdot nu_h,\del[0]{I-\Pi_h^{\mathcal{F}}}\del[0]{\varphi e_u^c}
      \rangle_{\partial\mathcal{T}_h^i\cup\Omega_T}
      \\
      &\quad
      +
      \langle
      {\beta\cdot n}\lambda_h+\beta_s\sbr[0]{\boldsymbol{u}_h}
      ,
      \del[0]{\Pi_h-\Pi_h^{\mathcal{F}}}
      \del[0]{\varphi e_u^c}
      \rangle_{\partial\mathcal{T}_h}
      -
      \langle
      \beta\cdot nu_h,\del[0]{\Pi_h-\Pi_h^{\mathcal{F}}}\del[0]{\varphi e_u^c}
      \rangle_{\partial\mathcal{T}_h}.
    \end{split}
  \end{equation}
  For the $k_h$, $b_h$ and $\mathfrak{R}_h$ terms on the right-hand side
  of \cref{eq:upperboundbreakdown4}, we have
  \begin{equation}
    \label{eq:upperboundbreakdown6}
    \begin{split}
      &k_h(\boldsymbol{u}_h,\boldsymbol{\Pi}_h\del[0]{\varphi e_u^c,\varphi e_u^c})
      =
      -\langle
      \varepsilon\sbr[0]{\boldsymbol{u}_h},
      \overline{\nabla}_{\overline{n}}
      \del[0]{\Pi_h\del[0]{\varphi e_u^c}}
      \rangle_{\mathcal{Q}_h}
      -\langle
      \varepsilon
      \overline{\nabla}_{\overline{n}}u_h,
      \del[0]{\Pi_h-\Pi_h^{\mathcal{F}}}\del[0]{\varphi e_u^c}
      \rangle_{\mathcal{Q}_h},
      \\
      &b_h(\lambda_h,\Pi_h^{\mathcal{F}}\del[0]{\varphi e_u^c})
      =
      \langle \zeta^+\beta\cdot n\lambda_h,
      \Pi_h^{\mathcal{F}}\del[0]{\varphi e_u^c}
      \rangle_{\partial\mathcal{E}_N},
      \\
      &\mathfrak{R}_h
      =
      -
      \del[0]{
        \varepsilon\overline{\nabla}
        \del[0]{I-\mathcal{I}_h^c}u_h,
        \overline{\nabla}\del[0]{\varphi e_u^c}
      }_{\mathcal{T}_h}
      +
      \del[0]{
        \beta\del[0]{I-\mathcal{I}_h^c}u_h
        ,\nabla\del[0]{\varphi e_u^c}
      }_{\mathcal{T}_h}.
    \end{split}
  \end{equation}
  Using
  \cref{eq:upperboundbreakdown4,eq:upperboundbreakdown5,eq:upperboundbreakdown6}
  and the definition of $-\mathfrak{B}_h$ we obtain
  \cref{eq:err_stronger_stab_full}.
\end{proof}

\begin{lemma}
  \label{lem:upperbndwitheuc}
  Let $\varphi$ be as in \cref{eq:weight_func} and assume that
  $\delta t_{\mathcal{K}} = \mathcal{O}(h_K^2)$. The following
  estimate holds:
  \begin{equation*}
    \begin{split}
      T
      \sum_{\mathcal{K}\in\mathcal{T}_h}
      \varepsilon
      \norm[0]{\overline{\nabla} e_u^c}_{\mathcal{K}}^2
      &+
      \tfrac{1}{2}
      \sum_{\mathcal{K}\in\mathcal{T}_h}
      \norm{e_u^c}_{\mathcal{K}}^2
      +
      T
      \sum_{F\in\partial\mathcal{E}_N}
      \norm[0]{\envert[0]{\beta\cdot n}^{1/2}e_u^c}_{F}^2
      \\
      \le &
      c\sum_{\mathcal{K}\in\mathcal{T}_h}
      \big(
      T^2\del[0]{\eta_R^{\mathcal{K}}}^2
      +
      T^2\del[0]{\eta_{J,1}^{\mathcal{K}}}^2
      +
      T^2\varepsilon^{-1}\del[0]{\eta_{J,2,1}^{\mathcal{K}}}^2
      +
      T^2\varepsilon^{-1}\del[0]{\eta_{J,2,2}^{\mathcal{K}}}^2
      \\
      &\qquad
      +
      T^2\varepsilon^{-1}\del[0]{\eta_{J,3,\mathcal{R}}^{\mathcal{K}}}^2
      +
      T^2\del[0]{\eta_{J,3}^{\mathcal{K}}}^2
      +
      T^2\del[0]{\eta_{BC,1}^{\mathcal{K}}}^2
      +
      T\del[0]{\eta_{BC,2}^{\mathcal{K}}}^2
      \big).
    \end{split}
  \end{equation*}
\end{lemma}
\begin{proof}
  We start by bounding the $T_i$, $i=1,\hdots,6$, terms in
  \cref{lem:err_stronger_stab_full}.
  \\
  \textbf{Bound for $T_1$.} Using the Cauchy--Schwarz inequality, the
  local quasi-interpolation estimate
  \cref{eq:local_quasi_tnorm_st_step_ratio_1}, that
  $T+\chi\le \envert[0]{\varphi}\le eT+\chi$ and that
  $1\le \envert[0]{\partial_t\varphi}\le e$, Young's inequality, that
  $\delta t_{\mathcal{K}} = \mathcal{O}(h_K^2)$, and $T \ge 1$, we
  find that
  \begin{equation*}
    \begin{split}
      T_1
      &\leq
      c
      \sum_{\mathcal{K}\in\mathcal{T}_h}
      \eta_R^{\mathcal{K}}
      \del[0]{T+\chi}
      \del[1]{
        h_K\varepsilon^{1/2}\norm[0]{\partial_t{e_u^c}}_{\mathcal{K}}
        +
        \varepsilon^{1/2}
        \norm[0]{\overline{\nabla}{e_u^c}}_{\mathcal{K}}
        +
        \norm[0]{e_u^c}_{\mathcal{K}}
      }
      \\
      &\leq
      \tfrac{c}{\delta}
      \del[0]{T+\chi}^2
      \sum_{\mathcal{K}\in\mathcal{T}_h}
      \del[0]{\eta_R^{\mathcal{K}}}^2
      +
      \tfrac{c\delta}{2}
      \del[0]{T+\chi}
      \sum_{\mathcal{K}\in\mathcal{T}_h}
      \varepsilon
      \norm[0]{\overline{\nabla}{e_u^c}}_{\mathcal{K}}^2
      +
      \tfrac{c\delta}{2}
      \sum_{\mathcal{K}\in\mathcal{T}_h}
      \tau_{\varepsilon}\norm[0]{\partial_t{e_u^c}}_{\mathcal{K}}^2
      +
      \tfrac{c\delta}{2}
      \sum_{\mathcal{K}\in\mathcal{T}_h}
      \norm[0]{e_u^c}_{\mathcal{K}}^2.
    \end{split}
  \end{equation*}
  \textbf{Bound for $T_2$.} We write $T_2$ as
  \begin{equation*}
    T_2
    =
    \underbrace{-
    \langle
    \varepsilon\overline{\nabla}_{\overline{n}}u_h,
    \del[0]{I-\Pi_h^{\mathcal{F}}}\del[0]{\varphi e_u^c}
    \rangle_{\mathcal{Q}_h^i}}_{T_{21}}
    +
    \underbrace{\langle R_h^N,(I-\Pi_{h}^{\mathcal{F}})(\varphi
      e_u^c)\rangle_{\partial\mathcal{E}_N\setminus(\Omega_T\cup\Omega_0)}}_{T_{22}}
    +
    \underbrace{\langle R_h^N,(I-\Pi_{h}^{\mathcal{F}})(\varphi
    e_u^c)\rangle_{\Omega_0}}_{T_{23}}.
  \end{equation*}
  For $T_{21}$ we write element boundary integrals as facet integrals
  in which we use that
  $\del[0]{I-\Pi_h^{\mathcal{F}}}\del[0]{\varphi e_u^c}$ is continuous
  across a facet, use the Cauchy--Schwarz inequality, the triangle
  inequality, the quasi-interpolation estimate
  \cref{eq:local_quasi_tnorm_st_step_ratio_3}, the projection bound
  \cref{eq:otherusefulbnds_2}, that
  $\delta t_{\mathcal{K}} = \mathcal{O}(h_K^2)$, $T\ge 1$, and Young's
  inequality to find
  \begin{equation*}
    T_{21}
    \leq
    \tfrac{c}{\delta}
    \del[0]{T+\chi}^2
    \sum_{\mathcal{K}\in\mathcal{T}_{h}}
    \del[0]{\eta_{J,1}^{\mathcal{K}}}^2
    +
    \tfrac{c\delta}{2}
    \del[0]{T+\chi}
    \sum_{\mathcal{K}\in\mathcal{T}_h}
    \varepsilon
    \norm[0]{\overline{\nabla}{e_u^c}}_{\mathcal{K}}^2
    +
    \tfrac{c\delta}{2}
    \sum_{\mathcal{K}\in\mathcal{T}_h}
    \tau_\varepsilon
    \norm[0]{\partial_t{e_u^c}}_{\mathcal{K}}^2
    +
    \tfrac{c\delta}{2}
    \sum_{\mathcal{K}\in\mathcal{T}_h}
    \norm[0]{e_u^c}_{\mathcal{K}}^2.
  \end{equation*}
  Term $T_{22}$ can be bounded similarly:
  \begin{equation*}
    T_{22}
    \leq
      \tfrac{c}{\delta}
      \del[0]{T+\chi}^2
      \sum_{\mathcal{K}\in\mathcal{T}_{h}}
      \del[0]{\eta_{BC,1}^{\mathcal{K}}}^2
      +
      \tfrac{c\delta}{2}
      \del[0]{T+\chi}
      \sum_{\mathcal{K}\in\mathcal{T}_h}
      \varepsilon
      \norm[0]{\overline{\nabla}{e_u^c}}_{\mathcal{K}}^2
      +
      \tfrac{c\delta}{2}
      \sum_{\mathcal{K}\in\mathcal{T}_h}
      \tau_\varepsilon
      \norm[0]{\partial_t{e_u^c}}_{\mathcal{K}}^2
      +
      \tfrac{c\delta}{2}
      \sum_{\mathcal{K}\in\mathcal{T}_h}
      \norm[0]{e_u^c}_{\mathcal{K}}^2.
  \end{equation*}
  For $T_{23}$ we have, by the Cauchy--Schwarz inequality, boundedness
  of $\Pi_h^{\mathcal{F}}$, that $|\beta \cdot n|=1$ on $\Omega_0$,
  and Young's inequality:
  \begin{equation*}
    T_{23}
      \le
      \tfrac{c}{2\delta}
      \del[0]{T+\chi}
      \sum_{\mathcal{K}\in\mathcal{T}_{h}}
      \del[0]{\eta_{BC,2}^{\mathcal{K}}}^2
      +
      \tfrac{c\delta}{2}
      \del[0]{T+\chi}
      \sum_{F_{\mathcal{R}}\subset\Omega_0}
      \norm[0]{\envert[0]{\tfrac{1}{2}\beta\cdot n}^{1/2}e_u^c
      }_{F_{\mathcal{R}}}^2.
  \end{equation*}
  Combining the bounds for $T_{21}$, $T_{22}$, and $T_{23}$, we
  obtain:
  \begin{equation*}
    \begin{split}
      T_2
      \le
      &
      \tfrac{c}{\delta}
      \del[0]{T+\chi}^2
      \sum_{\mathcal{K}\in\mathcal{T}_{h}}
      \del[0]{\eta_{J,1}^{\mathcal{K}}}^2
      +
      \tfrac{c}{\delta}
      \del[0]{T+\chi}^2
      \sum_{\mathcal{K}\in\mathcal{T}_{h}}
      \del[0]{\eta_{BC,1}^{\mathcal{K}}}^2
      +
      \tfrac{c}{2\delta}
      \del[0]{T+\chi}
      \sum_{\mathcal{K}\in\mathcal{T}_{h}}
      \del[0]{\eta_{BC,2}^{\mathcal{K}}}^2
      \\
      &
      +
      c\delta
      \del[0]{T+\chi}
      \sum_{\mathcal{K}\in\mathcal{T}_h}
      \varepsilon
      \norm[0]{\overline{\nabla}{e_u^c}}_{\mathcal{K}}^2
      +
      c\delta
      \sum_{\mathcal{K}\in\mathcal{T}_h}
      \tau_\varepsilon
      \norm[0]{\partial_t{e_u^c}}_{\mathcal{K}}^2
      +
      c\delta
      \sum_{\mathcal{K}\in\mathcal{T}_h}
      \norm[0]{e_u^c}_{\mathcal{K}}^2
      \\
      &
      +
      \tfrac{c\delta}{2}
      \del[0]{T+\chi}
      \sum_{F_{\mathcal{R}}\subset\Omega_0}
      \norm[0]{\envert[0]{\tfrac{1}{2}\beta\cdot n}^{1/2}e_u^c
      }_{F_{\mathcal{R}}}^2.
    \end{split}
  \end{equation*}
  \textbf{Bound for $T_3$.} We find the following bound for $T_3$
  using the Cauchy--Schwarz inequality, the projection bound
  \cref{eq:otherusefulbnds_2}, the trace inequality
  \cref{eq:eg_inv_3}, the first bound in
  \cref{eq:derivativeofprojections}, and Young's inequality:
  \begin{equation*}
    T_3
    \le
    \tfrac{c}{\delta}
    \del[0]{T+\chi}
    \sum_{\mathcal{K}\in\mathcal{T}_{h}}
    \del[0]{\eta_{J,2,1}^{\mathcal{K}}}^2
    +
    c\delta
    \del[0]{T+\chi}
    \sum_{\mathcal{K}\in\mathcal{T}_{h}}
    \varepsilon
    \norm[0]{\overline{\nabla}e_u^c}_{\mathcal{K}}^2.
  \end{equation*}
  \textbf{Bound for $T_4$.} Using that
  $\langle \beta\cdot
  n\lambda_h,\del[0]{I-\Pi_h^{\mathcal{F}}}\del[0]{\varphi e_u^c}
  \rangle_{\partial\mathcal{T}_h^i}=0$ we start by writing $T_4$ as
  \begin{equation*}
      T_4
      =
      \langle
      \beta\cdot n\sbr[0]{\boldsymbol{u}_h},\del[0]{I-\Pi_h^{\mathcal{F}}}\del[0]{\varphi e_u^c}
      \rangle_{\partial\mathcal{T}_h^i}
      +
      \langle
      \del[0]{\beta_s-\beta\cdot n}\sbr[0]{\boldsymbol{u}_h}
      ,
      \del[0]{\Pi_h-\Pi_h^{\mathcal{F}}}
      \del[0]{\varphi e_u^c}
      \rangle_{\partial\mathcal{T}_h^i}.
  \end{equation*}
  Next, by a triangle inequality, using \cref{eq:betasinfmax}, that
  $|\beta_s - \tfrac{1}{2}\beta\cdot n|^{1/2} \le c$, and the
  Cauchy--Schwarz inequality,
  \begin{equation*}
    \begin{split}
      T_4
      \le
      &
      c
      \sum_{\mathcal{K}\in\mathcal{T}_h}
      \norm[0]{
        \envert[0]{\beta_s-\tfrac{1}{2}\beta\cdot n}^{1/2}
        \sbr[0]{\boldsymbol{u}_h}
      }_{\mathcal{Q}_{\mathcal{K}}}
      \del[1]{
        \norm[0]{
          \del[0]{I-\Pi_h}
          \del[0]{\varphi e_u^c}
        }_{\mathcal{Q}_{\mathcal{K}}}
        +
        \norm[0]{
          \del[0]{\Pi_h-\Pi_h^{\mathcal{F}}}
          \del[0]{\varphi e_u^c}
        }_{\mathcal{Q}_{\mathcal{K}}}
      }
      \\
      &\quad+
      c
      \sum_{\mathcal{K}\in\mathcal{T}_h}
      \norm[0]{
        \envert[0]{\beta_s-\tfrac{1}{2}\beta\cdot n}^{1/2}
        \sbr[0]{\boldsymbol{u}_h}
      }_{\mathcal{R}_{\mathcal{K}}}
      \del[1]{
        \norm[0]{
          \del[0]{I-\Pi_h}
          \del[0]{\varphi e_u^c}
        }_{\mathcal{R}_{\mathcal{K}}}
        +
        \norm[0]{
          \del[0]{\Pi_h-\Pi_h^{\mathcal{F}}}
          \del[0]{\varphi e_u^c}
        }_{\mathcal{R}_{\mathcal{K}}}
      }
      \\
      =:& T_{41} + T_{42}.
    \end{split}
  \end{equation*}
  For $T_{41}$ we use the quasi-interpolation estimate
  \cref{eq:local_quasi_tnorm_st_step_ratio_3}, the projection estimate
  \cref{eq:otherusefulbnds_2}, and Young's inequality to find:
  \begin{equation*}
      T_{41}
      \le
      \tfrac{c}{\delta}
      \del[0]{T+\chi}^2
      \sum_{\mathcal{K}\in\mathcal{T}_{h}}
      \del[0]{\eta_{J,2,2}^{\mathcal{K}}}^2
      +
      \tfrac{c\delta}{2}
      \del[0]{T+\chi}
      \sum_{\mathcal{K}\in\mathcal{T}_{h}}
      \varepsilon
      \norm[0]{\overline{\nabla}e_u^c}_{\mathcal{K}}^2
      +
      \tfrac{c\delta}{2}
      \sum_{\mathcal{K}\in\mathcal{T}_{h}}
      \tau_{\varepsilon}
      \norm[0]{\partial_te_u^c}_{\mathcal{K}}^2
      +
      \tfrac{c\delta}{2}
      \sum_{\mathcal{K}\in\mathcal{T}_{h}}
      \norm[0]{e_u^c}_{\mathcal{K}}^2.
  \end{equation*}
  For $T_{42}$ we use the quasi-interpolation estimate
  \cref{eq:local_quasi_tnorm_st_step_ratio_4}, the projection estimate
  \cref{eq:otherusefulbnds_3} using that
  $\delta t_{\mathcal{K}} = \mathcal{O}(h_K^2)$, and Young's
  inequality,
  \begin{equation*}
    T_{42}
      \le
      \tfrac{c}{\delta}
      \varepsilon^{-1}
      \del[0]{T+\chi}^2
      \sum_{\mathcal{K}\in\mathcal{T}_{h}}
      \del[0]{\eta_{J,3,\mathcal{R}}^{\mathcal{K}}}^2
      +
      \tfrac{c\delta}{2}
      \del[0]{T+\chi}
      \sum_{\mathcal{K}\in\mathcal{T}_{h}}
      \varepsilon
      \norm[0]{\overline{\nabla}e_u^c}_{\mathcal{K}}^2
      +
      \tfrac{c\delta}{2}
      \sum_{\mathcal{K}\in\mathcal{T}_{h}}
      \tau_{\varepsilon}
      \norm[0]{\partial_te_u^c}_{\mathcal{K}}^2
      +
      \tfrac{c\delta}{2}
      \sum_{\mathcal{K}\in\mathcal{T}_{h}}
      \norm[0]{e_u^c}_{\mathcal{K}}^2.
  \end{equation*}
  Combining the bounds for $T_{41}$ and $T_{42}$ we obtain:
  \begin{multline*}
      T_4
      \le
      \tfrac{c}{\delta}
      \del[0]{T+\chi}^2
      \sum_{\mathcal{K}\in\mathcal{T}_{h}}
      \del[0]{\eta_{J,2,2}^{\mathcal{K}}}^2
      +
      \tfrac{c}{\delta}
      \varepsilon^{-1}
      \del[0]{T+\chi}^2
      \sum_{\mathcal{K}\in\mathcal{T}_{h}}
      \del[0]{\eta_{J,3,\mathcal{R}}^{\mathcal{K}}}^2
      \\
      +
      c\delta
      \del[0]{T+\chi}
      \sum_{\mathcal{K}\in\mathcal{T}_{h}}
      \varepsilon
      \norm[0]{\overline{\nabla}e_u^c}_{\mathcal{K}}^2
      +
      c\delta
      \sum_{\mathcal{K}\in\mathcal{T}_{h}}
      \tau_{\varepsilon}
      \norm[0]{\partial_te_u^c}_{\mathcal{K}}^2
      +
      c\delta
      \sum_{\mathcal{K}\in\mathcal{T}_{h}}
      \norm[0]{e_u^c}_{\mathcal{K}}^2.
  \end{multline*}
  \textbf{Bound for $T_5$.} We write $T_5$ as
  \begin{equation*}
    T_5 =
    \underbrace{
      \del[0]{
        \varepsilon\overline{\nabla}
        \del[0]{I-\mathcal{I}_h^c}u_h,
        \overline{\nabla}\del[0]{\varphi e_u^c}
      }_{\mathcal{T}_h}
    }_{T_{51}}
    \underbrace{
      -
      \del[0]{
        \beta\del[0]{I-\mathcal{I}_h^c}u_h
        ,\nabla\del[0]{\varphi e_u^c}
      }_{\mathcal{T}_h}
    }_{T_{52}}.
  \end{equation*}
  For $T_{51}$ we use the Cauchy--Schwarz inequality, the inverse
  inequality \cref{eq:eg_inv_2}, the approximation estimate of the
  averaging operator \cref{eq:oswald_local_st}, that
  $\delta t_{\mathcal{K}} = \mathcal{O}(h_K^2)$, and Young's
  inequality to find
  \begin{equation*}
    T_{51}
      \le
      \tfrac{c}{2\delta}
      \del[0]{T+\chi}
      \sum_{\mathcal{K}\in\mathcal{T}_{h}}
      \del[1]{
        \del[0]{\eta_{J,2,1}^{\mathcal{K}}}^2
        +
        \del[0]{\eta_{J,3,\mathcal{R}}^{\mathcal{K}}}^2
      }
      +
      \tfrac{c\delta}{2}
      \del[0]{T+\chi}
      \sum_{\mathcal{K}\in\mathcal{T}_{h}}
      \varepsilon
      \norm[0]{\overline{\nabla}e_u^c}_{\mathcal{K}}^2.
  \end{equation*}
  For $T_{52}$ we first write
  \begin{equation*}
    T_{52}=
    \underbrace{-
    \del[0]{
      u_h-\mathcal{I}_h^cu_h
      ,
      \partial_t\del{\varphi{e_u^c}}
    }_{\mathcal{T}_h}
    }_{T_{521}}
    \underbrace{-
    \del[0]{
      \overline{\beta}{\del[0]{u_h-\mathcal{I}_h^cu_h}}
      ,
      \overline{\nabla}\del[0]{\varphi{e_u^c}}
    }_{\mathcal{T}_h}}_{T_{522}}.
  \end{equation*}
  We bound $T_{521}$ using the Cauchy--Schwarz inequality, the
  approximation estimate of the averaging operator
  \cref{eq:oswald_local_st}, and that
  $\delta t_{\mathcal{K}} = \mathcal{O}(h_K^2)$. We further note that
  on $\mathcal{T}_h^x$ we have that
  $\tilde{\varepsilon}^{-1/2} h_K^{-1/2} < h_K^{1/4} \varepsilon^{-1}$
  and on $\mathcal{T}_h^c$ we have that
  $\tilde{\varepsilon}^{-1/2}h_K^{-1/2} \le
  \varepsilon^{-1}h_K^{1/2}$. Therefore,
  \begin{equation*}
    \begin{split}
      T_{521}
      \le
      &c
      \sum_{\mathcal{K}\in\mathcal{T}_h}
      \del[1]{
        \sum_{F\subset\check{\mathcal{Q}}^i_{\mathcal{K}}}
        \sum_{\mathcal{K}'\subset\omega_F}
        \widetilde{\varepsilon}^{-1/2}
        h_K^{-1/2}
        \norm[0]{\sbr[0]{\boldsymbol{u}_h}}_{\mathcal{Q}_{\mathcal{K}'}}
      }
      \del[0]{
        \del[0]{T+\chi}
        \tau_{\varepsilon}^{1/2}
        \norm{
          \partial_t{{e_u^c}}
        }_\mathcal{K}
        +
        \norm{ {e_u^c} }_\mathcal{K}
      }
      \\
      &\quad
      +c
      \sum_{\mathcal{K}\in\mathcal{T}_h}
      \del[1]{
        \sum_{F\subset\check{\mathcal{R}}^i_{\mathcal{K}}}
        \sum_{\mathcal{K}'\subset\omega_F}
        \widetilde{\varepsilon}^{-1/2}
        \norm[0]{
          \envert[0]{\beta_s-\tfrac{1}{2}\beta\cdot n}^{1/2}
          \sbr[0]{\boldsymbol{u}_h}
        }_{\mathcal{R}_{\mathcal{K}'}}
      }
      \del[0]{
        \del[0]{T+\chi}
        \tau_{\varepsilon}^{1/2}
        \norm{
          \partial_t{{e_u^c}}
        }_\mathcal{K}
        +
        \norm{ {e_u^c} }_\mathcal{K}
      }
      \\
      \le&
      \tfrac{c}{\delta}
      \varepsilon^{-1}
      \del[0]{T+\chi}^2
      \sum_{\mathcal{K}\in\mathcal{T}_{h}}
      \del[0]{\eta_{J,2,1}^{\mathcal{K}}}^2
      +
      \tfrac{c}{\delta}
      \varepsilon^{-1}
      \del[0]{T+\chi}^2
      \sum_{\mathcal{K}\in\mathcal{T}_{h}}
      \del[0]{\eta_{J,2,2}^{\mathcal{K}}}^2
      \\
      &\quad
      +
      \tfrac{c}{\delta}
      \varepsilon^{-1}
      \del[0]{T+\chi}^2
      \sum_{\mathcal{K}\in\mathcal{T}_{h}}
      \del[0]{\eta_{J,3,\mathcal{R}}^{\mathcal{K}}}^2
      +
      2c\delta
      \sum_{\mathcal{K}\in\mathcal{T}_{h}}
      \tau_{\varepsilon}
      \norm[0]{\partial_te_u^c}_{\mathcal{K}}^2
      +
      2c\delta
      \sum_{\mathcal{K}\in\mathcal{T}_{h}}
      \norm[0]{e_u^c}_{\mathcal{K}}^2.
    \end{split}
  \end{equation*}
  For $T_{522}$, using the Cauchy--Schwarz inequality, the
  approximation estimate of the averaging operator
  \cref{eq:oswald_local_st}
  \begin{equation*}
      T_{522}
      \le
      \tfrac{c}{2\delta}
      \del[0]{T+\chi}
      \sum_{\mathcal{K}\in\mathcal{T}_{h}}
      \del[1]{
        \del[0]{\eta_{J,2,2}^{\mathcal{K}}}^2
        +
        \varepsilon^{-1}\del[0]{\eta_{J,3,\mathcal{R}}^{\mathcal{K}}}^2
      }
      +
      \tfrac{c\delta}{2}
      \del[0]{T+\chi}
      \sum_{\mathcal{K}\in\mathcal{T}_{h}}
      \varepsilon
      \norm[0]{\overline{\nabla}e_u^c}_{\mathcal{K}}^2.
  \end{equation*}
  Combining the bounds for $T_{521}$ and $T_{522}$ we find that
  \begin{equation*}
    \begin{split}
      T_{52} \le
      &
      \tfrac{c}{\delta}
      \varepsilon^{-1}
      \del[0]{T+\chi}^2
      \sum_{\mathcal{K}\in\mathcal{T}_{h}}
      \del[0]{\eta_{J,2,1}^{\mathcal{K}}}^2
      +
      \tfrac{c}{\delta}
      \del[0]{T+\chi}
      \sbr[2]{
        \varepsilon^{-1}
        \del[0]{T+\chi}
        +
        \tfrac{1}{2}
      }
      \sum_{\mathcal{K}\in\mathcal{T}_{h}}
      \del[0]{\eta_{J,2,2}^{\mathcal{K}}}^2
      \\
      &
      +
      \tfrac{c}{\delta}
      \varepsilon^{-1}
      \del[0]{T+\chi}
      \sbr[2]{
        \del[0]{T+\chi}
        +
        \tfrac{1}{2}
      }
      \sum_{\mathcal{K}\in\mathcal{T}_{h}}
      \del[0]{\eta_{J,3,\mathcal{R}}^{\mathcal{K}}}^2
      \\
      &
      +
      2c\delta
      \sum_{\mathcal{K}\in\mathcal{T}_{h}}
      \tau_{\varepsilon}
      \norm[0]{\partial_te_u^c}_{\mathcal{K}}^2
      +
      2c\delta
      \sum_{\mathcal{K}\in\mathcal{T}_{h}}
      \norm[0]{e_u^c}_{\mathcal{K}}^2
      +
      \tfrac{c\delta}{2}
      \del[0]{T+\chi}
      \sum_{\mathcal{K}\in\mathcal{T}_{h}}
      \varepsilon
      \norm[0]{\overline{\nabla}e_u^c}_{\mathcal{K}}^2.
    \end{split}
  \end{equation*}
  Combining the bounds for $T_{51}$ and $T_{52}$, we obtain:
  \begin{equation*}
    \begin{split}
      T_5 \le
      &
      \tfrac{c}{\delta}
      \del[0]{T+\chi}
      \sbr[2]{
        \varepsilon^{-1}
        \del[0]{T+\chi}
        +
        \tfrac{1}{2}
      }
      \sum_{\mathcal{K}\in\mathcal{T}_{h}}
      \del[0]{\eta_{J,2,1}^{\mathcal{K}}}^2
      +
      \tfrac{c}{\delta}
      \del[0]{T+\chi}
      \sbr[2]{
        \varepsilon^{-1}
        \del[0]{T+\chi}
        +
        \tfrac{1}{2}
      }
      \sum_{\mathcal{K}\in\mathcal{T}_{h}}
      \del[0]{\eta_{J,2,2}^{\mathcal{K}}}^2
      \\
      &
      +
      \tfrac{c}{\delta}
      \del[0]{T+\chi}
      \cbr[2]{
        \varepsilon^{-1}
        \sbr[1]{
          \del[0]{T+\chi}
          +
          \tfrac{1}{2}
        }
        +
        \tfrac{1}{2}
      }
      \sum_{\mathcal{K}\in\mathcal{T}_{h}}
      \del[0]{\eta_{J,3,\mathcal{R}}^{\mathcal{K}}}^2
      \\
      &
      +
      2c\delta
      \sum_{\mathcal{K}\in\mathcal{T}_{h}}
      \tau_{\varepsilon}
      \norm[0]{\partial_te_u^c}_{\mathcal{K}}^2
      +
      2c\delta
      \sum_{\mathcal{K}\in\mathcal{T}_{h}}
      \norm[0]{e_u^c}_{\mathcal{K}}^2
      +
      c\delta
      \del[0]{T+\chi}
      \sum_{\mathcal{K}\in\mathcal{T}_{h}}
      \varepsilon
      \norm[0]{\overline{\nabla}e_u^c}_{\mathcal{K}}^2.
    \end{split}
  \end{equation*}
  \textbf{Bound for $T_6$.} We write $T_6$ as follows:
  \begin{equation*}
    \begin{split}
      T_6
      =
      &
      \langle
      \del[0]{\beta_s-\beta\cdot n}\sbr[0]{\boldsymbol{u}_h}
      ,
      \del[0]{\Pi_h-\Pi_h^{\mathcal{F}}}
      \del[0]{\varphi e_u^c}
      \rangle_{\partial\mathcal{E}_D}
      \\
      &
      + \sbr[2]{
      \langle
      {\zeta^+}\beta\cdot n\del[0]{u_h-\mathcal{I}_h^cu_h},{\varphi e_u^c}
      \rangle_{\Omega_T}
      -
      \langle \zeta^+\beta\cdot n\sbr[0]{\boldsymbol{u}_h},
      \Pi_h^{\mathcal{F}}\del[0]{\varphi e_u^c}
      \rangle_{\Omega_T}}
      \\
      &
      + \big[ \langle
      {\zeta^+}\beta\cdot n\del[0]{u_h-\mathcal{I}_h^cu_h},{\varphi e_u^c}
      \rangle_{\partial\mathcal{E}_N\setminus\Omega_T}
      -
      \langle \zeta^+\beta\cdot n\sbr[0]{\boldsymbol{u}_h},
      \Pi_h^{\mathcal{F}}\del[0]{\varphi e_u^c}
      \rangle_{\partial\mathcal{E}_N\setminus\Omega_T}
      \\
      & \qquad +
      \langle
      \del[0]{\beta_s-\beta\cdot n}\sbr[0]{\boldsymbol{u}_h}
      ,
      \del[0]{\Pi_h-\Pi_h^{\mathcal{F}}}
      \del[0]{\varphi e_u^c}
      \rangle_{\partial\mathcal{E}_N \setminus (\Omega_T \cup \Omega_0)} \big]
      \\
      & +
      \langle
      \del[0]{\beta_s-\beta\cdot n}\sbr[0]{\boldsymbol{u}_h}
      ,
      \del[0]{\Pi_h-\Pi_h^{\mathcal{F}}}
      \del[0]{\varphi e_u^c}
      \rangle_{\Omega_0}
      \\
      =&: T_{61} + T_{62} + T_{63} + T_{64}.
    \end{split}
  \end{equation*}
  For $T_{61}$, we use the Cauchy--Schwarz inequality and the
  projection bound \cref{eq:otherusefulbnds_2}
  \begin{equation*}
    T_{61}
      \le
      \tfrac{c}{2\delta}
      \del[0]{T+\chi}
      \sum_{\mathcal{K}\in\mathcal{T}_{h}}
      \del[0]{\eta_{J,2,2}^{\mathcal{K}}}^2
      +
      \tfrac{c\delta}{2}
      \del[0]{T+\chi}
      \sum_{\mathcal{K}\in\mathcal{T}_{h}}
      \varepsilon
      \norm[0]{\overline{\nabla}e_u^c}_{\mathcal{K}}^2.
  \end{equation*}
  For $T_{62}$ we first write
  \begin{equation*}
    T_{62}
    =
    \underbrace{
      \langle
      {\zeta^+}\beta\cdot n
      \del[0]{u_h-\mathcal{I}_h^cu_h},{\varphi e_u^c}
      \rangle_{\Omega_T}
    }_{T_{621}}
    \underbrace{-
      \langle
      \zeta^+\beta\cdot n
      \sbr[0]{\boldsymbol{u}_h}, \Pi_h^{\mathcal{F}}
      \del[0]{\varphi e_u^c}
      \rangle_{\Omega_T}.
    }_{T_{622}}
  \end{equation*}
  We bound $T_{621}$, using the Cauchy--Schwarz inequality, the trace
  inequality \cref{eq:eg_inv_4}, the approximation estimate of the
  averaging operator \cref{eq:oswald_local_st}, that
  $\delta t_{\mathcal{K}}=\mathcal{O}(h_K^2)$, and Young's inequality:
  \begin{equation*}
    T_{621}
      \le
      \tfrac{c}{2\delta}
      \varepsilon^{-1}
      \del[0]{T+\chi}
      \sum_{\mathcal{K}\in\mathcal{T}_{h}}
      \del[0]{\eta_{J,2,1}^{\mathcal{K}}}^2
      +
      \tfrac{c}{2\delta}
      \del[0]{T+\chi}
      \sum_{\mathcal{K}\in\mathcal{T}_{h}}
      \del[0]{\eta_{J,3,\mathcal{R}}^{\mathcal{K}}}^2
      +
      c\delta
      \del[0]{T+\chi}
      \sum_{\mathcal{K} \in \mathcal{T}_h}
      \norm[0]{\envert[0]{\beta\cdot n}^{1/2}e_u^c}_{\mathcal{R}_{\mathcal{K}} \cap \Omega_T}^2.
  \end{equation*}
  Next, we bound $T_{622}$ using the Cauchy--Schwarz inequality, the
  boundedness of the projection operator $\Pi_h^{\mathcal{F}}$, and
  Young's inequality:
  \begin{equation*}
    T_{622}
      \le
      \tfrac{c}{2\delta}
      \del[0]{T+\chi}
      \sum_{\mathcal{K}\in\mathcal{T}_{h}}
      \del[0]{\eta_{J,3,\mathcal{R}}^{\mathcal{K}}}^2
      +
      \tfrac{c\delta}{2}
      \del[0]{T+\chi}
      \sum_{\mathcal{K}\in\mathcal{T}_{h}}
      \norm[0]{\envert[0]{\beta\cdot n}^{1/2}e_u^c}_{\mathcal{R}_{\mathcal{K}} \cap \Omega_T}^2.
  \end{equation*}
  Combining the bounds for $T_{621}$ and $T_{622}$ we find that
  \begin{equation*}
      T_{62} \le
      \tfrac{c}{2\delta}
      \varepsilon^{-1}
      \del[0]{T+\chi}
      \sum_{\mathcal{K}\in\mathcal{T}_{h}}
      \del[0]{\eta_{J,2,1}^{\mathcal{K}}}^2
      +
      \tfrac{c}{\delta}
      \del[0]{T+\chi}
      \sum_{\mathcal{K}\in\mathcal{T}_{h}}
      \del[0]{\eta_{J,3,\mathcal{R}}^{\mathcal{K}}}^2
      +
      \tfrac{3c\delta}{2}
      \del[0]{T+\chi}
      \sum_{F\subset\Omega_T}
      \norm[0]{\envert[0]{\beta\cdot n}^{1/2}e_u^c}_{F}^2.
  \end{equation*}
  For $T_{63}$ we write $T_{63} = T_{631} + T_{632} + T_{633}$ where
  \begin{align*}
      T_{631}
      &:=
      \langle
      {\zeta^+}\beta\cdot n\del[0]{u_h-\mathcal{I}_h^cu_h},{\varphi e_u^c}
      \rangle_{\partial\mathcal{E}_N\cap \mathcal{Q}_h},
      &
      T_{632}
      &:=
      -\langle \zeta^+\beta\cdot n\sbr[0]{\boldsymbol{u}_h},
      \Pi_h^{\mathcal{F}}\del[0]{\varphi e_u^c}
      \rangle_{\partial\mathcal{E}_N\cap \mathcal{Q}_h},
      \\
      T_{633}
      &:=
      \langle
      \del[0]{\beta_s-\beta\cdot n}\sbr[0]{\boldsymbol{u}_h}
      ,
      \del[0]{\Pi_h-\Pi_h^{\mathcal{F}}}
      \del[0]{\varphi e_u^c}
        \rangle_{\partial\mathcal{E}_N\cap \mathcal{Q}_h}.
        &&
  \end{align*}
  To bound $T_{631}$, we use the Cauchy--Schwarz inequality, the trace
  inequality \cref{eq:eg_inv_3}, the approximation estimate of the
  averaging operator \cref{eq:oswald_local_st}, and Young's
  inequality:
  \begin{equation*}
      T_{631}
      \le
      \tfrac{c}{2\delta}
      \varepsilon^{-1}
      \del[0]{T+\chi}
      \sum_{\mathcal{K}\in\mathcal{T}_{h}}
      \del[0]{\eta_{J,2,1}^{\mathcal{K}}}^2
      +
      \tfrac{c}{2\delta}
      \del[0]{T+\chi}
      \sum_{\mathcal{K}\in\mathcal{T}_{h}}
      \del[0]{\eta_{J,3,\mathcal{R}}^{\mathcal{K}}}^2
      +
      c\delta
      \del[0]{T+\chi}
      \sum_{F\in\partial\mathcal{E}_N}
      \norm[0]{\envert[0]{\beta\cdot n}^{1/2}e_u^c}_{F}^2.
  \end{equation*}
  For $T_{632}$, we first note that using similar steps as used to
  bound $M_{52}$ in the proof of \cref{thm:st_time_apos}, that
  $\beta_s \norm[0]{e_u^c }_{F}^2 \le \norm[0]{|\beta \cdot n|^{1/2}
    e_u^c }_{F}^2 + ch_K \norm[0]{e_u^c }_{F}^2$. Then, using the
  Cauchy--Schwarz inequality on $T_{632}$, \cref{eq:betasinfmax},
  boundedness of the projection operator $\Pi_h^{\mathcal{F}}$,
  \cref{eq:localtrace}, and Young's inequality we then find
  \begin{equation*}
    \begin{split}
      T_{632}
      \le&
      \tfrac{c}{\delta}
      \del[0]{T+\chi}^2
      \sum_{\mathcal{K}\in\mathcal{T}_{h}}
      \del[0]{\eta_{J,3,\mathcal{Q}}^{\mathcal{K}}}^2
      +
      \tfrac{c\delta}{2}
      \sum_{\mathcal{K}\in\mathcal{T}_h}
      \norm[0]{e_u^c}_{\mathcal{K}}^2
      +
      \tfrac{c\delta}{2}
      \del[0]{T+\chi}
      \sum_{F\subset\partial\mathcal{E}_N}
      \norm[0]{\envert[0]{\beta\cdot n}^{1/2}e_u^c}_{F}^2
      \\
      &
      +
      \tfrac{c}{2\delta}
      \del[0]{T+\chi}
      \sum_{\mathcal{K}\in\mathcal{T}_{h}}
      \del[0]{\eta_{J,2,2}^{\mathcal{K}}}^2
      +
      \tfrac{c\delta}{2}
      \del[0]{T+\chi}
      \sum_{\mathcal{K}\in\mathcal{T}_h}
      \varepsilon
      \norm[0]{\overline{\nabla}e_u^c}_{\mathcal{K}}^2.
    \end{split}
  \end{equation*}
  Next, we consider $T_{633}$. Using the Cauchy--Schwarz inequality,
  the projection estimate \cref{eq:otherusefulbnds_2}, and Young's
  inequality we find
  \begin{equation*}
    T_{633}
      \le
      \tfrac{c}{2\delta}
      \del[0]{T+\chi}
      \sum_{\mathcal{K}\in\mathcal{T}_{h}}
      \del[0]{\eta_{J,2,2}^{\mathcal{K}}}^2
      +
      \tfrac{c\delta}{2}
      \del[0]{T+\chi}
      \sum_{\mathcal{K}\in\mathcal{T}_{h}}
      \varepsilon
      \norm[0]{\overline{\nabla}e_u^c}_{\mathcal{K}}^2.
  \end{equation*}
  Combining the bounds for $T_{631}$, $T_{632}$, and $T_{633}$ we find
  that
  \begin{equation*}
    \begin{split}
      T_{63} \le
      &
      \tfrac{c}{2\delta}
      \varepsilon^{-1}
      \del[0]{T+\chi}
      \sum_{\mathcal{K}\in\mathcal{T}_{h}}
      \del[0]{\eta_{J,2,1}^{\mathcal{K}}}^2
      +
      \tfrac{c}{\delta}
      \del[0]{T+\chi}
      \sum_{\mathcal{K}\in\mathcal{T}_{h}}
      \del[0]{\eta_{J,2,2}^{\mathcal{K}}}^2
      +
      \tfrac{c}{\delta}
      \del[0]{T+\chi}^2
      \sum_{\mathcal{K}\in\mathcal{T}_{h}}
      \del[0]{\eta_{J,3}^{\mathcal{K}}}^2
      \\
      &
      +
      \tfrac{c\delta}{2}
      \sum_{\mathcal{K} \in \mathcal{T}_h}
      \norm[0]{e_u^c}_{\mathcal{K}}^2
      +
      c\delta
      \del[0]{T+\chi}
      \sum_{\mathcal{K}\in\mathcal{T}_h}
      \varepsilon
      \norm[0]{\overline{\nabla}e_u^c}_{\mathcal{K}}^2
      +
      \tfrac{3c\delta}{2}
      \del[0]{T+\chi}
      \sum_{F\subset\partial\mathcal{E}_N}
      \norm[0]{\envert[0]{\beta\cdot n}^{1/2}e_u^c}_{F}^2.
    \end{split}
  \end{equation*}
  For $T_{64}$, we use the Cauchy--Schwarz inequality, the projection
  estimate \cref{eq:otherusefulbnds_3}, and Young's inequality to find
  \begin{equation*}
    T_{64}
      \le
      \tfrac{c}{\delta}
      \varepsilon^{-1}
      \del[0]{T+\chi}^2
      \sum_{\mathcal{K}\in\mathcal{T}_{h}}
      \del[0]{\eta_{J,3,\mathcal{R}}^{\mathcal{K}}}^2
      +
      \tfrac{c\delta}{2}
      \sum_{\mathcal{K}\in\mathcal{T}_{h}}
      \tau_{\varepsilon}
      \norm[0]{\partial_te_u^c}_{\mathcal{K}}^2
      +
      \tfrac{c\delta}{2}
      \sum_{\mathcal{K}\in\mathcal{T}_{h}}
      \norm[0]{e_u^c}_{\mathcal{K}}^2.
  \end{equation*}
  Combining the bounds for $T_{61}$, $T_{62}$, $T_{63}$ and $T_{64}$,
  we obtain
  \begin{equation*}
    \begin{split}
      T_6 \le
      &
      \tfrac{c}{\delta} \varepsilon^{-1}
      \del[0]{T+\chi}
      \sum_{\mathcal{K}\in\mathcal{T}_{h}}
      \del[0]{\eta_{J,2,1}^{\mathcal{K}}}^2
      +
      \tfrac{3c}{2\delta}
      \del[0]{T+\chi}
      \sum_{\mathcal{K}\in\mathcal{T}_{h}}
      \del[0]{\eta_{J,2,2}^{\mathcal{K}}}^2
      +
      \tfrac{2c}{\delta}
      \varepsilon^{-1}
      \del[0]{T+\chi}^2
      \sum_{\mathcal{K}\in\mathcal{T}_{h}}
      \del[0]{\eta_{J,3,R}^{\mathcal{K}}}^2
      \\
      &+
      \tfrac{c}{\delta}
      \del[0]{T+\chi}^2
      \sum_{\mathcal{K}\in\mathcal{T}_{h}}
      \del[0]{\eta_{J,3}^{\mathcal{K}}}^2
      +
      \tfrac{c\delta}{2}
      \sum_{\mathcal{K}\in\mathcal{T}_{h}}
      \tau_{\varepsilon}
      \norm[0]{\partial_te_u^c}_{\mathcal{K}}^2
      +
      c\delta
      \sum_{\mathcal{K}\in\mathcal{T}_{h}}
      \norm[0]{e_u^c}_{\mathcal{K}}^2
      +
      \tfrac{3c\delta}{2}
      \del[0]{T+\chi}
      \sum_{\mathcal{K}\in\mathcal{T}_{h}}
      \varepsilon
      \norm[0]{\overline{\nabla}e_u^c}_{\mathcal{K}}^2
      \\
      &
      +
      3c\delta
      \del[0]{T+\chi}
      \sum_{F \subset \partial\mathcal{E}_N}
      \norm[0]{
        \envert[0]{\beta\cdot n}^{1/2}e_u^c
      }_{F}^2.
    \end{split}
  \end{equation*}
  With each of the terms $T_i$, $i=1,\hdots,6$ bounded, we now bound
  $\sum_{\mathcal{K} \in \mathcal{T}_h} \tau_{\varepsilon}
  \norm[0]{\partial_te_u^c}_{\mathcal{K}}^2$. By the triangle
  inequality and \cref{eq:timederivativeest},
  \begin{equation*}
      \sum_{\mathcal{K}\in\mathcal{T}_h}
      \tau_{\varepsilon}
      \norm[0]{\partial_te_u^c}_{\mathcal{K}}^2
      \le c
      T^2 \varepsilon^{-1} \eta^2
      +c
      \sum_{\mathcal{K}\in\mathcal{T}_h}
      \tau_{\varepsilon}
      \norm[0]{\partial_t(u_h-\mathcal{I}_h^cu_h)}_{\mathcal{K}}^2.
  \end{equation*}
  For the second term on the right hand side, using the inverse
  inequality \cref{eq:eg_inv_1}, the approximation estimate of the
  averaging operator \cref{eq:oswald_local_st}, H\"older's inequality
  for sums, and that $\delta t_{\mathcal{K}}=\mathcal{O}(h_K^2)$,
  \begin{equation}
    \label{eq:bound-for-oswald}
    \begin{split}
      &\sum_{\mathcal{K}\in\mathcal{T}_h}
      \tau_{\varepsilon}
      \norm[0]{\partial_t(u_h-\mathcal{I}_h^cu_h)}_{\mathcal{K}}^2
      \\
      &\qquad
      \le c
      \sum_{\mathcal{K}\in\mathcal{T}_h}
      \tilde{\varepsilon}
      \del[1]{
        \sum_{F\subset\check{\mathcal{Q}}^i_{\mathcal{K}}}
        h_K
        \delta t_{\mathcal{K}}^{-1}
        \norm[0]{\jump{u_h}}_F^2
        +
        \sum_{F\subset\check{\mathcal{R}}^i_{\mathcal{K}}}
        \norm[0]{\jump{u_h}}_F^2
      }
      \\
      &\qquad
      \le c
      \sum_{\mathcal{K}\in\mathcal{T}_h}
      \tilde{\varepsilon}
      \del[1]{
        \sum_{F\subset\check{\mathcal{Q}}^i_{\mathcal{K}}}
        \sum_{\mathcal{K}'\subset\omega_F}
        h_K^{-1}
        \norm[0]{\sbr[0]{\boldsymbol{u}_h}}_{\mathcal{Q}_{\mathcal{K}'}}^2
        +
        \sum_{F\subset\check{\mathcal{R}}^i_{\mathcal{K}}}
        \sum_{\mathcal{K}'\subset\omega_F}
        \norm[0]{
          \envert[0]{\beta_s-\tfrac{1}{2}\beta\cdot n}^{1/2}
          \sbr[0]{\boldsymbol{u}_h}
        }_{\partial\mathcal{K}'}
      }
      \\
      &\qquad
      \le
      c\varepsilon^{-1}
      \sum_{\mathcal{K}\in\mathcal{T}_h}
      \del[0]{\eta_{J,2,1}^{\mathcal{K}}}^2
      +
      c
      \sum_{\mathcal{K}\in\mathcal{T}_h}
      \del[0]{\eta_{J,3}^{\mathcal{K}}}^2.
    \end{split}
  \end{equation}
  Combining \cref{eq:err_stronger_stab_full,eq:bound-for-oswald} with
  the bounds for $T_1$ to $T_6$, we have
  \begin{equation*}
    \begin{split}
      &\chi
      \sum_{\mathcal{K}\in\mathcal{T}_h}
      \varepsilon
      \norm[0]{\overline{\nabla} e_u^c}_{\mathcal{K}}^2
      +
      \tfrac{1}{2}
      \sum_{\mathcal{K}\in\mathcal{T}_h}
      \norm{e_u^c}_{\mathcal{K}}^2
      +
      \tfrac{1}{2}\chi
      \sum_{F\subset\partial\mathcal{E}_N}
      \norm[0]{\envert[0]{\beta\cdot n}^{1/2}e_u^c}_{F}^2
      \\
      &\quad\le
      c\delta\del[0]{T+\chi}
      \sum_{\mathcal{K}\in\mathcal{T}_h}
      \varepsilon
      \norm[0]{\overline{\nabla} e_u^c}_{\mathcal{K}}^2
      +
      c\delta
      \sum_{\mathcal{K}\in\mathcal{T}_h}
      \norm{e_u^c}_{\mathcal{K}}^2
      +
      c\delta\del[0]{T+\chi}
      \sum_{F\subset\mathcal{E}_N}
      \norm[0]{\envert[0]{\beta\cdot n}^{1/2}e_u^c}_{F}^2
      \\
      &\qquad
      +
      \tfrac{c}{\delta}
      \del[0]{T+\chi}^2
      \sum_{\mathcal{K}\in\mathcal{T}_h}
      \del[0]{\eta_R^{\mathcal{K}}}^2
      +
      \tfrac{c}{\delta}
      \del[0]{T+\chi}^2
      \sum_{\mathcal{K}\in\mathcal{T}_h}
      \del[0]{\eta_{J,1}^{\mathcal{K}}}^2
      \\
      &\qquad
      +
      \tfrac{c}{\delta}
      \del[0]{T+\chi}^2
      \sum_{\mathcal{K}\in\mathcal{T}_h}
      \del[0]{\eta_{BC,1}^{\mathcal{K}}}^2
      +
      \tfrac{c}{\delta}
      \del[0]{T+\chi}
      \sum_{\mathcal{K}\in\mathcal{T}_h}
      \del[0]{\eta_{BC,2}^{\mathcal{K}}}^2
      \\
      &\qquad
      +
      c\del[2]{\tfrac{1}{\delta} + \delta}
      \del[0]{T+\chi}^2\varepsilon^{-1}
      \sum_{\mathcal{K}\in\mathcal{T}_h}
      \del[0]{\eta_{J,2,1}^{\mathcal{K}}}^2
      +
      \tfrac{c}{\delta}
      \del[0]{T+\chi}^2\varepsilon^{-1}
      \sum_{\mathcal{K}\in\mathcal{T}_h}
      \del[0]{\eta_{J,2,2}^{\mathcal{K}}}^2
      \\
      &\qquad
      +
      \tfrac{c}{\delta}
      \del[0]{T+\chi}^2\varepsilon^{-1}
      \sum_{\mathcal{K}\in\mathcal{T}_h}
      \del[0]{\eta_{J,3,\mathcal{R}}^{\mathcal{K}}}^2
      +
      c\del[2]{\tfrac{1}{\delta} + \delta}
      \del[0]{T+\chi}^2
      \sum_{\mathcal{K}\in\mathcal{T}_h}
      \del[0]{\eta_{J,3}^{\mathcal{K}}}^2.
    \end{split}
  \end{equation*}
  The result follows by choosing $\chi = T$ and $\delta = 1/(8c)$.
\end{proof}

We end this section by proving \cref{thm:reliability}.

\begin{proof}[Proof of \cref{thm:reliability}]
  Using the triangle inequality, Young's inequality, and
  \cref{eq:betasinfmax}, we have
  \begin{equation}
    \label{eq:proofofthm41_1}
    \begin{split}
      \tnorm{\boldsymbol{u}-\boldsymbol{u}_h}_{sT,h}^2
      \le & c
      \del[1]{
        T\sum_{\mathcal{K}\in\mathcal{T}_h}
        \varepsilon\norm[0]{\overline{\nabla}e_u^c}_\mathcal{K}^2
        +
        \sum_{\mathcal{K}\in\mathcal{T}_h}
        \norm[0]{e_u^c}_\mathcal{K}^2
        +
        T\sum_{F\subset\partial\mathcal{E}_N}
        \norm[0]{
          \envert[0]{
            \tfrac{1}{2}
            \beta\cdot{n}
          }^{1/2}
          e_u^c
        }_F^2
      }
      \\
      & +
      \sum_{\mathcal{K}\in\mathcal{T}_h}
      \tau_\varepsilon
      \norm[0]{\partial_t e_u}^2_\mathcal{K}
      +
      \sum_{\mathcal{K}\in\mathcal{T}_h}
      \del[1]{
        \del[0]{\eta_{J,2,1}^{\mathcal{K}}}^2
        +
        T\del[0]{\eta_{J,3}^{\mathcal{K}}}^2
      }
      +
      I_1 + I_2 + I_3,
    \end{split}
  \end{equation}
  where
  \begin{align*}
    I_1 &= cT\sum_{\mathcal{K}\in\mathcal{T}_h}
          \varepsilon\norm[0]{\overline{\nabla}(I-\mathcal{I}_h^c)u_h}_\mathcal{K}^2,
    &
      I_2 &= c\sum_{\mathcal{K}\in\mathcal{T}_h}
      \norm[0]{(I-\mathcal{I}_h^c)u_h}_\mathcal{K}^2,
    &
      I_3 &= cT\sum_{F\subset\partial\mathcal{E}_N}
            \norm[0]{
            \envert[0]{
            \tfrac{1}{2}
            \beta\cdot{n}
            }^{1/2}
            (I-\mathcal{I}_h^c)u_h
            }_F^2.
  \end{align*}
  Using the inverse inequality \cref{eq:eg_inv_2}, the approximation
  estimate of the averaging operator \cref{eq:oswald_local_st}, and
  that $\delta t_{\mathcal{K}} = \mathcal{O}(h_K^2)$, we bound $I_1$
  as follows:
  \begin{equation*}
    \begin{split}
      I_1
      & \le c T
      \sum_{\mathcal{K}\in\mathcal{T}_h}
      \varepsilon
      \del[1]{
        \sum_{F\subset\check{\mathcal{Q}}^i_{\mathcal{K}}}
        h_K^{-1}
        \norm[0]{\jump{u_h}}_F^2
        +
        \sum_{F\subset\check{\mathcal{R}}^i_{\mathcal{K}}}
        \norm[0]{\jump{u_h}}_F^2
      }
      \\
      & \le c T
      \sum_{\mathcal{K}\in\mathcal{T}_h}
      \del[1]{
        \sum_{F\subset\check{\mathcal{Q}}^i_{\mathcal{K}}}
        \sum_{\mathcal{K}'\subset\omega_F}
        {\varepsilon} h_K^{-1}
        \norm[0]{\sbr[0]{\boldsymbol{u}_h}}_{\mathcal{Q}_{\mathcal{K}'}}^2
        +
        \sum_{F\subset\check{\mathcal{R}}^i_{\mathcal{K}}}
        \sum_{\mathcal{K}'\subset\omega_F}
        \norm[0]{
          \envert[0]{\beta_s-\tfrac{1}{2}\beta\cdot n}^{1/2}
          \sbr[0]{\boldsymbol{u}_h}
        }_{\mathcal{R}_{\mathcal{K}'} }^2
      }
      \\
      & \le c T
      \sum_{\mathcal{K}\in\mathcal{T}_h}
      \del[1]{
        \del[0]{\eta_{J,2,1}^{\mathcal{K}}}^2
        +
        \del[0]{\eta_{J,3,\mathcal{R}}^{\mathcal{K}}}^2
      }.
    \end{split}
  \end{equation*}
  Using the approximation estimate of the averaging operator
  \cref{eq:oswald_local_st}, then similar to the bound of $I_1$ we
  have:
  \begin{equation*}
    I_2
      \le c
      \sum_{\mathcal{K}\in\mathcal{T}_h}
      \del[1]{
        \del[0]{\eta_{J,2,2}^{\mathcal{K}}}^2
        +
        \del[0]{\eta_{J,3,\mathcal{R}}^{\mathcal{K}}}^2
      }.
  \end{equation*}
  Finally, using the trace inequalities
  \cref{eq:eg_inv_3,eq:eg_inv_4}, the approximation estimate of the
  averaging operator \cref{eq:oswald_local_st}, and that
  $\delta t_{\mathcal{K}} = \mathcal{O}(h_K^2)$, we can bound $I_3$ as
  follows:
  \begin{equation*}
      I_3
      \le  c T
      \sum_{\mathcal{K}\in\mathcal{T}_h}
      \del[1]{
        \varepsilon^{-1}
        \del[0]{\eta_{J,2,1}^{\mathcal{K}}}^2
        +
        \del[0]{\eta_{J,3,\mathcal{R}}^{\mathcal{K}}}^2
      }.
  \end{equation*}
  Combining the bounds for $I_1$, $I_2$, and $I_3$ with
  \cref{eq:proofofthm41_1} we find that
  \begin{multline*}
    \tnorm{\boldsymbol{u}-\boldsymbol{u}_h}_{sT,h}^2
    \le c
    \del[1]{
      T\sum_{\mathcal{K}\in\mathcal{T}_h}
      \varepsilon\norm[0]{\overline{\nabla}e_u^c}_\mathcal{K}^2
      +
      \sum_{\mathcal{K}\in\mathcal{T}_h}
      \norm[0]{e_u^c}_\mathcal{K}^2
      +
      T\sum_{F\subset\partial\mathcal{E}_N}
      \norm[0]{
        \envert[0]{
          \tfrac{1}{2}
          \beta\cdot{n}
        }^{1/2}
        e_u^c
      }_F^2
    }
    \\
    +
    \sum_{\mathcal{K}\in\mathcal{T}_h}
    \tau_\varepsilon
    \norm[0]{\partial_t e_u}^2_\mathcal{K}
    + c
    \sum_{\mathcal{K}\in\mathcal{T}_h}
    \del[0]{\eta_{J,2,2}^{\mathcal{K}}}^2
    + c T
    \sum_{\mathcal{K}\in\mathcal{T}_h}
    \del[1]{
      \varepsilon^{-1}
      \del[0]{\eta_{J,2,1}^{\mathcal{K}}}^2
      +
      \del[0]{\eta_{J,3}^{\mathcal{K}}}^2
    }.
  \end{multline*}
  By \cref{lem:upperbndwitheuc} this is further bound as:
  \begin{equation*}
    \begin{split}
      \tnorm{\boldsymbol{u}-\boldsymbol{u}_h}_{sT,h}^2
      \le &
      c \sum_{\mathcal{K}\in\mathcal{T}_h}
      \bigg(
      T^2\del[0]{\eta_R^{\mathcal{K}}}^2
      +
      T^2\del[0]{\eta_{J,1}^{\mathcal{K}}}^2
      +
      T^2\varepsilon^{-1}\del[0]{\eta_{J,2,1}^{\mathcal{K}}}^2
      +
      T^2\varepsilon^{-1}\del[0]{\eta_{J,2,2}^{\mathcal{K}}}^2
      \\
      &\qquad +
      T^2\varepsilon^{-1}\del[0]{\eta_{J,3,\mathcal{R}}^{\mathcal{K}}}^2
      +
      T^2\del[0]{\eta_{J,3}^{\mathcal{K}}}^2
      +
      T^2\del[0]{\eta_{BC,1}^{\mathcal{K}}}^2
      +
      T\del[0]{\eta_{BC,2}^{\mathcal{K}}}^2
      \bigg)
      +
      \sum_{\mathcal{K}\in\mathcal{T}_h}
      \tau_\varepsilon
      \norm[0]{\partial_t e_u}^2_\mathcal{K}
      \\
      \le &
      c T^2\varepsilon^{-1} \eta^2
      + \sum_{\mathcal{K}\in\mathcal{T}_h}
      \tau_\varepsilon
      \norm[0]{\partial_t e_u}^2_\mathcal{K}.
    \end{split}
  \end{equation*}
  We conclude \cref{eq:reliability} using \cref{thm:st_time_apos}.
\end{proof}

\subsection{Local efficiency of the error estimator}
\label{ss:eff}

In this section we prove \cref{thm:efficiency}. Given any space-time
element $\mathcal{K}$, we introduce the element bubble function
$\psi_{\mathcal{K}}=c_{\theta}\Pi_{i=1}^{2^{d+1}}\theta_{\mathcal{K},i}$,
where $\theta_{\mathcal{K},i}$ denotes the linear Lagrangian basis
polynomial associated with the $i$-th vertex of $\mathcal{K}$, and the
constant factor $c_{\theta}$ is such that
$\norm[0]{\psi_{\mathcal{K}}}_{L^\infty(\mathcal{K})} = 1$. We observe
that $\del[0]{\psi_{\mathcal{K}}}|_{\partial\mathcal{K}} = 0$. Given
any $v\in V_h$, the element bubble function satisfies the following
estimates (see \cite[Lemma 3.3]{Verfurth:1998} and \cite[Lemma
3.6]{Verfurth:2005}):
\begin{equation}
  \label{eq:bubbleestelem}
  \norm[0]{\psi_{\mathcal{K}}v}_{\mathcal{K}}
  \le c\norm[0]{v}_{\mathcal{K}},
  \qquad
  c\norm[0]{v}_{\mathcal{K}}^2
  \le \del[0]{v,\psi_{\mathcal{K}}v}_{\mathcal{K}},
\end{equation}

We also need facet bubble functions. For an element ${\mathcal{K}}$
and one of its $\mathcal{Q}$-facets $F\in\mathcal{Q}_{\mathcal{K}}$, we
first transform to the reference domain and consider
$\widehat{\mathcal{K}}=\Phi_{\mathcal{K}}^{-1}(\mathcal{K})$ and
$\widehat{F}=\Phi_{\mathcal{K}}^{-1}(F)$. Without loss of generality,
we let $\hat{x}_i$ denote the spatial coordinate such that
$\hat{x}_i\equiv -1$ on $\widehat{F}$. Given any number
$\kappa\in(0,1]$, we denote by $\Psi_\kappa$ the mapping from
$(\hat{t},\hat{x}_1,\dots,\hat{x}_i,\dots,\hat{x}_d)$ to
$(\hat{t},\hat{x}_1,\dots,\kappa(\hat{x}_i+1)-1,\dots,\hat{x}_d)$ and
we let
$\widehat{\mathcal{K}}_\kappa:=\Psi_\kappa(\widehat{\mathcal{K}})$.
We introduce the following facet bubble function
\begin{equation*}
  \hat{\psi}_{\mathcal{K},F,\kappa}
  =
  \begin{cases}
    c_{\theta,F}
    \Pi_{i=1}^{2^d}\hat{\theta}_{\mathcal{K},F,i,\kappa}
    &\text{ on }
    \widehat{\mathcal{K}}_\kappa,
    \\
    0
    &\text{ on }
    \widehat{\mathcal{K}}\setminus\widehat{\mathcal{K}}_\kappa,
  \end{cases}
\end{equation*}
where $\hat{\theta}_{\mathcal{K},F,i,\kappa}$ denotes the linear
Lagrangian basis polynomial associated with the $i$-th vertex of
$\widehat{\mathcal{K}}_\kappa$ that is also on $\widehat{F}$.
Similarly, the constant factor $c_{\theta,F}$ is such that
$\norm[0]{\hat{\psi}_{\mathcal{K},F,\kappa}}_{L^\infty(\widehat{F})}
= 1$.

Furthermore, given any $\mu\in M_h$ and considering
$\hat{\mu}=\mu\circ\Phi_{\mathcal{K}}$, we have the following
estimates:
\begin{equation}
  \label{eq:bubbleestface-ref}
  \begin{aligned}
    &
    \norm[0]{
      \hat{\psi}_{\mathcal{K},F,\kappa}\hat{\mu}
    }_{\widehat{F}}
    \le c
    \norm[0]{\hat{\mu}}_{\widehat{F}},
    &
    &c\norm[0]{\hat{\mu}}_{\widehat{F}}^2
    \le
    \langle{\hat{\mu}
      ,\hat{\psi}_{\mathcal{K},F,\kappa}\hat{\mu}}
    \rangle_{\widehat{F}},
    \\
    &\norm[0]{
      \hat{\psi}_{\mathcal{K},F,\kappa}\hat{\mu}
    }_{\widehat{\mathcal{K}}}
    \le c \kappa^{1/2}
    \norm[0]{\hat{\mu}}_{\widehat{F}},
    &
    &\norm[0]{
      \widehat{\overline{\nabla}
      }\hat{\psi}_{\mathcal{K},F,\kappa}\hat{\mu}}_{\widehat{\mathcal{K}}}
    \le c
    \kappa^{-1/2}
    \norm[0]{\hat{\mu}}_{\widehat{F}},
  \end{aligned}
\end{equation}
where the first estimate is a result of
$\norm[0]{\hat{\psi}_{\mathcal{K},F,\kappa}}_{L^\infty(\widehat{F})} =
1$ and the remaining estimates are shown in \cite[Lemma
3.4]{Verfurth:1998b}.

We remark that the facet function $\mu$ in \cref{eq:bubbleestface-ref}
is continued to functions on elements using the continuation operator
defined in \cite{Verfurth:1998}. Furthermore,
\cref{eq:bubbleestelem,eq:bubbleestface-ref} are proven in
\cite{Verfurth:1998,Verfurth:1998b,Verfurth:2005} on $n$-simplices and
parallelepipeds, with $n \ge 2$. However, these inequalities also hold
for our mesh due to the assumptions on $\phi_{\mathcal{K}}$
\cref{eq:diffeom_jac} resulting in a Jacobian bounded independent of
$h_K$ and $\delta t_{\mathcal{K}}$.

To define the facet bubble function on $\omega_F$, we consider
three cases:
\begin{enumerate}[label={Case \arabic*}]
\item $\mathcal{K}_{nb}$, the neighboring element of $\mathcal{K}$
  across $F$, is at the same refinement element as
  $\mathcal{K}$. \label{cases-1}
  \vspace{-0.5em}
\item The $2^d$ neighboring elements of $\mathcal{K}$ across $F$,
  denoted by $\mathcal{K}_{nb,i}$ with $i=1,\dots,2^d$, are finer.
  \label{cases-2}
  \vspace{-0.5em}
\item The neighboring element of $\mathcal{K}$ with respect to
	facet
  $F$ is coarser, which is denoted by
  $\mathcal{K}_{nb,0}$. \label{cases-3}
\end{enumerate}
For \ref{cases-1}, we let
\begin{equation}
  \label{eq:facebubble}
  \psi_{F,\kappa}:=
  \begin{cases}
    \hat{\psi}_{\mathcal{K},F,\kappa}\circ \Phi_{\mathcal{K}}^{-1}
    &\text{ on }\mathcal{K},
    \\
    \hat{\psi}_{\mathcal{K}_{nb},F,\kappa}\circ
    \Phi_{\mathcal{K}_{nb}}^{-1}
    &\text{ on }\mathcal{K}_{nb}.
  \end{cases}
\end{equation}
For \ref{cases-2}, we consider the refinement of
$\mathcal{K}:=\cup_{i=1}^{2^d}\mathcal{K}_i$ such that $F$ is refined
to the set of $\cbr[0]{F_i}_{i=1}^{2^d}$ where
$F_i = \mathcal{Q}_{\mathcal{K}_i} \cap
\mathcal{Q}_{\mathcal{K}_{nb,i}}$. We further denote by $\omega_{F_i}$
the union of $\mathcal{K}_i$ and $\mathcal{K}_{nb,i}$ and define a
$\psi_{F,\kappa,i}$ on $\omega_{F_i}$ as in \cref{eq:facebubble} on
each $F_i$.

For \ref{cases-3}, we consider the coarsest refinement of
$\mathcal{K}_{nb,0}$ such that one of the refined elements
$\mathcal{K}_{nb}$ has the property that
$F = \mathcal{Q}_{\mathcal{K}} \cap
\mathcal{Q}_{\mathcal{K}_{nb}}$. We denote the union of
$\mathcal{K}$ and $\mathcal{K}_{nb}$ by $\omega_{F,*}$. Then,
$\psi_{F,\kappa}$ is defined on $\omega_{F,*}$ as in
\cref{eq:facebubble}.

Applying the scaling arguments \cref{eq:lb7sudir,eq:scalingequiv} to
\cref{eq:bubbleestface-ref}, using the definition of $\psi_{F,\kappa}$
described above, choosing
$\kappa = \tilde{\varepsilon}^{1/2}\varepsilon^{1/2}$, and dropping
the subscript $\kappa$ from $\psi_{F,\kappa}$, we obtain the following
estimates:
\begin{equation}
  \label{eq:bubbleestface}
  \begin{aligned}
    &
    \norm[0]{\psi_F\mu}_{F}
    \le c
	 \norm[0]{\mu}_{F},
	 &
	 &
	 c\norm[0]{\mu}_{F}^2
    \le
    \langle{\mu,\psi_{F}\mu}\rangle_{F},
    \\
    &
	 \norm[0]{\psi_{F}\mu}_{\omega_F}
    \le c h_K^{1/2}
	 \tilde{\varepsilon}^{1/4} \varepsilon^{1/4}
    \norm[0]{\mu}_{F},
	 &
	 &\norm[0]{\overline{\nabla}\psi_{F}\mu}_{\omega_F}
    \le c
	 h_K^{-1/2}
	 \tilde{\varepsilon}^{-1/4} \varepsilon^{-1/4}
    \norm[0]{\mu}_{F}.
  \end{aligned}
\end{equation}

With the above bubble functions defined, we proceed with proving
\cref{thm:efficiency}.

\begin{proof}[Proof of \cref{thm:efficiency}]
  Each term of $\eta^{\mathcal{K}}$ will be bound separately. However,
  let us first note that since $\eta_{J,2}^{\mathcal{K}}$ and
  $\eta_{J,3}^{\mathcal{K}}$ are part of
  $\tnorm{\boldsymbol{u} - \boldsymbol{u}_h}_{sT,h,\mathcal{K}}$,
  these terms are trivially bounded.

  \textbf{Bound for $\eta_R^{\mathcal{K}}$.} By the triangle
  inequality and Young's inequality,
  \begin{equation}
    \label{eq:RhKtriangleineq}
    \norm[0]{R_h^{\mathcal{K}}}_{\mathcal{K}}^2
    \le
    2\norm[0]{\Pi_h R_h^{\mathcal{K}}}_{\mathcal{K}}^2
    +
    2\norm[0]{(I-\Pi_h)R_h^{\mathcal{K}}}_{\mathcal{K}}^2.
  \end{equation}
  We bound the first term on the right hand side. Using estimate
  \cref{eq:bubbleestelem}, with $c_1$ and $c_2$ the constants in the
  first and second inequalities of \cref{eq:bubbleestelem},
  respectively, the Cauchy--Schwarz inequality, and Young's
  inequality with constant $c_1$, we note that
  \begin{equation}
    \label{eq:PihRhKbounds}
    \tfrac{c_1}{2}\norm[0]{\Pi_hR_h^{\mathcal{K}}}_{\mathcal{K}}^2
    \le
    \del[0]{R_h^{\mathcal{K}},\psi_{\mathcal{K}}\Pi_hR_h^{\mathcal{K}}}_{\mathcal{K}}
    +
    \tfrac{c_2^2}{2c_1}
    \norm[0]{(I-\Pi_h)R_h^{\mathcal{K}}}_{\mathcal{K}}^2.
  \end{equation}
  Combining \cref{eq:RhKtriangleineq,eq:PihRhKbounds}, and using the
  boundedness of the projection $\Pi_h$ so that
  $\norm[0]{(I-\Pi_h)R_h^{\mathcal{K}}}_{\mathcal{K}}^2 \le c
  \norm[0]{(I-\Pi_h)R_h^{\mathcal{K}}}_{\mathcal{K}}
  \norm[0]{R_h^{\mathcal{K}}}_{\mathcal{K}}$, we obtain
  \begin{equation}
    \label{eq:lowerbnd2}
    \lambda_{\mathcal{K}}\norm[0]{R_h^{\mathcal{K}}}_{\mathcal{K}}^2
    \le
    c \lambda_{\mathcal{K}} \del[0]{R_h^{\mathcal{K}},\psi_{\mathcal{K}}\Pi_hR_h^{\mathcal{K}}}_{\mathcal{K}}
    +
    c \lambda_{\mathcal{K}} \norm[0]{(I-\Pi_h)R_h^{\mathcal{K}}}_{\mathcal{K}}
    \norm[0]{R_h^{\mathcal{K}}}_{\mathcal{K}}.
  \end{equation}
  To bound the first term on the right hand side of
  \cref{eq:lowerbnd2}, we use the definition of $R_h^{\mathcal{K}}$,
  integrate by parts, and use that $\nabla \cdot \beta=0$, to find for
  any $z\in H_0^1(\mathcal{K})$,
  \begin{equation}
    \label{eq:lowerbnd1}
      \del[0]{R_h^{\mathcal{K}},z}_{\mathcal{K}}
      =
      \del[0]{\varepsilon^{1/2}\overline{\nabla}(u-u_h),\varepsilon^{1/2}\overline{\nabla}z}_{\mathcal{K}}
      +
      \del[0]{\overline{\beta}\cdot\overline{\nabla}(u-u_h),z}_{\mathcal{K}}
      +
      \del[0]{\partial_t(u-u_h),z}_{\mathcal{K}}.
  \end{equation}
  Choosing $z = \psi_{\mathcal{K}}\Pi_hR_h^{\mathcal{K}}$, we bound
  each term on the right hand side of \cref{eq:lowerbnd1}
  separately. Using the Cauchy--Schwarz inequality, the inequality
  \cref{eq:eg_inv_2}, estimate \cref{eq:bubbleestelem}, and
  boundedness
  of the projection $\Pi_h$, we obtain:
  \begin{subequations}
    \label{eq:septermsbndsRhK}
    \begin{align}
      \del[0]{\varepsilon^{1/2}\overline{\nabla}(u-u_h),
      \varepsilon^{1/2}\overline{\nabla}(\psi_{\mathcal{K}}\Pi_hR_h^{\mathcal{K}})}_{\mathcal{K}}
      \le&
           c
           \varepsilon^{1/2}\norm[0]{\overline{\nabla}(u-u_h)}_{\mathcal{K}}
           \varepsilon^{1/2}h_K^{-1}\norm[0]{R_h^{\mathcal{K}}}_{\mathcal{K}},
      \\
      \del[0]{\overline{\beta}\cdot\overline{\nabla}(u-u_h),\psi_{\mathcal{K}}\Pi_hR_h^{\mathcal{K}}}_{\mathcal{K}}
      \le & c
            \varepsilon^{1/2}\norm[0]{\overline{\nabla}(u-u_h)}_{\mathcal{K}}
            \varepsilon^{-1/2}\norm[0]{R_h^{\mathcal{K}}}_{\mathcal{K}},
      \\
      \del[0]{\partial_t(u-u_h),\psi_{\mathcal{K}}\Pi_hR_h^{\mathcal{K}}}_{\mathcal{K}}
      \le& c
           \tau_{\varepsilon}^{1/2}\norm[0]{\partial_t(u-u_h)}_{\mathcal{K}}
           \tau_{\varepsilon}^{-1/2}\norm[0]{R_h^{\mathcal{K}}}_{\mathcal{K}}.
    \end{align}
  \end{subequations}
  From \cref{eq:lowerbnd1} with
  $z = \psi_{\mathcal{K}}\Pi_hR_h^{\mathcal{K}}$ and
  \cref{eq:septermsbndsRhK} we therefore obtain:
  \begin{equation}
    \label{eq:RhboundtermszpiRh}
      \del[0]{R_h^{\mathcal{K}},\psi_{\mathcal{K}}\Pi_hR_h^{\mathcal{K}}}_{\mathcal{K}}
      \le
      c \bigg(
      (\varepsilon^{1/2}h_K^{-1} + \varepsilon^{-1/2}) \varepsilon^{1/2}\norm[0]{\overline{\nabla}(u-u_h)}_{\mathcal{K}}
      +
      \norm[0]{\partial_t(u-u_h)}_{\mathcal{K}}
      \bigg)\norm[0]{R_h^{\mathcal{K}}}_{\mathcal{K}}.
  \end{equation}
  Using that $\delta t_{\mathcal{K}} = \mathcal{O}(h_K^2)$ we note
  that
  $\lambda_{\mathcal{K}} (\varepsilon^{1/2}h_K^{-1} +
  \varepsilon^{-1/2}) < c \tilde{\varepsilon}^{-1/2}
  \varepsilon^{-1/2}$. Therefore, multiplying both sides of
  \cref{eq:RhboundtermszpiRh} by $\lambda_{\mathcal{K}}$, we find
  \begin{equation}
    \label{eq:RhboundtermszpiRhF}
    \lambda_{\mathcal{K}}\del[0]{R_h^{\mathcal{K}},\psi_{\mathcal{K}}\Pi_hR_h^{\mathcal{K}}}_{\mathcal{K}}
    \le
    c \varepsilon^{-1/2} \tilde{\varepsilon}^{-1/2}
    \del[1]{
      \varepsilon^{1/2}\norm[0]{\overline{\nabla}(u-u_h)}_{\mathcal{K}}
      +
      \tau_{\varepsilon}^{1/2}\norm[0]{\partial_t(u-u_h)}_{\mathcal{K}}
      }
      \norm[0]{R_h^{\mathcal{K}}}_{\mathcal{K}}.
  \end{equation}
  Combining \cref{eq:lowerbnd2,eq:RhboundtermszpiRhF}, and using the
  definitions of $\tnorm{\boldsymbol{u}-\boldsymbol{u}_h}_{sT,h,\mathcal{K}}$,
  $\eta_R^{\mathcal{K}}$ and $\mathrm{osc}_h^{\mathcal{K}}$,
  \begin{equation*}
    \eta_R^{\mathcal{K}}
    \le
    c \varepsilon^{-1/2} \tilde{\varepsilon}^{-1/2}
	 \tnorm{\boldsymbol{u}-\boldsymbol{u}_h}_{sT,h,\mathcal{K}}
    +
    c\, \mathrm{osc}_h^{\mathcal{K}}.
  \end{equation*}

  \textbf{Bound for $\eta_{J,1}^{\mathcal{K}}$.} Let $F$ be a facet
  such that
  $F\subset\mathcal{Q}_{\mathcal{K}} \setminus
  \partial\mathcal{E}$. To bound $\eta_{J,1}^{\mathcal{K}}$ we
  consider separately \ref{cases-1}, \ref{cases-2}, and \ref{cases-3}.

  \textbf{\ref{cases-1}.} For any $F\subset\mathcal{Q}_{\mathcal{K}}$ and
  $z\in H^1_0(\omega_{F})$, using integration by parts, we have
  \begin{equation*}
      \langle
      \varepsilon\jump{\overline{\nabla}_{\overline{n}}u_h},z
      \rangle_F
      =
      -\del[0]{\varepsilon^{1/2}\overline{\nabla}(u-u_h),\varepsilon^{1/2}\overline{\nabla}z}_{\omega_{F}}
      -
      \del[0]{\partial_t(u-u_h),z}_{\omega_{F}}
      -
      \del[0]{\overline{\beta}\cdot\overline{\nabla}(u-u_h),z}_{\omega_{F}}
      +
      \del[0]{R_h^{\mathcal{K}},z}_{\omega_{F}}.
  \end{equation*}
  Choosing
  $z = \psi_F \varepsilon \jump{ \overline{\nabla}_{\overline{n}} u_h
  }$, using \cref{eq:bubbleestface}, and the Cauchy--Schwarz
  inequality, we obtain
  \begin{multline}
    \label{eq:spatialgradjump1}
    h_K^{1/2}\varepsilon^{1/2}
    \norm[0]{\jump{\overline{\nabla}_{\overline{n}}u_h}}_{F}
    \le c
    \varepsilon^{-1/4}
    \tilde{\varepsilon}^{-1/4}
    \varepsilon^{1/2}
    \norm[0]{\overline{\nabla}(u-u_h)}_{\omega_F}
    \\
    +
    h_K\varepsilon^{-1/2}
    \varepsilon^{1/4}
    \tilde{\varepsilon}^{1/4}
    \del[1]{
      \tau_{\varepsilon}^{-1/2}
      \tau_{\varepsilon}^{1/2}
      \norm[0]{\partial_t(u-u_h)}_{\omega_F}
      +
      \varepsilon^{-1/2}
      \varepsilon^{1/2}
      \norm[0]{\overline{\nabla}(u-u_h)}_{\omega_F}
      +
      \lambda_{\mathcal{K}}^{-1}
      \lambda_{\mathcal{K}}
      \norm[0]{R_h^{\mathcal{K}}}_{\omega_F}
    }.
  \end{multline}
  Using $\delta t_{\mathcal{K}}=\mathcal{O}(h_K^2)$,
  $h_K\varepsilon^{-1/2}\tilde{\varepsilon}^{1/2}\le 1$ and
  $ \varepsilon^{-1/4} \tilde{\varepsilon}^{1/4}
  \max\cbr[0]{h_K,\varepsilon^{1/2}}\le 1$, we find
  \begin{equation*}
    h_K^{1/2}\varepsilon^{1/2}
    \norm[0]{\jump{\overline{\nabla}_{\overline{n}}u_h}}_{F}
    \le
    c\sum_{\mathcal{K}\subset\omega_F}
    \varepsilon^{-1/4}
    \tilde{\varepsilon}^{-1/4}
    \tnorm{\boldsymbol{u}-\boldsymbol{u}_h}_{sT,h,\mathcal{K}}
    +
    \lambda_{\mathcal{K}}\norm[0]{R_h^{\mathcal{K}}}_{\omega_F}.
  \end{equation*}

  \textbf{\ref{cases-2}.} Identical steps as in \ref{cases-1} gives
  \begin{equation*}
    h_K^{1/2}\varepsilon^{1/2}
    \norm[0]{\jump{\overline{\nabla}_{\overline{n}}u_h}}_{F_i}
    \le c
    \varepsilon^{-1/4}
    \tilde{\varepsilon}^{-1/4}
    \varepsilon^{1/2}
    \norm[0]{\overline{\nabla}(u-u_h)}_{\omega_{F_i}}
    +
    c
    \varepsilon^{1/4}
    \tilde{\varepsilon}^{1/4}
    \tau_{\varepsilon}^{1/2}
    \norm[0]{\partial_t(u-u_h)}_{\omega_{F_i}}
    +
    \lambda_{\mathcal{K}}
    \norm[0]{R_h^{\mathcal{K}}}_{\omega_{F_i}}.
  \end{equation*}
  Summing over all $F_i$'s,
  \begin{equation*}
    h_K^{1/2}\varepsilon^{1/2}
    \norm[0]{\jump{\overline{\nabla}_{\overline{n}}u_h}}_{F}
    \le
	 c\sum_{\mathcal{K}\subset\omega_{F}}
    \varepsilon^{-1/4}
    \tilde{\varepsilon}^{-1/4}
    \tnorm{\boldsymbol{u}-\boldsymbol{u}_h}_{sT,h,\mathcal{K}}
    +
	 \lambda_{\mathcal{K}}\norm[0]{R_h^{\mathcal{K}}}_{\omega_{F}}.
  \end{equation*}

  \textbf{\ref{cases-3}.} Identical steps as in \ref{cases-1} gives
  \begin{equation*}
    h_K^{1/2}\varepsilon^{1/2}
    \norm[0]{\jump{\overline{\nabla}_{\overline{n}}u_h}}_{F}
    \le c
    \varepsilon^{-1/4}
    \tilde{\varepsilon}^{-1/4}
    \varepsilon^{1/2}
    \norm[0]{\overline{\nabla}(u-u_h)}_{\omega_{F,*}}
    +
    c
    \varepsilon^{1/4}
    \tilde{\varepsilon}^{1/4}
    \tau_{\varepsilon}^{1/2}
    \norm[0]{\partial_t(u-u_h)}_{\omega_{F,*}}
    +
    \lambda_{\mathcal{K}}
    \norm[0]{R_h^{\mathcal{K}}}_{\omega_{F,*}}.
  \end{equation*}
  Since $\omega_{F,*} \subset \omega_F$, we then find
  \begin{equation*}
    h_K^{1/2}\varepsilon^{1/2}
    \norm[0]{\jump{\overline{\nabla}_{\overline{n}}u_h}}_{F}
    \le c
    \sum_{\mathcal{K} \subset \omega_F}
    \varepsilon^{-1/4}
    \tilde{\varepsilon}^{-1/4}
    \tnorm{\boldsymbol{u}-\boldsymbol{u}_h}_{sT,h,\mathcal{K}}
    +
    \lambda_{\mathcal{K}}
    \norm[0]{R_h^{\mathcal{K}}}_{\omega_{F}}.
  \end{equation*}

  For each of the three cases, summing over all facets
  $F \subset \mathcal{Q}_{\mathcal{K}} \setminus \partial\mathcal{E}$,
  and using the definitions of $\eta_{J,1}^{\mathcal{K}}$ and
  $\eta_R^{\mathcal{K}}$, we find
  \begin{equation*}
    \eta_{J,1}^{\mathcal{K}}
    \le
    \sum_{F \in \mathcal{Q}_{\mathcal{K}} \setminus \partial\mathcal{E}} \sum_{\mathcal{K}\subset\omega_F}
    \sbr[2]{
      c \varepsilon^{-1/4}
      \tilde{\varepsilon}^{-1/4}
      \tnorm{\boldsymbol{u}-\boldsymbol{u}_h}_{sT,h,\mathcal{K}}
      +
      \eta_R^{\mathcal{K}}
    }.
  \end{equation*}

  \textbf{Bound for $\eta_{BC,1}^{\mathcal{K}}$.}
  To bound $\eta_{BC,1}^{\mathcal{K}}$, let $F$ be a facet such
  that $F \subset \mathcal{Q}_{\mathcal{K}} \cap
  \partial\mathcal{E}_N$. By the triangle inequality and Young's
  inequality,
  \begin{equation}
    \label{eq:RhBC1triangleineq}
    \norm[0]{R_h^{N}}_{F}^2
    \le
    2\norm[0]{\Pi_h^{\mathcal{F}}R_h^{N}}_{F}^2
    +
    2\norm[0]{(I-\Pi_h^{\mathcal{F}})R_h^{N}}_{F}^2.
  \end{equation}
  We bound the first term on the right hand side. Using estimate
  \cref{eq:bubbleestface}, with $c_1$ and $c_2$ the constants in the
  first and second inequalities of \cref{eq:bubbleestface},
  respectively, the Cauchy--Schwarz inequality, and Young's inequality
  with constant $c_1$, we note that
  \begin{equation}
    \label{eq:PihRhBC1bounds}
    \tfrac{c_1}{2}
    \norm[0]{\Pi_h^{\mathcal{F}}R_h^{N}}_{F}^2
    \le
    \langle{R_h^{N},\psi_{F}\Pi_h^{\mathcal{F}}R_h^{N}}
    \rangle_{F}
    +
    \tfrac{c_2^2}{2c_1}
    \norm[0]{(I-\Pi_h^{\mathcal{F}})R_h^{N}}_F^2.
  \end{equation}
  Combining \cref{eq:RhBC1triangleineq,eq:PihRhBC1bounds}, and using
  the boundedness of the projection $\Pi_h^{\mathcal{F}}$ so that \\
  $\norm[0]{(I-\Pi_h^{\mathcal{F}})R_h^{N}}_{F}^2 \le c
  \norm[0]{(I-\Pi_h^{\mathcal{F}})R_h^{N}}_{F} \norm[0]{R_h^{N}}_{F}$,
  we obtain
  \begin{equation}
    \label{eq:lowerbnd3}
    \norm[0]{R_h^{N}}_{F}^2
    \le
    c\langle{R_h^{N},\psi_{F}\Pi_h^{\mathcal{F}}R_h^{N}}
    \rangle_{F}
    +
    c\norm[0]{(I-\Pi_h^{\mathcal{F}})R_h^{N}}_{F}
    \norm[0]{R_h^{N}}_{F}.
  \end{equation}
  Let $z\in H^1(\omega_F)$ be such that
  $z|_{\partial\omega_F\setminus F}=0$. Note that
  $\omega_F = \mathcal{K}$. Using integration by parts, we
  have:
  \begin{equation*}
    \del[0]{R_h^{\mathcal{K}},z}_{\mathcal{K}}
    =
    \del[0]{\varepsilon^{1/2}\overline{\nabla}(u-u_h),\varepsilon^{1/2}\overline{\nabla}z}_{\mathcal{K}}
    +
    \del[0]{\partial_t(u-u_h),z}_{\mathcal{K}}
    +
    \del[0]{\overline{\beta}\cdot\overline{\nabla}(u-u_h),z}_{\mathcal{K}}
    -
    \langle
    \varepsilon\overline{\nabla}_{\bar{n}} (u-u_h),z
    \rangle_F.
  \end{equation*}
  The last term on the right hand side can be rewritten using
  \cref{eq:st_adr_bcN} resulting in
  \begin{equation}
    \label{eq:neumannjump1}
    \begin{split}
      \langle R_h^N,z \rangle_F
      =&
      \del[0]{\varepsilon^{1/2}\overline{\nabla}(u-u_h),\varepsilon^{1/2}\overline{\nabla}z}_{\mathcal{K}}
      +
      \del[0]{\partial_t(u-u_h),z}_{\mathcal{K}}
      +
      \del[0]{\overline{\beta}\cdot\overline{\nabla}(u-u_h),z}_{\mathcal{K}}
      \\
      &
      -\del[0]{ R_h^{\mathcal{K}},z }_{\mathcal{K}}
      -\langle
      \zeta^-(u-\mu_h)\beta\cdot n ,z
      \rangle_F
      +\langle
      \zeta^-\sbr[0]{\boldsymbol{u}_h}\beta\cdot n ,z
      \rangle_F.
    \end{split}
  \end{equation}
  Choosing $z=\psi_F\Pi_h^{\mathcal{F}}R_h^N$ in
  \cref{eq:neumannjump1} and using
  \cref{eq:bubbleestface,eq:betasinfmax} and boundedness of
  $\Pi_h^{\mathcal{F}}$
  \begin{equation*}
    \begin{split}
      &ch_K^{1/2}\varepsilon^{-1/2}
      \langle R_h^N,\psi_F\Pi_h^{\mathcal{F}}R_h^N \rangle_F
      \le
      \bigg(\varepsilon^{-1/4}
      \tilde{\varepsilon}^{-1/4}
      \varepsilon^{1/2}
      \norm[0]{\overline{\nabla}(u-u_h)}_{\mathcal{K}}
      \\
      &\hspace{10em}
      +
      c h_K
      \varepsilon^{-1/2}
      \varepsilon^{1/4}
      \tilde{\varepsilon}^{1/4}
      \del[1]{
        \norm[0]{\partial_t(u-u_h)}_{\mathcal{K}}
        +
        \norm[0]{\overline{\nabla}(u-u_h)}_{\mathcal{K}}
        +
        \norm[0]{R_h^{\mathcal{K}}}_{\mathcal{K}}
      }
      \\
      &\hspace{10em}
      +
      c{h_K^{1/2}}\varepsilon^{-1/2}
      \del[1]{
        \norm[0]{\envert[0]{\tfrac{1}{2}\beta\cdot n}^{1/2}(u-\mu_h)}_{F}
        +
        \norm[0]{\envert[0]{\beta_s-\tfrac{1}{2}\beta\cdot
            n}^{1/2}\sbr[0]{\boldsymbol{u}_h}}_{F}
      }\bigg)\norm[0]{R_h^{\mathcal{F}}}_F.
    \end{split}
  \end{equation*}
  The first two terms on the right-hand side are identical to the
  right hand side in \cref{eq:spatialgradjump1} and so can be bounded
  similarly:
  \begin{multline}
    \label{eq:neumannjump2}
    ch_K^{1/2}\varepsilon^{-1/2}
    \langle R_h^N,\psi_F\Pi_h^{\mathcal{F}}R_h^N \rangle_F
    \le
    \bigg(
    \varepsilon^{-1/4}
    \tilde{\varepsilon}^{-1/4}
    \tnorm{\boldsymbol{u}-\boldsymbol{u}_h}_{sT,h,\mathcal{K}}
    +
    \lambda_{\mathcal{K}}\norm[0]{R_h^{\mathcal{K}}}_{\mathcal{K}}
    \\
    +
    c{h_K^{1/2}}\varepsilon^{-1/2}
    \del[1]{
      \norm[0]{\envert[0]{\tfrac{1}{2}\beta\cdot n}^{1/2}(u-\mu_h)}_{F}
      +
      \norm[0]{\envert[0]{\beta_s-\tfrac{1}{2}\beta\cdot
          n}^{1/2}\sbr[0]{\boldsymbol{u}_h}}_{F}
    }\bigg)\norm[0]{R_h^{\mathcal{F}}}_F.
  \end{multline}
  At this point, let us note that
  $h_K^{1/2} \varepsilon^{-1/2} \le \tilde{\varepsilon}^{-1/2}$ for
  $\delta t_{\mathcal{K}} = \mathcal{O}(h_K^2)$. Therefore, for the
  last term on the right hand side of \cref{eq:neumannjump2} we have
  \begin{equation}
    \label{eq:neumannjump3}
    c{h_K^{1/2}}\varepsilon^{-1/2}
    \del[1]{
      \norm[0]{\envert[0]{\tfrac{1}{2}\beta\cdot n}^{1/2}(u-\mu_h)}_{F}
      +
      \norm[0]{\envert[0]{\beta_s-\tfrac{1}{2}\beta\cdot
          n}^{1/2}\sbr[0]{\boldsymbol{u}_h}}_{F}
    }
    \le c\tilde{\varepsilon}^{-1/2}
    \tnorm{\boldsymbol{u}-\boldsymbol{u}_h}_{sT,h,\mathcal{K}}.
  \end{equation}
  Combining \cref{eq:lowerbnd3,eq:neumannjump2,eq:neumannjump3},
  summing over all
  $F \in \mathcal{Q}_{\mathcal{K}} \cap \partial\mathcal{E}_N$, using
  that
  $\tilde{\varepsilon}^{-1/2} \le
  \varepsilon^{-1/4}\tilde{\varepsilon}^{-1/4}$, and the definitions
  of $\eta_{BC,1}^{\mathcal{K}}$ and $\eta_R^{\mathcal{K}}$, we find
  that
  \begin{equation*}
    \eta_{BC,1}^{\mathcal{K}}
    \le
    c\varepsilon^{-1/4}\tilde{\varepsilon}^{-1/4}
    \tnorm{\boldsymbol{u}-\boldsymbol{u}_h}_{sT,h,\mathcal{K}}
    + c\eta_R^{\mathcal{K}}
    + c\,\mathrm{osc}_h^{N}.
  \end{equation*}

  \textbf{Bound for $\eta_{BC,2}^{\mathcal{K}}$.} Let $F$ be a facet
  such that $F \subset \mathcal{R}_{\mathcal{K}} \cap \Omega_0$. By
  \cref{eq:st_adr_bcN} we have that $g=-u\beta\cdot n=u$. Therefore,
  \begin{equation*}
      \eta_{BC,2}^{\mathcal{K}}
      =
      \norm[0]{u-u_h}_{F}
      \le
      \norm[0]{u-\mu_h}_{F}
      +
      \norm[0]{\sbr[0]{\boldsymbol{u}_h}}_{F}
      \le c \tnorm{\boldsymbol{u}-\boldsymbol{u}_h}_{sT,h,\mathcal{K}}.
  \end{equation*}

  \textbf{Bound for $\eta_{J,2,2}^{\mathcal{K}}$.}  Let $F$ be a facet
  such that
  $F \subset \mathcal{Q}_{\mathcal{K}} \setminus \partial\mathcal{E}$.
  Using again that
  $h_K^{1/2} \varepsilon^{-1/2} \le \tilde{\varepsilon}^{-1/2}$ for
  $\delta t_{\mathcal{K}} = \mathcal{O}(h_K^2)$, we have
  \begin{equation*}
      \eta_{J,2,2}^{\mathcal{K}}
      \le
      \varepsilon^{-1/4}\tilde{\varepsilon}^{-3/4}
      \tnorm{\boldsymbol{u}-\boldsymbol{u}_h}_{sT,h,\mathcal{K}}.
  \end{equation*}

  Combining the bounds for $\eta_R^{\mathcal{K}}$,
  $\eta_{J,1}^{\mathcal{K}}$, $\eta_{BC,1}^{\mathcal{K}}$,
  $\eta_{BC,2}^{\mathcal{K}}$ and $\eta_{J,2,2}^{\mathcal{K}}$, and
  since
  $\varepsilon^{-1/4}\tilde{\varepsilon}^{-1/4} \le \varepsilon^{-1/2}
  \tilde{\varepsilon}^{-1/2}$, we conclude \cref{eq:efficiency}.
\end{proof}

\section{Numerical Examples}
\label{s:numericalEx}

In this section, we solve the space-time HDG method
\cref{eq:st_hdg_adr_compact} with AMR using the a posteriori error
estimator $\eta^{\mathcal{K}}$ introduced in
\cref{eq:apos_st_hdg_ests_total}. The implementation uses the finite
element library deal.II \cite{dealII95,dealii2019design} on
unstructured hexahedral space-time meshes with p4est
\cite{Bangerth:2011} to obtain distributed mesh information.
Furthermore, in our implementation we choose the penalty parameter
$\alpha = 8 p_s^2$ (see, for example, \cite{Riviere:book}). The linear
system is solved all-at-once using the Multifrontal Massively Parallel
Solver (MUMPS) \cite{mumps:1,mumps:2}. In each refinement cycle, the
local error estimate $\eta^{\mathcal{K}}$ is computed for all
$\mathcal{K}\in\mathcal{T}_h$ and then ordered according to the
magnitude of $\eta^{\mathcal{K}}$. The top 25\% of elements are marked
for refinement and the bottom 10\% of elements are marked for
coarsening. The test cases in this section are implemented for both
$\delta t_{\mathcal{K}}=h_K$ and
$\delta t_{\mathcal{K}}=\mathcal{O}(h_K^2)$. In each example we will
also investigate the efficiency index, which is defined as
$\eta / \tnorm{\boldsymbol{u} - \boldsymbol{u}_h}_{sT,h}$.

\begin{remark}
  \label{rem:efficiencyindex}
  By \cref{thm:reliability}, \cref{thm:efficiency} and
  \cref{rem:efficiencyrefined-new} we expect the efficiency index to
  be bounded below by $\mathcal{O}(\varepsilon^{1/2})$ and above by
  $\mathcal{O}(\varepsilon^{-1})$ in the pre-asymptotic regime and
  above by $\mathcal{O}(\varepsilon^{-1/2})$ in the asymptotic regime.
\end{remark}

\subsection{A rotating Gaussian pulse test}
\label{ss:rot-pulse}

This test case involves a Gaussian pulse on the spatial domain
$\Omega = (-0.5,0.5)^2$ and we simulate its rotation in the time
interval $I=(0,1]$. We set $\beta = \del[0]{1,-4x_2,4x_1}^\intercal$
and $f=0$. Initial and boundary conditions are then chosen such that
the exact solution to the problem is given by
\begin{equation*}
  u(t,x_1,x_2)
  =
  \tfrac{\sigma^2}{\sigma^2+2\varepsilon t}
  \exp\del[1]{
    -\tfrac{\del[0]{\widetilde{x}_1-x_{1c}}^2+\del[0]{\widetilde{x}_2-x_{2c}}^2}{2\sigma^2+4\varepsilon t}
  },
\end{equation*}
where $\widetilde{x}_1 := x_1\cos(4t)+x_2\sin(4t)$ and
$\widetilde{x}_2 := -x_1\sin(4t)+x_2\cos(4t)$. We choose
$\sigma = 0.1$ and $(x_{1c},x_{2c})=(-0.2,0.1)$. To demonstrate the
motion of the pulse and the adaptive mesh refinement, we plot the
spatial meshes and the solutions at $t=0.2,0.5,0.8$ for
$\varepsilon=10^{-4}$ in \cref{fig:rot-pulse}.

We perform three convergence tests with $\varepsilon = 10^{-3}$,
$10^{-4}$, and $10^{-5}$. In \cref{fig:rot-pulse-conv}, for
$\delta t_{\mathcal{K}} = \mathcal{O}(h_K)$ and
$\delta t_{\mathcal{K}} = \mathcal{O}(h_K^2)$, we present the
convergence histories of the error estimator $\eta$, the true error
$\tnorm{\boldsymbol{u}-\boldsymbol{u}_h}_{sT,h}$ when using AMR, and
the true error $\tnorm{\boldsymbol{u}-\boldsymbol{u}_h}_{sT,h}$ when
using uniform refinement. Additionally, we compute the efficiency index
after each refinement cycle and plot its history. All tests are
implemented with $p_t=p_s=1$.

For both $\delta t_{\mathcal{K}} = \mathcal{O}(h_K)$ and
$\delta t_{\mathcal{K}} = \mathcal{O}(h_K^2)$ we observe
\cref{fig:rot-pulse-conv} that solutions on adaptively refined meshes
are slightly more accurate than their counterparts on uniformly
refined meshes although there is not too much advantage of using AMR
for this smooth test case. Both solutions exhibit convergence rate
$\mathcal{O}(N^{-1/2})$ which is optimal in the pre-asymptotic regime
(see \cite[Remark 1]{Wang:2023}). These results correspond to what we
expect from reliability and efficiency of the estimator proven in
\cref{thm:reliability} and \cref{thm:efficiency}. Nonrobustness of the
error estimator $\eta$ is observed with the efficiency index being of
order $\varepsilon^{-1/2}$. This lies within the interval commented on
in \cref{rem:efficiencyindex}.

\begin{figure}[tbp]
  \centering
  \begin{subfigure}{0.32\textwidth}
    \centering
    \includegraphics[width=\textwidth]{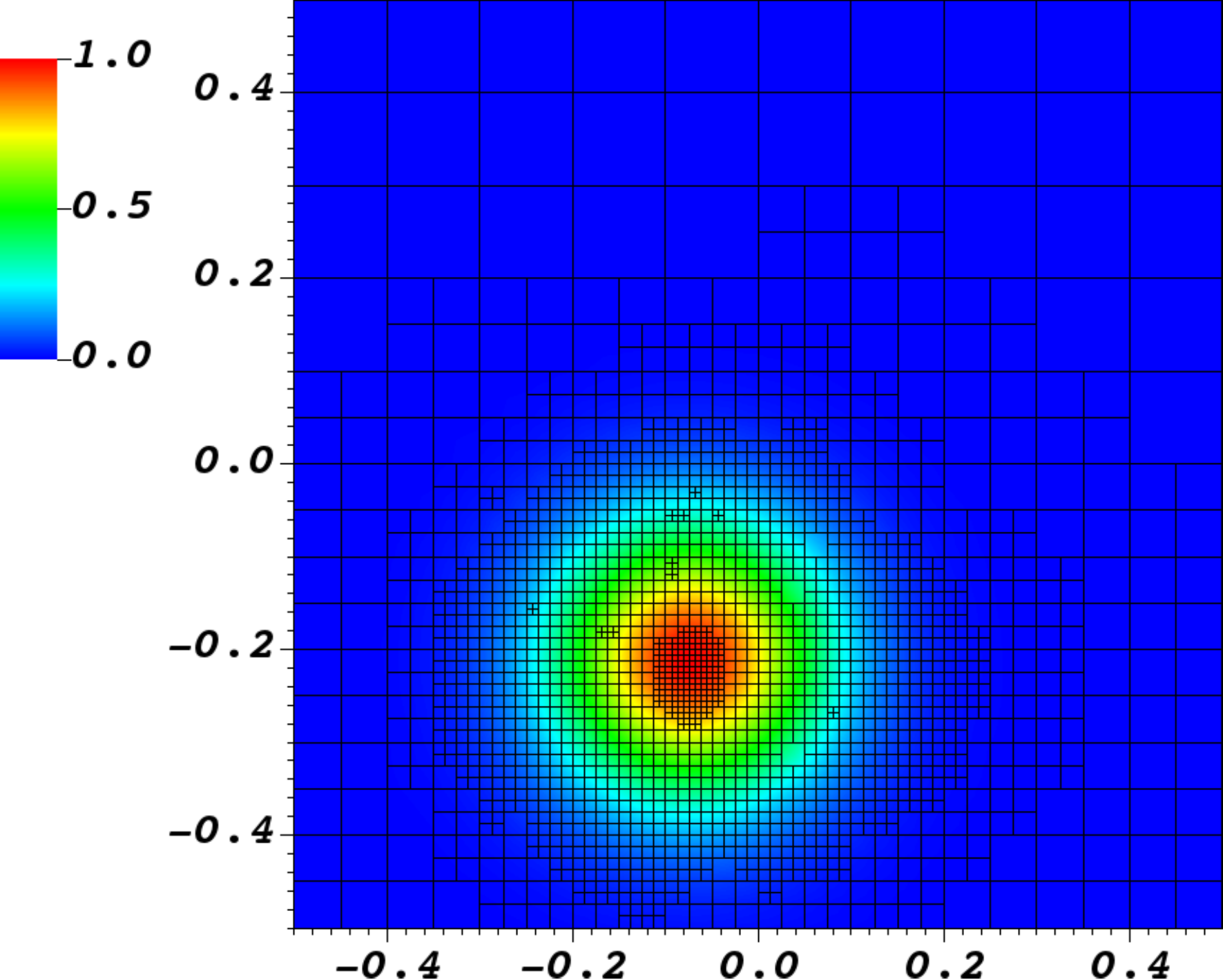}
  \end{subfigure}
  \begin{subfigure}{0.32\textwidth}
    \centering
    \includegraphics[width=\textwidth]{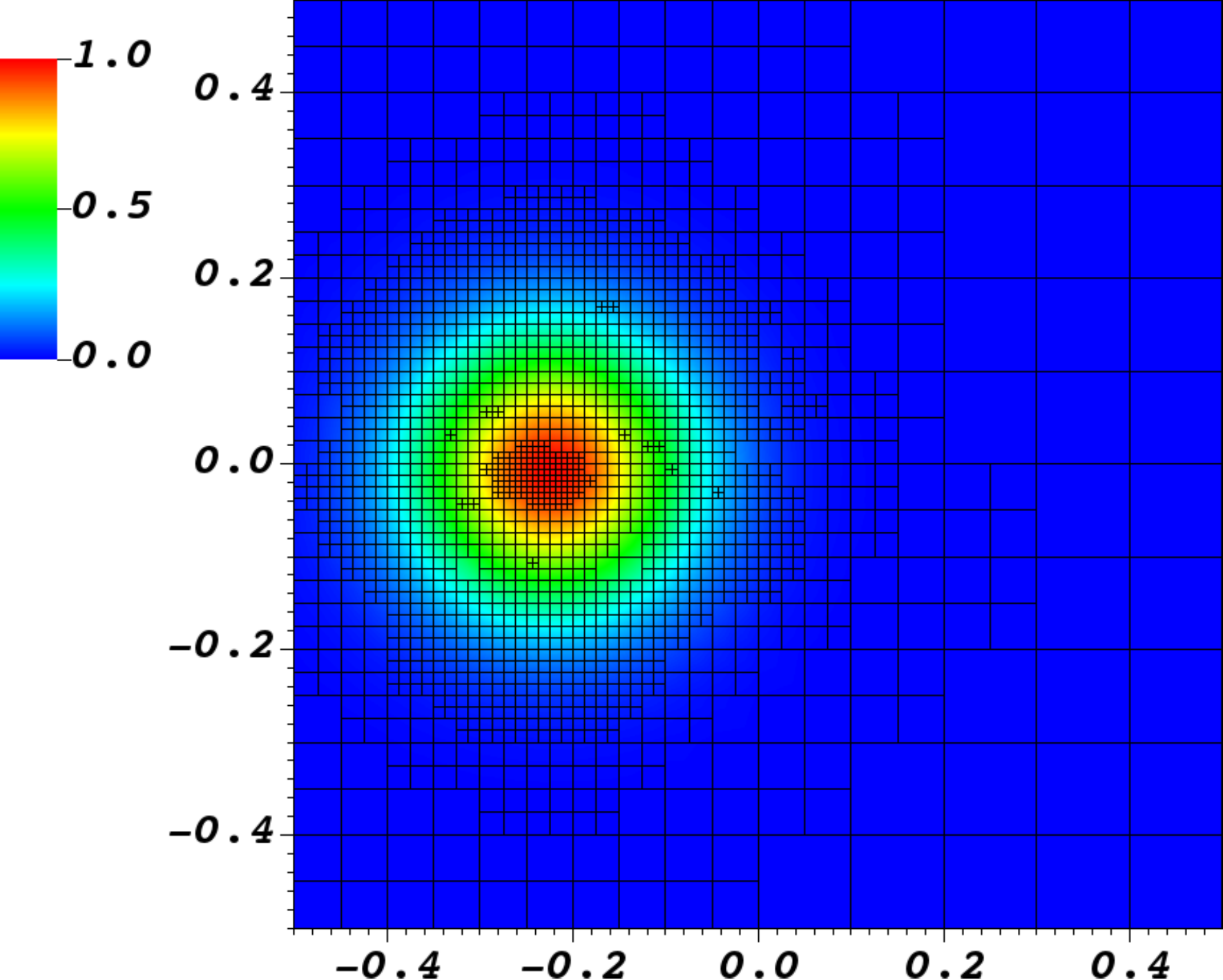}
  \end{subfigure}
  \begin{subfigure}{0.32\textwidth}
    \centering
    \includegraphics[width=\textwidth]{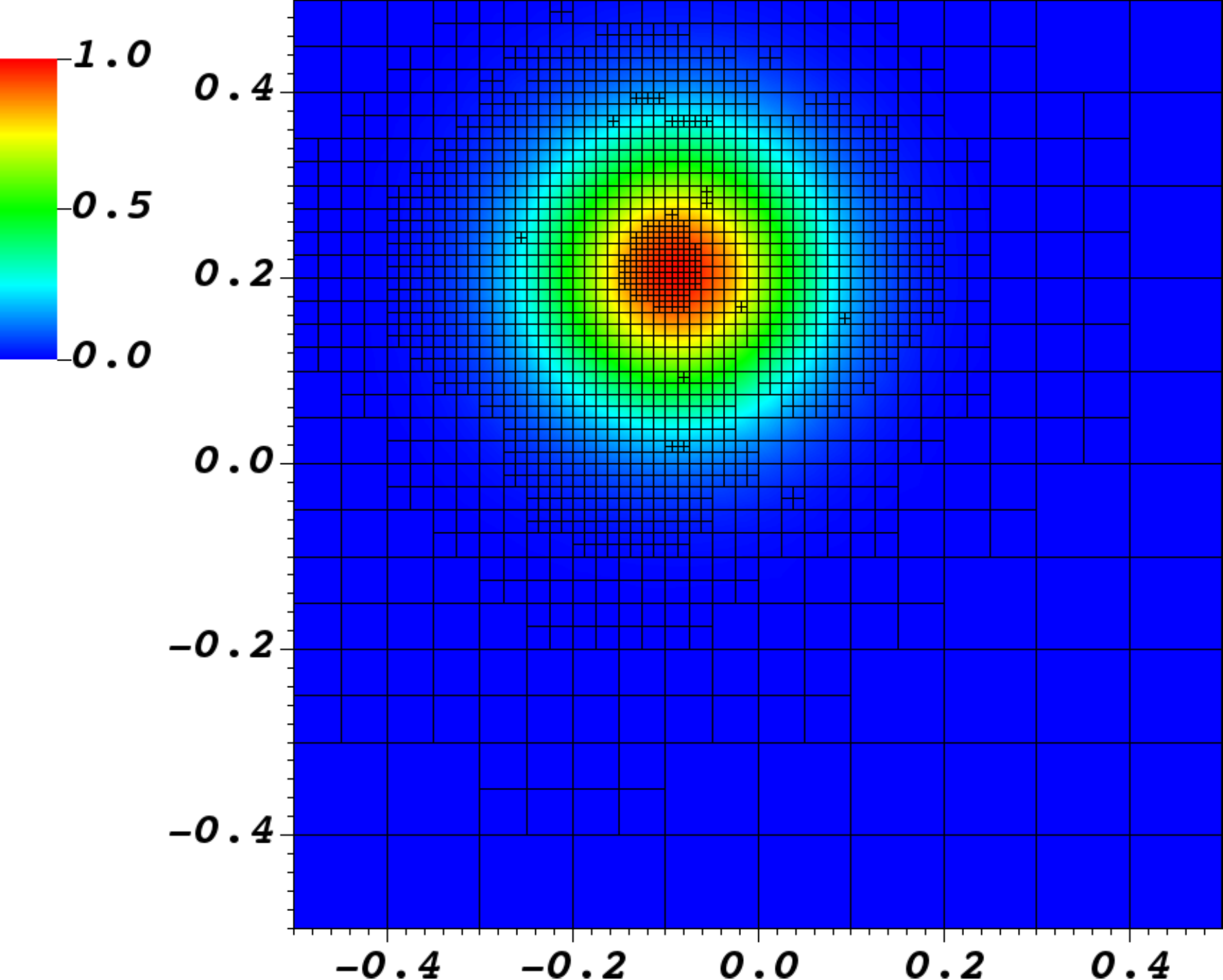}
  \end{subfigure}
  \caption{The spatial mesh and the rotating pulse. The solution is
    shown for $\varepsilon=10^{-4}$. Plots correspond to time levels
    $t = 0.2,0.5,0.8$ from left to right.}
  \label{fig:rot-pulse}
\end{figure}

\begin{figure}[tbp]
  \centering
  \includegraphics[width=\textwidth]{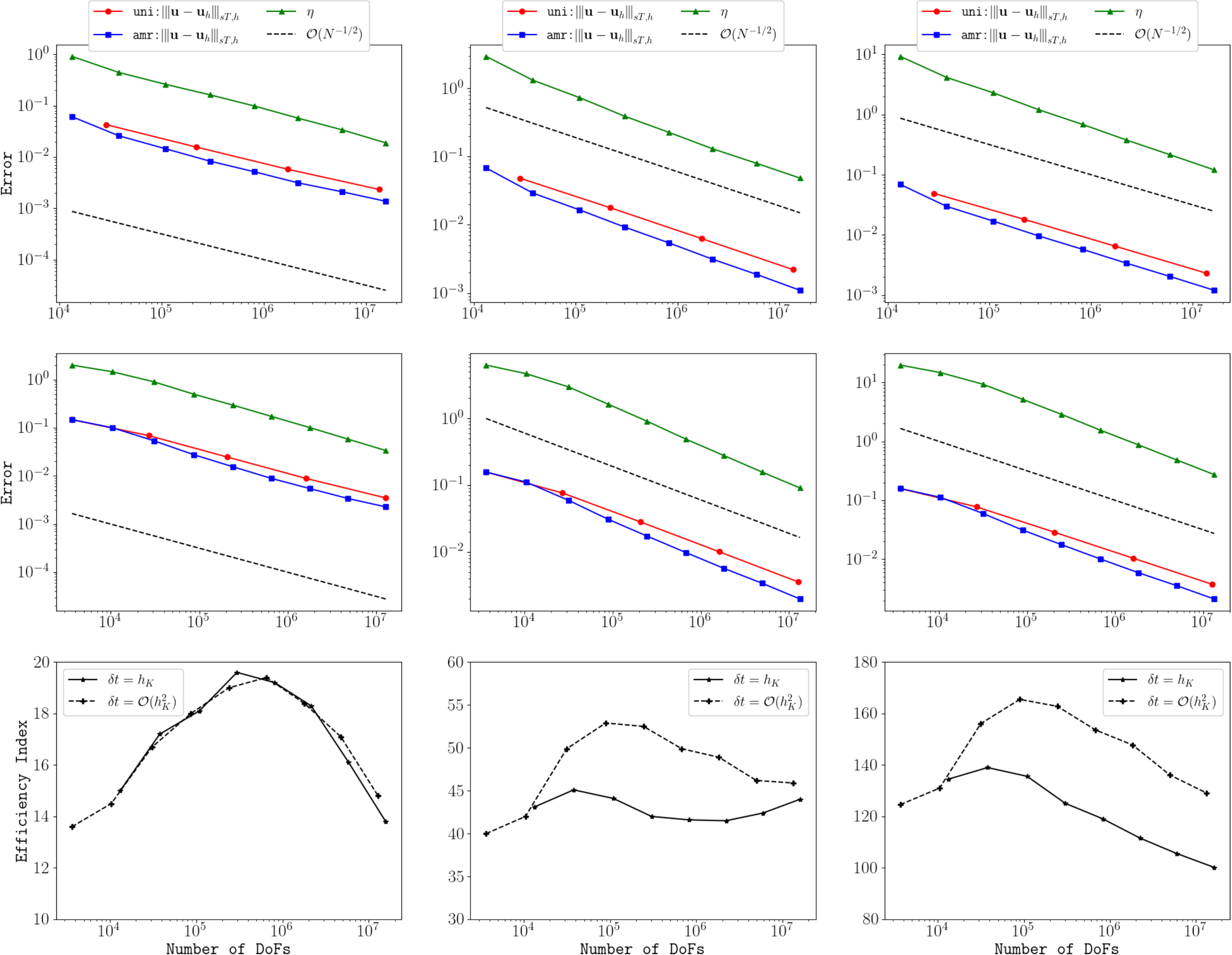}
  \caption{
	  Convergence histories of the rotating pulse test case. From
	  left to right: $\varepsilon=10^{-3}$, $\varepsilon=10^{-4}$
	  and $\varepsilon=10^{-5}$. Top row: $\delta
	  t_{\mathcal{K}}=h_K$; middle row: $\delta
	  t_{\mathcal{K}}=\mathcal{O}(h_K^2)$; bottom row: efficiency
	  index for both $\delta t_{\mathcal{K}}=h_K$ and $\delta
	  t_{\mathcal{K}}=\mathcal{O}(h_K^2)$.
  }
  \label{fig:rot-pulse-conv}
\end{figure}

\subsection{A boundary layer test}
\label{ss:bnd-layer}

We now consider problem \cref{eq:advdif} in which the solution
exhibits boundary layers. The problem is set up on the spatial domain
$\Omega = (0,1)^2$ and the time interval $I=(0,1]$ with
$\beta = \del[0]{1,1,1}^\intercal$. The initial and boundary
conditions and the source term are chosen such that the exact solution
is given by
\begin{equation*}
  u(t,x_1,x_2)
  =
  \del[0]{1-\exp(-t)}
  \del[1]{
    \tfrac{\exp((x_1-1)/\varepsilon)-1}{\exp(-1/\varepsilon)-1}+x_1-1
  }
  \del[1]{
    \tfrac{\exp((x_2-1)/\varepsilon)-1}{\exp(-1/\varepsilon)-1}+x_2-1
  }.
\end{equation*}
It is known that for small $\varepsilon$, the solution features
boundary layers of width $\mathcal{O}(\varepsilon)$ at the outflow
boundary of the spatial domain. See \cref{fig:bnd-layer} for an
example when $\varepsilon=10^{-3}$ and $\mathcal{T}_h$ has $20663$
elements.

\begin{figure}[tbp]
	\subfloat[Boundary layer test case.\label{fig:bnd-layer}]{
		\includegraphics[width=0.4\textwidth]{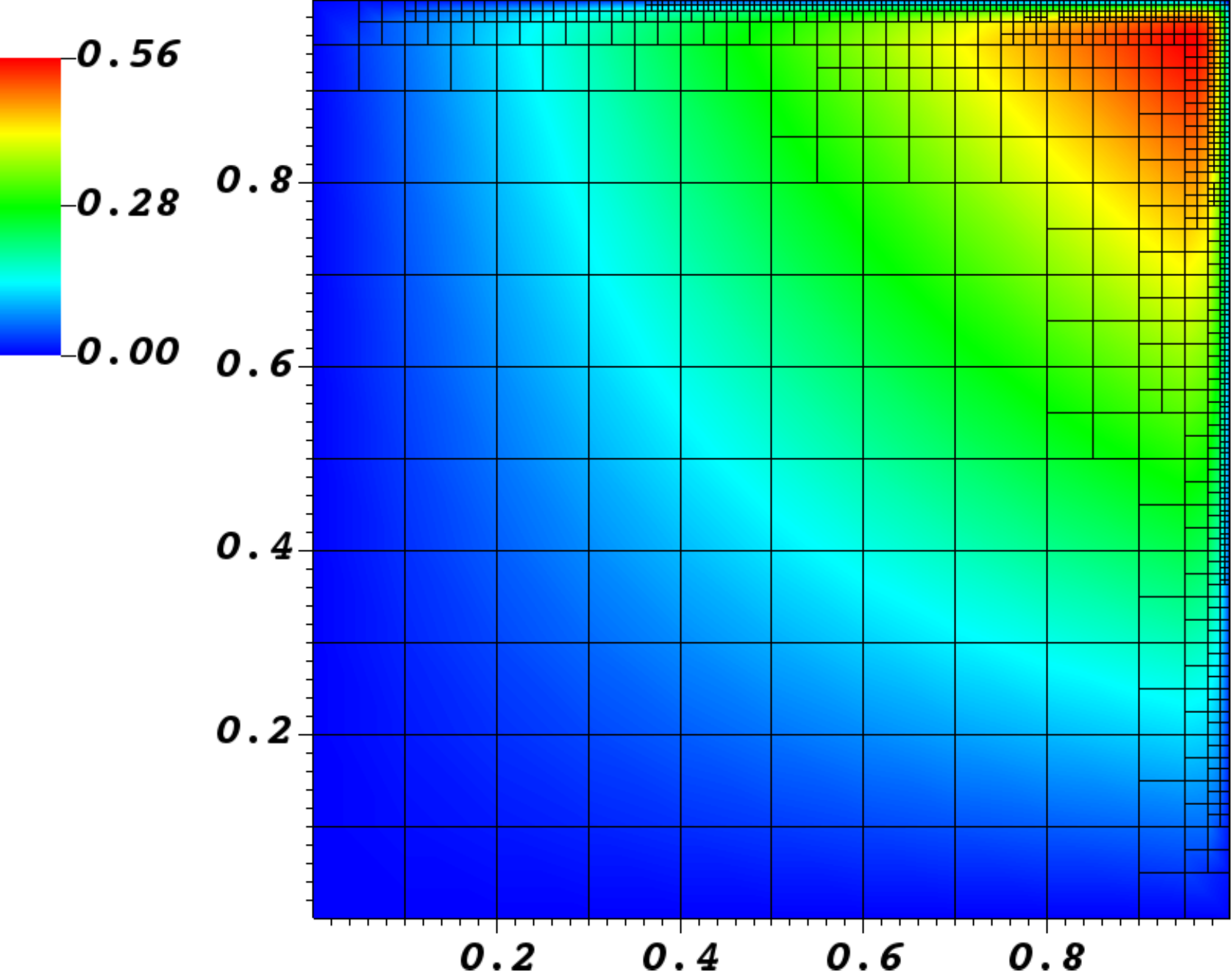}
	}
	\;
	\subfloat[Interior layer test case.\label{fig:int-layer}]{
		\includegraphics[width=0.4\textwidth]{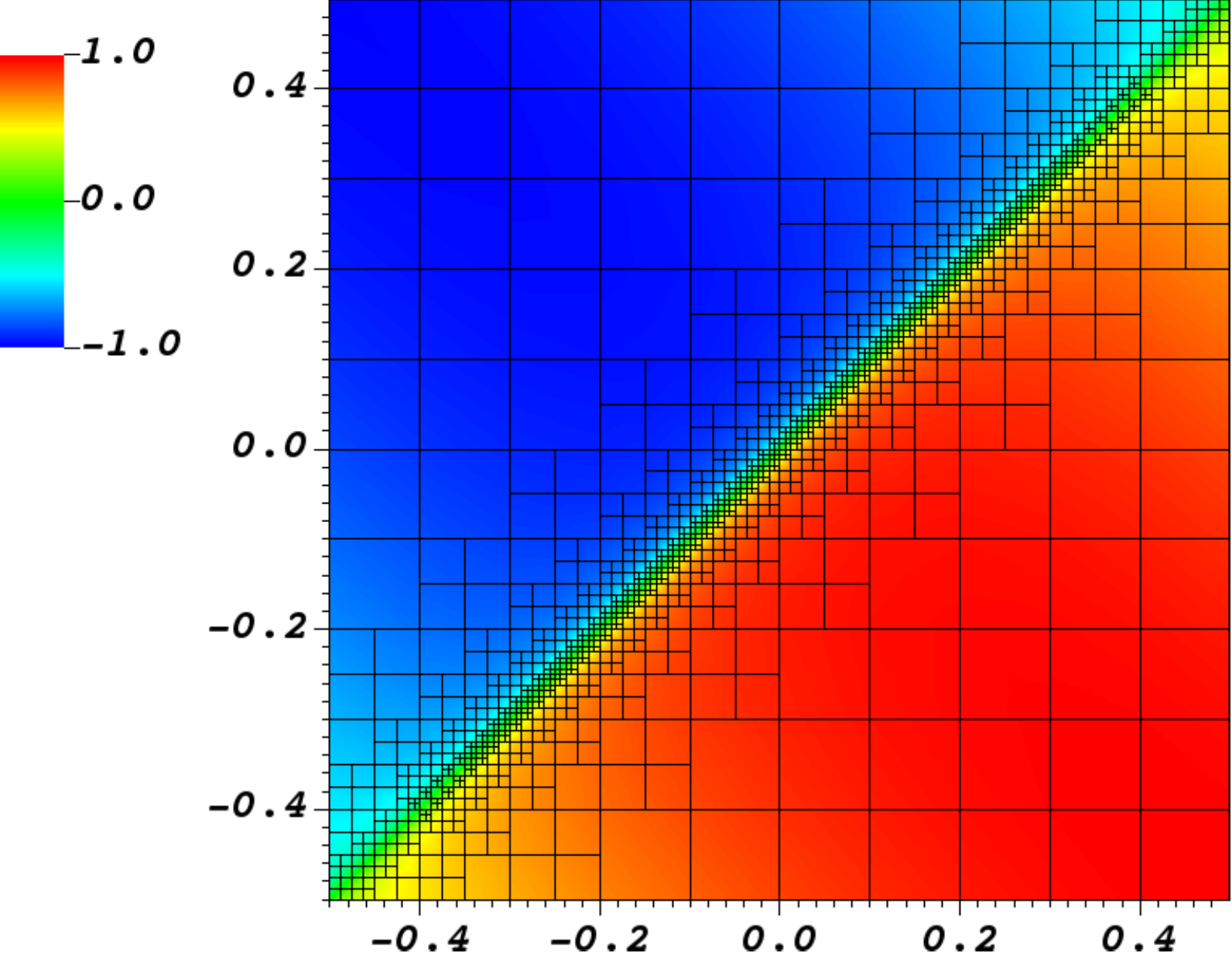}
	}
  \caption{The boundary and interior layer solutions at time
  $t=1.0$ for \cref{ss:bnd-layer,ss:int-layer} respectively. Both
	solutions are for $\varepsilon=10^{-3}$.}
  \label{fig:int-bnd-layers}
\end{figure}

Set $p_t=p_s=1$. We perform three convergence tests with
$\varepsilon=10^{-2}$, $10^{-3}$, and $10^{-4}$. For
$\delta t_{\mathcal{K}}=\mathcal{O}(h_K)$ and
$\delta t_{\mathcal{K}}=\mathcal{O}(h_K^2)$ we present in
\cref{fig:bnd-layer-conv} the convergence histories of
$\tnorm{\boldsymbol{u}-\boldsymbol{u}_h}_{sT,h}$, for both uniform and
adaptive mesh refinements, and of $\eta$ for adaptive mesh refinement.

For both $\delta t_{\mathcal{K}}=\mathcal{O}(h_K)$ and
$\delta t_{\mathcal{K}}=\mathcal{O}(h_K^2)$ we observe that for
$\varepsilon=10^{-2},10^{-3}$ and with AMR, the error
$\tnorm{\boldsymbol{u}-\boldsymbol{u}_h}_{sT,h}$ converges with
optimal rate $\mathcal{O}(N^{-1/3})$ in the asymptotic regime where
the layer has been sufficiently resolved. This is not the case for
$\varepsilon=10^{-4}$ where more refinement cycles are needed to
resolve the layer. However, solutions on adaptively refined meshes
show better accuracy than those on uniformly refined meshes. These
results verify reliability and efficiency of the estimator proven in
\cref{thm:reliability} and \cref{thm:efficiency}. Furthermore, the
efficiency indices depicted in \cref{fig:bnd-layer-conv} show
nonrobustness of order $\varepsilon^{-1/2}$ in the pre-asymptotic
regime and robustness in the asymptotic regime. These results once
again lie within the interval commented on in
\cref{rem:efficiencyindex}.

\begin{figure}[tbp]
  \centering
  \includegraphics[width=\textwidth]{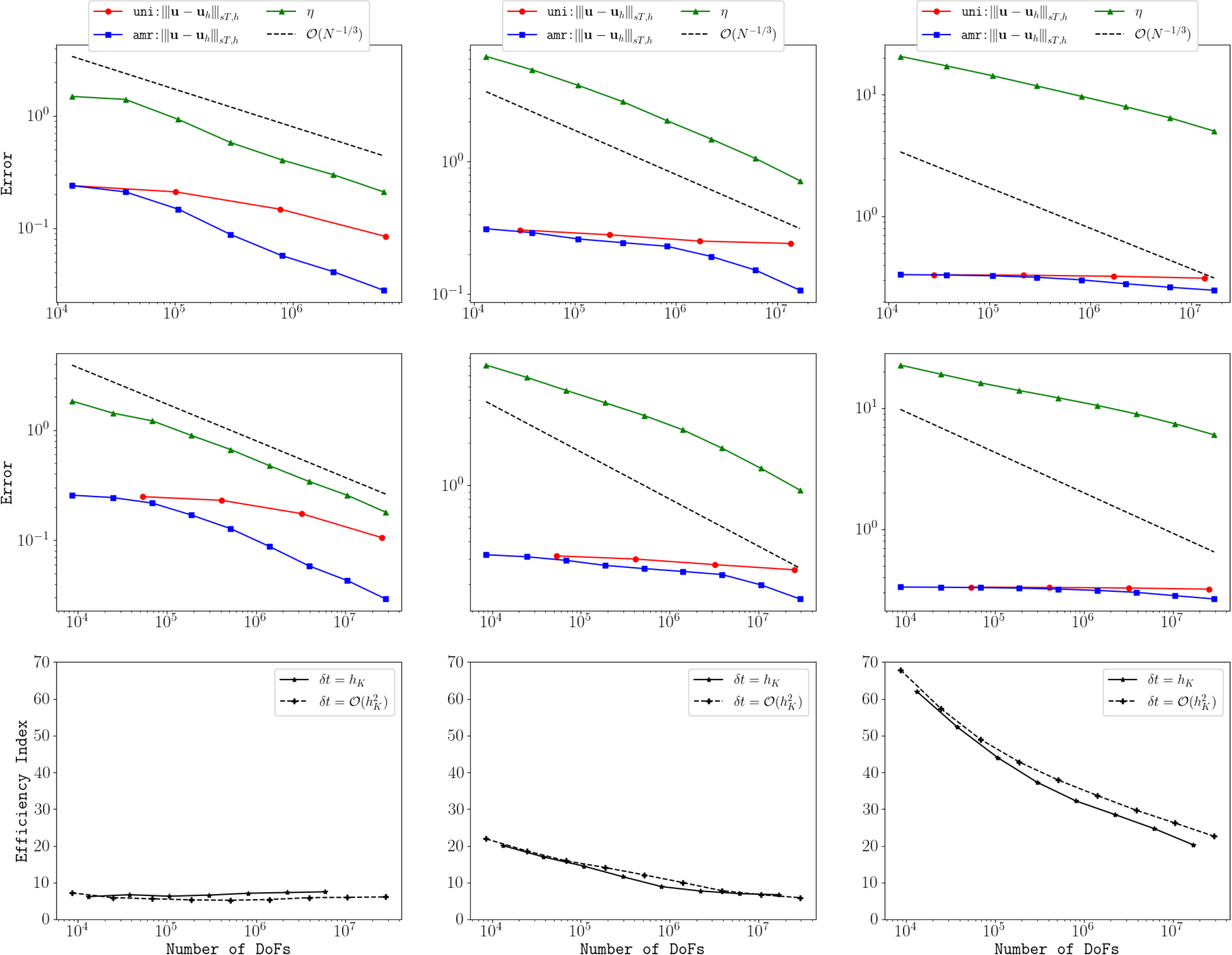}
  \caption{
	  Convergence histories of the boundary layer test case. From
	  left to right: $\varepsilon=10^{-2}$, $\varepsilon=10^{-3}$
	  and $\varepsilon=10^{-4}$. Top row: $\delta
	  t_{\mathcal{K}}=h_K$; middle row: $\delta
	  t_{\mathcal{K}}=\mathcal{O}(h_K^2)$; bottom row: efficiency
	  index for both $\delta t_{\mathcal{K}}=h_K$ and $\delta
	  t_{\mathcal{K}}=\mathcal{O}(h_K^2)$.
  }
  \label{fig:bnd-layer-conv}
\end{figure}

\subsection{An interior layer test}
\label{ss:int-layer}

In this test case, problem \cref{eq:advdif} is set up on the spatial
domain $\Omega = (-0.5,0.5)^2$ and the time interval $I=(0,1]$. We set
$\beta = \del[0]{1,1,1}^\intercal$ and set the initial condition,
boundary condition and the source term such that the exact solution is
given by
\begin{equation*}
  u(t,x_1,x_2)
  =
  \del[0]{1-\exp(-t)}
  \del[1]{
    \arctan({\tfrac{y-x}{\sqrt{2}\varepsilon}})
  }
  \del[1]{
    1-\tfrac{(x+y)^2}{2}
  }.
\end{equation*}
This solution has a diagonal interior layer on the spatial domain. See
\cref{fig:int-layer} for an example when $\varepsilon=10^{-3}$ and
when $\mathcal{T}_h$ has $23169$ elements.

As in \cref{ss:bnd-layer}, we perform three convergence tests with
$\varepsilon=10^{-2}$, $10^{-3}$, and $10^{-4}$. For
$\delta t_{\mathcal{K}}=\mathcal{O}(h_K)$ and
$\delta t_{\mathcal{K}}=\mathcal{O}(h_K^2)$, we present in
\cref{fig:int-layer-conv} the convergence histories of
$\tnorm{\boldsymbol{u}-\boldsymbol{u}_h}_{sT,h}$, for both uniform and
adaptive mesh refinements, and of $\eta$ for the adaptive mesh
refinement.

Both for $\delta t_{\mathcal{K}}=\mathcal{O}(h_K)$ and
$\delta t_{\mathcal{K}}=\mathcal{O}(h_K^2)$, when
$\varepsilon=10^{-2}$, solutions obtained on adaptively refined meshes
converge with the optimal rate $\mathcal{O}(N^{-1/3})$. On uniformly
refined meshes, solutions converge slightly slower than the optimal
rate. For $\varepsilon=10^{-3}$, adaptive meshes yield better
solutions which converge slightly faster than the optimal rate in the
asymptotic regime. Solutions on uniformly refined meshes converge with
a sub-optimal rate. For $\varepsilon=10^{-4}$, both solutions on
adaptively refined meshes and uniformly refined meshes converge
sub-optimally.  However, the former still performs better than the
latter.  Efficiency indices for all three cases are bounded above by
$10$, demonstrating robustness of the error estimator $\eta$ for this
test case.

The results from \cref{fig:int-layer-conv} verify reliability and
efficiency of the estimator proven in \cref{thm:reliability} and
\cref{thm:efficiency}. The robustness result of the error estimator
$\eta$ again lies within the interval commented on in
\cref{rem:efficiencyindex}.

\begin{figure}[tbp]
  \centering
  \includegraphics[width=\textwidth]{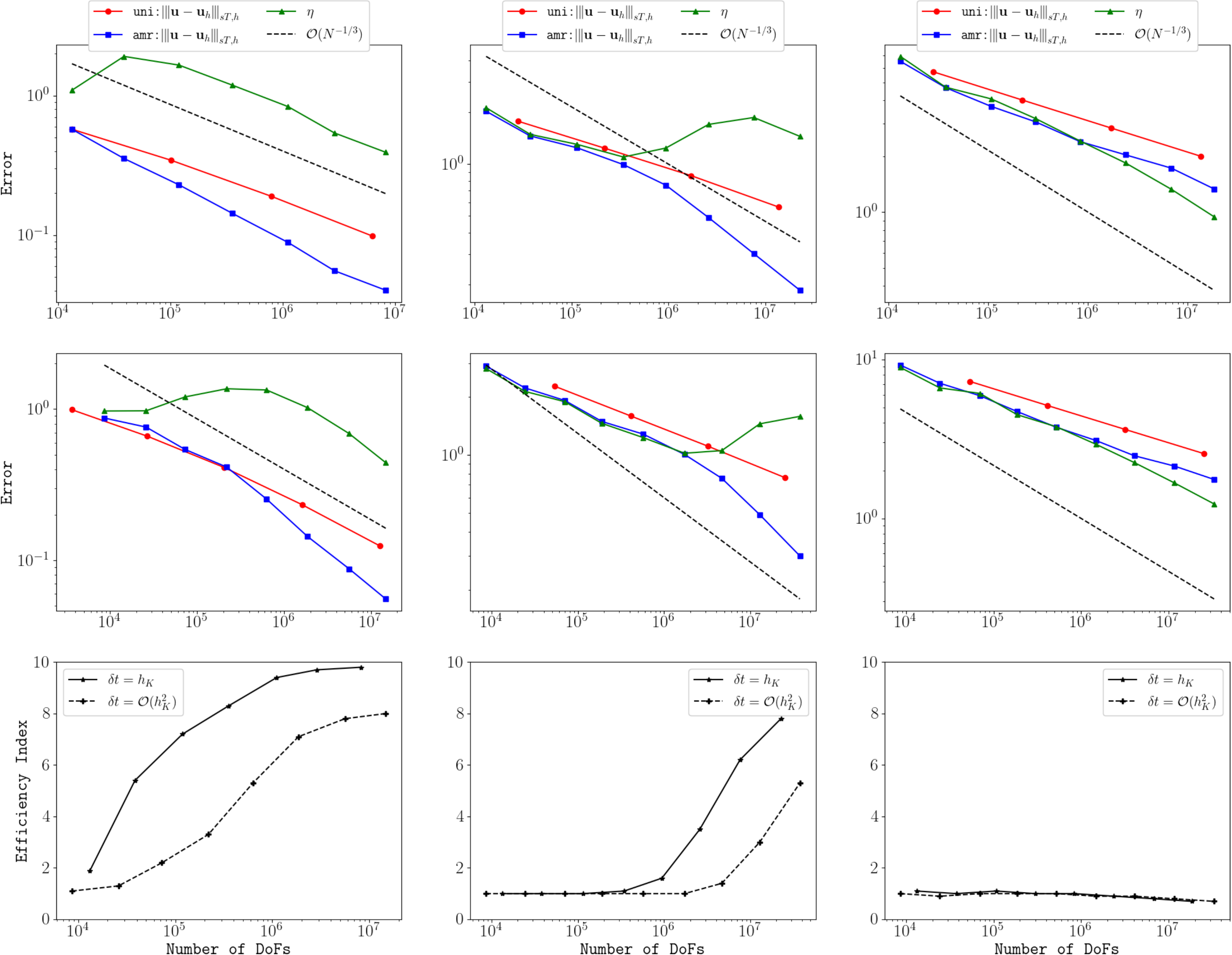}
  \caption{
	  Convergence histories of the interior layer test case. From
	  left to right: $\varepsilon=10^{-2}$, $\varepsilon=10^{-3}$
	  and $\varepsilon=10^{-4}$. Top row: $\delta
	  t_{\mathcal{K}}=h_K$; middle row: $\delta
	  t_{\mathcal{K}}=\mathcal{O}(h_K^2)$; bottom row: efficiency
	  index for both $\delta t_{\mathcal{K}}=h_K$ and $\delta
	  t_{\mathcal{K}}=\mathcal{O}(h_K^2)$.
  }
  \label{fig:int-layer-conv}
\end{figure}

\section{Conclusions}
\label{s:conclusions}

In this paper we presented and analyzed an a posteriori error
estimator for a space-time hybridizable discontinuous Galerkin
discretization of the time-dependent advection-diffusion problem with
adaptive mesh refinement. We proved, and verified numerically,
reliability and local efficiency of the error estimator with respect
to a locally computable norm. Numerical simulations showed, through
the AMR procedure, that the error estimator is able to produce meshes
on which solutions converge optimally. In particular, when sharp
layers are present, optimal convergence occurs in the asymptotic
regime. As expected from the reliability and local efficiency results
derived in \cref{thm:reliability} and \cref{thm:efficiency}, different
degrees of robustness and nonrobustness are observed in the numerical
simulations. Finally, we remark that the proofs of
\cref{thm:reliability} and \cref{thm:efficiency} assume
$\delta t_{\mathcal{K}}=\mathcal{O}(h_K^2)$. The numerical examples,
however, have shown that this assumption may be relaxed since similar
numerical results are obtained with
$\delta t_{\mathcal{K}}=\mathcal{O}(h_K)$.

\appendix
\section{Proofs of lemma's in \cref{ss:ineqapproxbounds}}
\label{s:moreinverseandtrace}

Let us first collect some useful results. Following \cite[Definition
2.9]{Georgoulis:thesis}, we define
\begin{equation*}
  H^1(\mathcal{K})
  :=
  \cbr[0]{
    u\in L^2(\mathcal{K}):\;
    (u\circ\phi_{\mathcal{K}})\in H^1(\widetilde{\mathcal{K}})
  }.
\end{equation*}
Consider an element $\widetilde{\mathcal{K}}$ and let
$\widetilde{F}_{\mathcal{Q}} \subset
\mathcal{Q}_{\widetilde{\mathcal{K}}}$ and
$\widetilde{F}_{\mathcal{R}} \subset
\mathcal{R}_{\widetilde{\mathcal{K}}}$. For
$\widetilde{u} \in H^1(\widetilde{\mathcal{K}})$, it follows from
\cite[Lemma B.7]{Sudirham:thesis} that
\begin{subequations}
  \label{eq:lb7sudir}
  \begin{align}
    \label{eq:lb7sudir-1}
    \norm[0]{\widehat{u}}_{\widehat{\mathcal{K}}}^2
    &=
      \del[1]{\tfrac{1}{2}}^{-d-1}
      \delta t_{\mathcal{K}}^{-1}
      h_K^{-d}
      \norm[0]{\widetilde{u}}_{\widetilde{\mathcal{K}}}^2,
    \\
    \label{eq:lb7sudir-2}
    \norm[0]{\widehat{u}}_{\widehat{F}_{\mathcal{Q}}}^2
    &= \del[1]{\tfrac{1}{2}}^{-d}
      \delta t_{\mathcal{K}}^{-1}h_K^{-d+1}
      \norm[0]{\widetilde{u}}_{\widetilde{F}_{\mathcal{Q}}}^2,
    \\
    \label{eq:lb7sudir-3}
    \norm[0]{\widehat{u}}_{\widehat{F}_{\mathcal{R}}}^2
    &= \del[1]{\tfrac{1}{2}}^{-d} h_K^{-d}
      \norm[0]{\widetilde{u}}_{\widetilde{F}_{\mathcal{R}}}^2,
    \\
    \label{eq:lb7sudir-4}
    \norm[0]{\widehat{\overline{\nabla}}\widehat{u}}_{\widehat{\mathcal{K}}}^2
    &= \del[1]{\tfrac{1}{2}}^{-d+1}
      \delta t_{\mathcal{K}}^{-1}h_K^{-d+2}
      \norm[0]{\widetilde{\overline{\nabla}}\widetilde{u}}_{\widetilde{\mathcal{K}}}^2.
  \end{align}
\end{subequations}
We also have,
\begin{subequations}
  \label{eq:scalingequiv}
  \begin{align}
    \label{eq:scalingequiv_1}
    c^{-1}\norm[0]{u}_{\mathcal{K}}
    \le
    &\norm[0]{\widetilde{u}}_{\widetilde{\mathcal{K}}}
      \le c
      \norm[0]{u}_{\mathcal{K}},
    \\
    \label{eq:scalingequiv_2}
    c^{-1}\norm[0]{u}_{{F}_{\mathcal{Q}}}
    \le
    &\norm[0]{\widetilde{u}}_{\widetilde{F}_{\mathcal{Q}}}
      \le c
      \norm[0]{u}_{{F}_{\mathcal{Q}}},
    \\
    \label{eq:scalingequiv_3}
    c^{-1}\norm[0]{u}_{{F}_{\mathcal{R}}}
    \le
    &\norm[0]{\widetilde{u}}_{\widetilde{F}_{\mathcal{R}}}
      \le c
      \norm[0]{u}_{{F}_{\mathcal{R}}},
    \\
    \label{eq:scalingequiv_4}
    \norm[0]{\widetilde{\overline{\nabla}}\widetilde{u}}_{\widetilde{\mathcal{K}}}
    \le c
    \norm[0]{\overline{\nabla}{u}}_{\mathcal{K}},
    &\quad
      \norm[0]{\partial_{\tilde{t}}\widetilde{u}}_{\widetilde{\mathcal{K}}}
      \le c
      \norm[0]{\partial_{{t}}{u}}_{\mathcal{K}}.
  \end{align}
\end{subequations}
Here, \cref{eq:scalingequiv_1} follows from a change of variables and
\cref{eq:diffeom_jac}, while
\cref{eq:scalingequiv_2,eq:scalingequiv_3} follow from a change of
variables and
\cref{eq:k_surface,eq:diffeom_regular_d}. \Cref{eq:scalingequiv_4}
follows from the chain rule, \cref{eq:diffeom_regular_special}, change
of variables and \cref{eq:diffeom_jac}.

When $\widetilde{u}\in Q^{(p_t,p_s)}(\widetilde{\mathcal{K}})$ we have
\begin{subequations}
  \label{eq:scalingphiextra}
  \begin{align}
    \label{eq:scalingphiextra-3}
    \norm[0]{\partial_{\hat{t}}\widehat{u}}_{\widehat{F}_{\mathcal{R}}}
     &=
       \del{\tfrac{1}{2}}^{1-d/2}
       \delta t_{\mathcal{K}}
       h_K^{-d/2}
       \norm[0]{\partial_{\tilde{t}}\widetilde{u}}_{\widetilde{F}_{\mathcal{R}}},
    \\
    \label{eq:scalingphiextra-2}
    \norm[0]{\partial_{\tilde{t}}\widetilde{u}}_{\widetilde{F}_{\mathcal{R}}}
     &\le c
       \norm[0]{\partial_tu}_{{F}_{\mathcal{R}}},
    \\
    \label{eq:scalingphiextra-1}
  c^{-1}
  \norm[0]{\widetilde{u}}_{\widetilde{E}_{\mathcal{K}}}
  &\le
  \norm[0]{u}_{E_{\mathcal{K}}}
  \le
  c \norm[0]{\widetilde{u}}_{\widetilde{E}_{\mathcal{K}}}.
  \end{align}
\end{subequations}
\Cref{eq:scalingphiextra-3} follows by extending \cite[Lemma
A.3]{Georgoulis:thesis} to $d+1$ dimensions in the space-time setting,
while \cref{eq:scalingphiextra-2} follows from the chain rule,
\cref{eq:diffeom_regular_special}, a change of variables, and
\cref{eq:k_surface,eq:diffeom_regular_d}. \Cref{eq:scalingphiextra-1}
holds on an edge $E_{\mathcal{K}}$ such that
$\phi_{\mathcal{K}}(\widetilde{E}_{\mathcal{K}})={E}_{\mathcal{K}}$. It
is analogous to \cref{eq:scalingequiv_2} in an integral domain with
one lower dimension and can be shown with similar steps.

When $u\in H^1(\mathcal{K})$, we have the standard anisotropic
projection estimate (see, for example, \cite[Lemma
3.13]{Georgoulis:thesis}) and the projection bounds shown in
\cite[eqs. (4.7a) and (4.7b)]{Wang:2023}:
\begin{subequations}
  \label{eq:projectionbounds}
  \begin{align}
    \label{eq:standardanisoprojection}
    \norm[0]{u-\Pi_hu}_{\mathcal{K}}
    &\le c
      \del[0]{
      \delta t_{\mathcal{K}}
      \norm[0]{\partial_tu}_{\mathcal{K}}
      +
      h_K\norm[0]{\overline{\nabla}u}_{\mathcal{K}}
      },
    \\
    \label{eq:derivativeofprojections}
    \norm[0]{\overline{\nabla}{(u-\Pi_hu)}}_{\mathcal{K}}
    &\le c
    \norm[0]{\overline{\nabla}u}_{\mathcal{K}},
      \quad
      \norm[0]{\partial_t{(u-\Pi_hu)}}_{\mathcal{K}}
      \le c
      \norm[0]{\partial_tu}_{\mathcal{K}}.
  \end{align}
\end{subequations}
Let us remark that the projection bounds \cite[eqs. (4.7a) and
(4.7b)]{Wang:2023} were proven on a moving mesh. On the fixed
mesh considered here, \cite[Eq. (4.7b)]{Wang:2023} reduces to the
second inequality of \cref{eq:derivativeofprojections}.

\subsection{Proof of \cref{lem:local_quasi_tnorm_st}}
\label{s:localquasitnormst}

Consider the following local trace inequality, which holds for all
$\mathcal{K}\in\mathcal{T}_h$,
$F_{\mathcal{Q}}\subset\mathcal{Q}_{\mathcal{K}}$ and
$v\in H^1(\mathcal{K})$ and it can be shown by combining \cite[Lemma
B.6]{Sudirham:thesis} and scaling arguments
\cref{eq:scalingequiv_2,eq:scalingequiv_3,eq:scalingequiv_4}
\begin{subequations}
  \label{eq:localtrace}
  \begin{align}
    \label{eq:localtrace_1}
    \norm[0]{v}_{F_{\mathcal{Q}}}^2
    &\le c
      \del[1]{
      h_K^{-1}
      \norm[0]{v}_{\mathcal{K}}^2
      +
      \norm[0]{v}_{\mathcal{K}}
      \norm[0]{\overline{\nabla}v}_{\mathcal{K}}
      },
    \\
    \label{eq:localtrace_2}
    \norm[0]{v}_{F_{\mathcal{R}}}^2
    &
      \le c
      \del[1]{
      \delta t_{\mathcal{K}}^{-1}
      \norm[0]{v}_{\mathcal{K}}^2
      +
      \norm[0]{v}_{\mathcal{K}}
      \norm[0]{\partial_tv}_{\mathcal{K}}
      }.
  \end{align}
\end{subequations}
Using \cref{eq:standardanisoprojection}, we have
\begin{equation}
  \label{eq:quasiprojproof1}
    \norm[0]{v-\Pi_hv}_{\mathcal{K}}
    \le c
    {h_{K}\varepsilon^{-{1}/{2}}}
    \del[0]{
      \varepsilon^{{1}/{2}}
      h_{K}^{-1}\delta t_{\mathcal{K}}
      \norm[0]{\partial_tv}_{\mathcal{K}}
      +
      \varepsilon^{{1}/{2}}
      \norm[0]{\overline{\nabla}v}_{\mathcal{K}}
      +
      \norm[0]{v}_{\mathcal{K}}
    },
\end{equation}
while boundedness of the projection operator $\Pi_h$ gives
\begin{equation}
  \label{eq:quasiprojproof2}
  \norm[0]{v-\Pi_hv}_{\mathcal{K}}
  \le c
  \norm[0]{v}_{\mathcal{K}}
  \le c
  \del[0]{
    \varepsilon^{{1}/{2}}
    h_{K}^{-1}\delta t_{\mathcal{K}}
    \norm[0]{\partial_tv}_{\mathcal{K}}
    +
    \varepsilon^{{1}/{2}}
    \norm[0]{\overline{\nabla}v}_{\mathcal{K}}
    +
    \norm[0]{v}_{\mathcal{K}}
  }.
\end{equation}
Combining \cref{eq:quasiprojproof1,eq:quasiprojproof2} yields
\begin{equation}
  \label{eq:local_quasi_tnorm_st_1}
  \norm[0]{v-\Pi_hv}_\mathcal{K}
  \le c
  \lambda_\mathcal{K}
  \del[1]{
    \varepsilon^{1/2}{h_{K}^{-1}}{\delta t_{\mathcal{K}}}
    \norm[0]{\partial_t v}_\mathcal{K}
    +
    \varepsilon^{1/2}
    \norm[0]{\overline{\nabla} v}_\mathcal{K}
    +
    \norm[0]{v}_\mathcal{K}
  }.
\end{equation}
Furthermore, combining the trace inequalities \cref{eq:localtrace}
with the projection bounds \cref{eq:projectionbounds}, we obtain:
\begin{subequations}
  \label{eq:local_quasi_tnorm_st_34}
  \begin{align}
    \label{eq:local_quasi_tnorm_st_3}
    \norm[0]{v-\Pi_hv}_{F_\mathcal{Q}}
    &
      \le c
      \varepsilon^{-1/2}
      \del[1]{
      {\delta t_{\mathcal{K}}}{h_{K}^{-1/2}}\varepsilon^{1/2}
      \norm[0]{\partial_t v}_\mathcal{K}
      +
      h_{K}^{1/2}
      \varepsilon^{1/2}
      \norm[0]{\overline{\nabla} v}_\mathcal{K}
      },
    \\
    \label{eq:local_quasi_tnorm_st_4}
    \norm[0]{v-\Pi_hv}_{F_\mathcal{R}}
    &
      \le c
      \varepsilon^{-1/2}
      \del[1]{
      \delta t_{\mathcal{K}}^{1/2}
      \varepsilon^{1/2}
      \norm[0]{\partial_t v}_\mathcal{K}
      +
      {h_K}{\delta t_{\mathcal{K}}^{-1/2}}
      \varepsilon^{1/2}
      \norm[0]{\overline{\nabla} v}_\mathcal{K}
      }.
  \end{align}
\end{subequations}
Lemma~\ref{lem:local_quasi_tnorm_st} is now an immediate consequence of
\cref{eq:local_quasi_tnorm_st_34}, \cref{eq:local_quasi_tnorm_st_1},
and using $\delta t_{\mathcal{K}} = \mathcal{O}(h_K^2)$.

\subsection{Proof of \cref{lem:oswald_local_st}}
\label{s:oswaldlocalst}

The proof below combines an estimate for the averaging operator on
conforming meshes (extended from \cite[Lemma 22.3]{Ern:book} to
space-time meshes), and an auxiliary mesh technique (see, for
example,
\cite{Houston:2008,Karakashian:2003,Schotzau:2011b,Schotzau:2011a}).

We start by proving \cref{eq:oswald_local_st} on a conforming
$\del{d+1}$-dimensional space-time mesh. Within this conforming mesh,
consider $\sigma_{\mathcal{K}}$ which consists of a space-time element
$\mathcal{K}$ and $\mathcal{K}_i$, $i=1,\hdots,3^{(d+1)}-1$. We map
$\sigma_{\mathcal{K}}$ to the reference domain while preserving
connectivity relations between the elements. This is achieved by
combining $\Phi_{\mathcal{K}}$ with $\Phi_{\mathcal{K}_i}$
($i=1,\hdots,3^{(d+1)}-1$), where $\Phi_{\mathcal{K}_i}$ are
$\Phi_{\mathcal{K}}$ with suitable linear translations.

Applying \cite[Lemma 22.3]{Ern:book} to $\sigma_{\mathcal{K}}$ in the
reference domain,
\begin{equation}
  \label{eq:st_oswald_local_ref}
  \norm[0]{\widehat{v}_h-\mathcal{I}^{c}_h\widehat{v}_h}_{\widehat{\mathcal{K}}}
  \le c
  \del[1]{
    \sum_{
	 \widehat{F}\subset\check{\mathcal{Q}}^i_{\widehat{\mathcal{K}}}}
    \norm[0]{\jump{\widehat{v}_h}}_{\widehat{F}}
    +
    \sum_{
	 \widehat{F}\subset\check{\mathcal{R}}^i_{\widehat{\mathcal{K}}}}
    \norm[0]{\jump{\widehat{v}_h}}_{\widehat{F}}
  }.
\end{equation}
We remark that in the proof of \cite[Lemma 22.3]{Ern:book}, the only
intermediate result that restricts the domain dimension to be lower
than or equal to three is \cite[Lemma 21.4]{Ern:book}. We argue that
\cite[Lemma 21.4]{Ern:book} can be extended to the space-time domain
$\mathcal{E}\subset\mathbb{R}^{d+1}$ due to it being Lipschitz. With
scaling arguments \cref{eq:lb7sudir,eq:scalingequiv},
\cref{eq:st_oswald_local_ref} is transformed back to the physical
domain:
\begin{equation}
  \label{eq:st_oswald_local}
  \norm[0]{{v}_h-\mathcal{I}^{c}_h{v}_h}_{{\mathcal{K}}}
  \le c
  \del[1]{
    \sum_{
	 F\subset\check{\mathcal{Q}}^i_{{\mathcal{K}}}}
    h_K^{{1}/{2}}
    \norm[0]{\jump{{v}_h}}_F
    +
    \sum_{
	 F\subset\check{\mathcal{R}}^i_{{\mathcal{K}}}}
    \delta t_{\mathcal{K}}^{1/2}
    \norm[0]{\jump{{v}_h}}_F
  }.
\end{equation}

We now consider the case of a 1-irregular mesh. Let
$\mathcal{K} \in \mathcal{T}_h$ and let $\mathcal{T}_h^c$ be the
coarsest refinement of $\mathcal{T}_h$. We consider two cases: (1)
$\mathcal{K}$ is not refined on $\mathcal{T}_h^c$; and (2)
$\mathcal{K}$ is refined on $\mathcal{T}_h^c$.

\textbf{Case 1.} If $\mathcal{K}$ is not refined on $\mathcal{T}_h^c$,
we denote by $\sigma_{\mathcal{K}}^c$ the local patch of elements
associated with $\mathcal{K}$ on $\mathcal{T}_h^c$. Applying
\cref{eq:st_oswald_local} on $\sigma_{\mathcal{K}}^c$ gives
\begin{equation}
  \label{eq:stnonconfref_3}
  \norm[0]{{v}_h-\mathcal{I}^{c}_h{v}_h}_{{\mathcal{K}}}
  \le c
  \del[1]{
    \sum_{
      F\subset
      \check{\mathcal{Q}}^{i,c}_{{\mathcal{K}}}
      \setminus
      \check{\mathfrak{Q}}^i_{{\mathcal{K}}}
    }
    h_K^{{1}/{2}}
    \norm[0]{\jump{{v}_h}}_F
    +
    \sum_{
      F\subset
      \check{\mathcal{R}}^{i,c}_{{\mathcal{K}}}
      \setminus
      \check{\mathfrak{R}}^i_{{\mathcal{K}}}
    }
    \delta t_{\mathcal{K}}^{{1}/{2}}
    \norm[0]{\jump{{v}_h}}_F
  },
\end{equation}
where $\check{\mathcal{Q}}^{i,c}_{{\mathcal{K}}}$ and
$\check{\mathcal{R}}^{i,c}_{{\mathcal{K}}}$ are defined similarly as
$\check{\mathcal{Q}}^{i}_{{\mathcal{K}}}$ and
$\check{\mathcal{R}}^{i}_{{\mathcal{K}}}$, but for $\mathcal{K}$ on
$\mathcal{T}_h^c$, and where $\check{\mathfrak{Q}}^i_{{\mathcal{K}}}$
and $\check{\mathfrak{R}}^i_{{\mathcal{K}}}$ are unions of newly
generated $\mathcal{Q}$-facets and $\mathcal{R}$-facets that divide an
element in $\mathcal{T}_h$ to create $\mathcal{T}_h^c$. Note that
$\jump{v_h}$ vanishes on $\check{\mathfrak{Q}}^i_{{\mathcal{K}}}$ and
$\check{\mathfrak{R}}^i_{{\mathcal{K}}}$, explaining why they are
excluded from the summation in \cref{eq:stnonconfref_3}.
\Cref{eq:oswald_local_st} then follows from \cref{eq:stnonconfref_3}
by noting that
\begin{equation*}
    \sum_{
      F\subset
      \check{\mathcal{Q}}^{i,c}_{{\mathcal{K}}}
      \setminus
      \check{\mathfrak{Q}}^i_{{\mathcal{K}}}
    }
    h_K^{{1}/{2}}
    \norm[0]{\jump{{v}_h}}_F
    \le c
    \sum_{
      F\subset
      \check{\mathcal{Q}}^{i}_{{\mathcal{K}}}
    }
    h_K^{{1}/{2}}
    \norm[0]{\jump{{v}_h}}_F,
    \quad
    \sum_{
      F\subset
      \check{\mathcal{R}}^{i,c}_{{\mathcal{K}}}
      \setminus
      \check{\mathfrak{R}}^i_{{\mathcal{K}}}
    }
    h_K^{{1}/{2}}
    \norm[0]{\jump{{v}_h}}_F
    \le c
    \sum_{
      F\subset
      \check{\mathcal{R}}^{i}_{{\mathcal{K}}}
    }
    h_K^{{1}/{2}}
    \norm[0]{\jump{{v}_h}}_F.
\end{equation*}

\textbf{Case 2.} When $\mathcal{K}$ is refined on $\mathcal{T}_h^c$
into $M_{\mathcal{K}}$ elements,
$\mathcal{K}=\cup_{j=1}^{M_{\mathcal{K}}}\mathcal{K}_{j}$, where
$M_{\mathcal{K}}\le 2^{d+1}$. We apply \cref{eq:st_oswald_local} on
each $\sigma_{\mathcal{K}_j}^c$ resulting in
\begin{equation}
  \label{eq:stnonconfref_1}
  \norm[0]{{v}_h-\mathcal{I}^{c}_h{v}_h}_{{\mathcal{K}_j}}
  \le c
  \del[1]{
    \sum_{
      F\subset
      \check{\mathcal{Q}}^{i,c}_{{\mathcal{K}_j}}
      \setminus
      \check{\mathfrak{Q}}^i_{{\mathcal{K}}_j}
    }
    h_K^{{1}/{2}}
    \norm[0]{\jump{{v}_h}}_F
    +
    \sum_{
      F\subset
      \check{\mathcal{R}}^{i,c}_{{\mathcal{K}_j}}
      \setminus
      \check{\mathfrak{R}}^i_{{\mathcal{K}}_j}
    }
    \delta t_{\mathcal{K}}^{{1}/{2}}
    \norm[0]{\jump{{v}_h}}_F
  }.
\end{equation}
Combining \cref{eq:stnonconfref_1} for all
$j=1,\dots, M_{\mathcal{K}}$ gives \cref{eq:oswald_local_st} by noting
that
\begin{equation*}
  \begin{split}
    \sum_{j=1}^{M_{\mathcal{K}}}
    \sum_{
      F\subset
      \check{\mathcal{Q}}^{i,c}_{{\mathcal{K}_j}}
      \setminus
      \check{\mathfrak{Q}}^i_{{\mathcal{K}}_j}
    }
    h_K^{{1}/{2}}
    \norm[0]{\jump{{v}_h}}_F
    &\le c
    \sum_{
      F\subset
      \check{\mathcal{Q}}^{i}_{{\mathcal{K}}}
    }
    h_K^{{1}/{2}}
    \norm[0]{\jump{{v}_h}}_F,
    \\
    \sum_{j=1}^{M_{\mathcal{K}}}
    \sum_{
      F\subset
      \check{\mathcal{R}}^{i,c}_{{\mathcal{K}_j}}
      \setminus
      \check{\mathfrak{R}}^i_{{\mathcal{K}}_j}
    }
    h_K^{{1}/{2}}
    \norm[0]{\jump{{v}_h}}_F
    &\le c
    \sum_{
      F\subset
      \check{\mathcal{R}}^{i}_{{\mathcal{K}}}
    }
    h_K^{{1}/{2}}
    \norm[0]{\jump{{v}_h}}_F.
  \end{split}
\end{equation*}

\subsection{Proof of \cref{lem:subgrid_proj_est}}
\label{s:subgridprojestproof}

We show \cref{eq:subgrid_proj_est} based on an idea in the proof
of \cite[Lemma 3.1]{Verfurth:1996}. On the reference element
$\widehat{\mathcal{K}}$, let $\widehat{v}_{\mathfrak{h}}$ be defined
as follows
\begin{equation*}
  \widehat{v}_\mathfrak{h}
  :=
  \begin{cases}
    \sum_{0\leq p_i\leq p_s,1\le i\le d}
    k_{p_1\dots p_d}^*
    \widehat{t}\widehat{x}_1^{p_1}\widehat{x}_2^{p_2}\cdots \widehat{x}_d^{p_d}
    +
    \sum_{0\leq p_i\leq p_s,1\le i\le d}
    b_{p_1\dots p_d}^*
    \widehat{x}_1^{p_1}\widehat{x}_2^{p_2}\cdots \widehat{x}_d^{p_d}
    &
    \text{ on }
    \widehat{\mathring{\mathcal{K}}}^*,
    \\
    \sum_{0\leq p_i\leq p_s,1\le i\le d}
    k_{p_1\dots p_d,*}
    \widehat{t}\widehat{x}_1^{p_1}\widehat{x}_2^{p_2}\cdots \widehat{x}_d^{p_d}
    +
    \sum_{0\leq p_i\leq p_s,1\le i\le d}
    b_{p_1\dots p_d,*}
    \widehat{x}_1^{p_1}\widehat{x}_2^{p_2}\cdots \widehat{x}_d^{p_d}
    &
    \text{ on }
    \widehat{\mathring{\mathcal{K}}}_*,
  \end{cases}
\end{equation*}
and let
\begin{multline*}
  \widehat{w}_h^\circ
  :=
  \sum_{0\leq p_i\leq p_s,1\le i\le d} \sbr[2]{
  \tfrac{1}{2}
  \del[0]{
    k_{p_1\dots p_d}^*
    +
    k_{p_1\dots p_d,*}
  }
  \widehat{t}\widehat{x}_1^{p_1}\widehat{x}_2^{p_2}\cdots \widehat{x}_d^{p_d}
  +
  \tfrac{1}{2}
  \del[0]{
    b_{p_1\dots p_d}^*
    +
    b_{p_1\dots p_d,*}
  }
  \widehat{x}_1^{p_1}\widehat{x}_2^{p_2}\cdots
  \widehat{x}_d^{p_d} }.
\end{multline*}
Then, by H\"older's inequality for sums and Fubini's theorem, we obtain
\begin{equation*}
  \begin{split}
    \norm[0]{
      \widehat{v}_{\mathfrak{h}}
      -\widehat{w}_h^\circ
    }_{\widehat{\mathcal{K}}}^2
    =&
    \int_{\widehat{\mathcal{K}}}
    \Big(
    \sum_{0\leq p_i\leq p_s}
    \tfrac{1}{2}
    \del[1]{
      k^*_{p_1\dots p_d}
      -
      k_{p_1\dots p_d,*}
    }
    \widehat{t}\widehat{x}_1^{p_1}\widehat{x}_2^{p_2}\cdots \widehat{x}_d^{p_d}
    \\
    &
    +
    \sum_{0\leq p_i\leq p_s}
    \tfrac{1}{2}
    \del[1]{
      b^*_{p_1\dots p_d}
      -
      b_{p_1\dots p_d,*}
    }
    \widehat{x}_1^{p_1}\widehat{x}_2^{p_2}\cdots \widehat{x}_d^{p_d}
    \Big)^2\dif \widehat{x} {\dif \widehat{t}}
    \le
    c \norm[0]{
      \jump{\widehat{v}_\mathfrak{h}}}_{\widehat{F}_{\mathring{\mathcal{R}}}}^2
    +
    \norm[0]{\jump{\partial_{\widehat{t}}\widehat{v}_\mathfrak{h}}}_{\widehat{F}_{\mathring{\mathcal{R}}}}^2.
  \end{split}
\end{equation*}
We conclude that
\begin{equation*}
  \inf_{\widehat{w}_h\in{Q}^{\del{1,p_s}}(\widehat{\mathcal{K}})}
  \norm[0]{\widehat{v}_\mathfrak{h}-\widehat{w}_h}_{\widehat{\mathcal{K}}}
  \le
  \norm[0]{\widehat{v}_\mathfrak{h}-\widehat{w}_h^\circ}_{\widehat{\mathcal{K}}}
  \le c
  \del[1]{
    \norm[0]{\jump{\widehat{v}_\mathfrak{h}}}_{\widehat{F}_{\mathring{\mathcal{R}}}} +
    \norm[0]{\jump{\partial_{\widehat{t}}\widehat{v}_\mathfrak{h}}}_{\widehat{F}_{\mathring{\mathcal{R}}}}
  }.
\end{equation*}
Since the $L^2$-projection is optimal \cite[eq.(18.32)]{Ern:book}, we
have shown \cref{eq:subgrid_proj_est_1} on the reference element,
i.e.,
\begin{equation*}
  \norm[0]{\del[0]{I-\widehat{i}^\mathcal{K}}\widehat{v}_\mathfrak{h}}_{\widehat{\mathcal{K}}}
  \le c
  \del[1]{
    \norm[0]{\jump{\widehat{v}_\mathfrak{h}}}_{\widehat{F}_{\mathring{\mathcal{R}}}}
    +
    \norm[0]{\jump{\partial_{\widehat{t}}\widehat{v}_\mathfrak{h}}}_{\widehat{F}_{\mathring{\mathcal{R}}}}
  }.
\end{equation*}
Scaling arguments
\cref{eq:lb7sudir-1,eq:lb7sudir-3,eq:scalingphiextra-3} now give us
\begin{equation}
  \label{eq:subgrid_proj_est_1_aff}
  \norm[0]{\del[0]{I-\widetilde{i}_h^\mathcal{K}}\widetilde{v}_\mathfrak{h}}_{\widetilde{\mathcal{K}}}
  \le c
  \del[0]{
    \delta t_{\mathcal{K}}^{{1}/{2}}\norm[0]{\jump{\widetilde{v}_\mathfrak{h}}}_{\widetilde{F}_{\mathring{\mathcal{R}}}}
    +
    \delta t_{\mathcal{K}}^{{3}/{2}}\norm[0]{\jump{\partial_{\widetilde{t}}\widetilde{v}_\mathfrak{h}}}_{\widetilde{F}_{\mathring{\mathcal{R}}}}
  }.
\end{equation}
Combining \cref{eq:subgrid_proj_est_1_aff},
\cref{eq:scalingequiv_1,eq:scalingequiv_3,eq:scalingphiextra-2}, we
conclude \cref{eq:subgrid_proj_est_1}. The proof of
\cref{eq:subgrid_proj_est_2} is similar, but in one lower spatial
dimension.

\subsection{Proof of \cref{lem:subgridhelperbnds}}
\label{s:subgridhelperbndsproof}

For \cref{eq:subgridhelperbnds_1}, we write the DG jump in terms of
HDG jumps by inserting the facet variable:
\begin{equation*}
  \begin{split}
    \del[0]{
      \jump{v_\mathfrak{h}}|_{F_{\mathring{\mathcal{R}}}}
    }^2
    &=
    \del[0]{
      \sbr[0]{\boldsymbol{v}_\mathfrak{h}}
      |_{\partial\mathring{\mathcal{K}}^*\cap F_{\mathring{\mathcal{R}}}}
      -
      \sbr[0]{\boldsymbol{v}_\mathfrak{h}}
      |_{\partial\mathring{\mathcal{K}}_*\cap F_{\mathring{\mathcal{R}}}}
    }^2
    \\
    &=
    \del[0]{
      \sqrt{2/3}\envert[0]{\beta_s-\tfrac{1}{2}\beta\cdot n}^{1/2}
      \sbr[0]{\boldsymbol{v}_\mathfrak{h}}|_{\partial\mathring{\mathcal{K}}^*\cap F_{\mathring{\mathcal{R}}}}
      -
      \sqrt{2}\envert[0]{\beta_s-\tfrac{1}{2}\beta\cdot n}^{1/2}
      \sbr[0]{\boldsymbol{v}_\mathfrak{h}}|_{\partial\mathring{\mathcal{K}}_*\cap F_{\mathring{\mathcal{R}}}}
    }^2,
  \end{split}
\end{equation*}
where we factor in
$\envert[0]{\beta_s-\tfrac{1}{2}\beta\cdot n}^{1/2}$ due to that
$\envert[0]{\beta_s-\tfrac{1}{2}\beta\cdot n}^{1/2} = \sqrt{3/2}$ on
an $\mathcal{R}$-facet if $n_t=-1$ and
$\envert[0]{\beta_s-\tfrac{1}{2}\beta\cdot n}^{1/2} = \sqrt{1/2}$ on
an $\mathcal{R}$-facet if $n_t=1$; \cref{eq:subgridhelperbnds_1} then
follows by the triangle inequality and the definition of
$\tnorm{\cdot}_{s,\mathfrak{h}}$.

For \cref{eq:subgridhelperbnds_2}, we expand the DG jump and apply the
triangle inequality; the trace inequality \cref{eq:eg_inv_4} then
concludes the bound.

To show \cref{eq:subgridhelperbnds_4} we require a more involved
splitting. For the edge $E_{\mathring{\mathcal{K}}}$ on a
$\mathcal{Q}$-facet $F_\mathcal{Q}$, we observe the following:
\begin{equation}
  \label{eq:edge_split}
    \ejump{\mu_\mathfrak{h}}|_{E_{\mathring{\mathcal{K}}}}
    =
    -\sbr[0]{\boldsymbol{v}_{\mathfrak{h}}^*}
    |_{\mathcal{Q}_{{\mathring{\mathcal{K}}}^*}\cap E_{\mathring{\mathcal{K}}}}
    +
    \sbr[0]{\boldsymbol{v}_{\mathfrak{h},*}}
    |_{\mathcal{Q}_{{\mathring{\mathcal{K}}}_*}\cap E_{\mathring{\mathcal{K}}}}
    +
    \ejump{v_{\mathfrak{h}}}|_{E_{\mathring{\mathcal{K}}}},
\end{equation}
where $\boldsymbol{v}_{\mathfrak{h}}^*$ and
$\boldsymbol{v}_{\mathfrak{h},*}$ denote the HDG solution pairs on
$\mathring{\mathcal{K}}^*$ and $\mathring{\mathcal{K}}_*$,
respectively. \Cref{eq:subgridhelperbnds_4} now follows by the
triangle inequality on
$\norm[0]{\ejump{\mu_\mathfrak{h}}}_{E_{\mathring{\mathcal{K}}}}$ and
using the trace inequalities \cref{eq:trace_ineq_subgrid_edge}.

To show \cref{eq:subgridhelperbnds_5} we again use the splitting
\cref{eq:edge_split}, followed by the triangle inequality, trace
inequalities \cref{eq:trace_ineq_subgrid_edge} and \cref{eq:eg_inv_4},
and inverse inequality \cref{eq:eg_inv_low_d_1_F}.

\subsection{Proof of \cref{lem:subgridprojdiff}}
\label{s:subgridprojdiffproof}

We verify the equivalence by showing that
${\widehat{i}_h^\mathcal{K}}\widehat{v}_\mathfrak{h} \equiv
{\widehat{i}_h^\mathcal{F}}
\widehat{v}_\mathfrak{h}$ on the reference domain.
On $\widehat{\mathcal{K}}$, let $\widehat{v}_{\mathfrak{h}}$
be defined as follows
\begin{equation*}
  \widehat{v}_\mathfrak{h}
  :=
  \begin{cases}
    \sum_{0\leq p_0\leq p_t,0\leq p_i\leq p_s,1\le i\le d}
    k_{p_0p_1\dots p_d}^*
    \widehat{t}^{p_0}\widehat{x}_1^{p_1}\widehat{x}_2^{p_2}\cdots \widehat{x}_d^{p_d}
    &
    \text{ on }
    \widehat{\mathring{\mathcal{K}}}^*,
    \\
    \sum_{0\le p_0\le p_t,0\leq p_i\leq p_s,1\le i\le d}
    k_{p_0p_1\dots p_d,*}
    \widehat{t}^{p_0}\widehat{x}_1^{p_1}\widehat{x}_2^{p_2}\cdots \widehat{x}_d^{p_d}
    &
    \text{ on }
    \widehat{\mathring{\mathcal{K}}}_*.
  \end{cases}
\end{equation*}
Suppose that
\begin{equation*}
  \widehat{i}_h^\mathcal{K}
  \widehat{v}_\mathfrak{h}
  =
  \sum_{0\le p_0\le p_t, 0\leq p_i\leq p_s,1\le i\le d}
  \widetilde{k}_{p_0p_1\dots p_d}
  \widehat{t}^{p_0}
  \widehat{x}_1^{p_1}
  \widehat{x}_2^{p_2}
  \cdots
  \widehat{x}_{d}^{p_{d}}.
\end{equation*}
By definition of the projection, for
any
$0\le q_0\le p_t$
and
$0\le q_i\le p_s$, $1\le i\le d$
\begin{equation*}
  \int_{\widehat{\mathcal{K}}}
  \del[0]{
    \widehat{v}_{\mathfrak{h}}
    -
    \widehat{i}_h^{\mathcal{K}}
    \widehat{v}_\mathfrak{h}
  }
  \del[0]{
    \widehat{t}^{q_0}
    \widehat{x}_1^{q_1}
    \widehat{x}_2^{q_2}
    \cdots
    \widehat{x}_d^{q_d}
  } \dif \widehat{x} \dif \widehat{t}
  =0.
\end{equation*}
Let us denote the $\mathcal{Q}$-facet on which $\widehat{x}_d=1$ by
$\widehat{F}_d$. Then, without loss of generality, using Fubini's
theorem,
\begin{align*}
  &
    \int_{\widehat{\mathring{\mathcal{K}}}^*}
    \sum_{0\le p_0\le p_t,0\leq p_i\leq p_s,1\le i\le d}
    \del[0]{
    k_{p_0\dots p_d}^*
    -
    \widetilde{k}_{p_0\dots p_d}
    }
    \widehat{t}^{p_0+q_0}
    \widehat{x}_1^{p_1+q_1}
    \cdots
    \widehat{x}_d^{p_d+q_d}
    \dif \widehat{x} \dif \widehat{t}
  \\
  &+
    \int_{\widehat{\mathring{\mathcal{K}}}_*}
    \sum_{0\le p_0\le p_t,0\leq p_i\leq p_s,1\le i\le d}
    \del[0]{
    k_{p_0\dots p_d,*}
    -
    \widetilde{k}_{p_0\dots p_d}
    }
    \widehat{t}^{p_0+q_0}
    \widehat{x}_1^{p_1+q_1}
    \cdots
    \widehat{x}_d^{p_d+q_d}
    \dif \widehat{x} \dif \widehat{t}
  \\
  =&
     \sum_{0\le p_0\le p_t,0\leq p_i\leq p_s,1\le i\le d}
     \int_{-1}^1
     \widehat{x}_d^{p_d+q_d}
     \dif\widehat{x}_d
     \Bigg(
     \int_{\widehat{\mathring{\mathcal{K}}}^*\cap\widehat{F}_d}
     \del[0]{
     k_{p_0\dots p_d}^*
     -
     \widetilde{k}_{p_0\dots p_d}
     }
     \widehat{t}^{p_0+q_0}
     \widehat{x}_1^{p_1+q_1}
     \cdots
     \widehat{x}_{d-1}^{p_{d-1}+q_{d-1}}
     \dif\widehat{s}
  \\
  &+
    \int_{\widehat{\mathring{\mathcal{K}}}_*\cap\widehat{F}_d}
    \del[0]{
    k_{p_0\dots p_d,*}
    -
    \widetilde{k}_{p_0\dots p_d}
    }
    \widehat{t}^{p_0+q_0}
    \widehat{x}_1^{p_1+q_1}
    \cdots
    \widehat{x}_{d-1}^{p_{d-1}+q_{d-1}}
    \dif\widehat{s}
    \Bigg).
\end{align*}
Note that $\int_{-1}^1 \widehat{x}_d^{p_d+q_d} \dif\widehat{x}_d = 0$
for $p_d + q_d$ odd. Then, for each $0\le q_d\le p_s$, leaving out
$p_d$'s such that $p_d+q_d$ is odd, using that
$\int_{-1}^1 \widehat{x}_d^{2k}\dif \widehat{x}_d = 2/(2k+1)$, and
introducing
\begin{multline*}
  z_{p_d}
  :=
  \sum_{0\le p_0\le p_t,0\leq p_i\leq p_s,1\le i\le d-1}
  \Big(
  \int_{\widehat{\mathring{\mathcal{K}}}^*\cap\widehat{F}_d}
  \del[0]{
    k_{p_0\dots p_d}^*
    -
    \widetilde{k}_{p_0\dots p_d}
  }
  \widehat{t}^{p_0+q_0}
  \widehat{x}_1^{p_1+q_1}
  \cdots
  \widehat{x}_{d-1}^{p_{d-1}+q_{d-1}}
  \dif\widehat{s}
  \\
  +
  \int_{\widehat{\mathring{\mathcal{K}}}_*\cap\widehat{F}_d}
  \del[0]{
    k_{p_0\dots p_d,*}
    -
    \widetilde{k}_{p_0\dots p_d}
  }
  \widehat{t}^{p_0+q_0}
  \widehat{x}_1^{p_1+q_1}
  \cdots
  \widehat{x}_{d-1}^{p_{d-1}+q_{d-1}}
  \dif\widehat{s}
  \Big),
\end{multline*}
we have
\begin{equation}
  \label{eq:2zpdequation}
  \sum_{p_d\text{ s.t. }p_d+q_d\text{ is even}}
  \frac{2z_{p_d}}{p_d+q_d+1}
  =0.
\end{equation}
It is possible to write \cref{eq:2zpdequation} as a linear system
where the system matrix is a $2\times 2$ block-diagonal matrix in
which each block is a Hankel matrix. Furthermore, each block can be
shown to be totally positive, see \cite[Example 0.1.8]{Fallat:book}.
We therefore conclude that the system matrix of \cref{eq:2zpdequation}
is nonsingular and so $z_i=0$ for $0\le i\le p_s$, i.e.,
\begin{equation}
  \label{eq:subgridfacetproj}
  \begin{split}
    &\sum_{0\le p_0\le p_t,0\leq p_i\leq p_s,1\le i\le d-1}
    \Big(
    \int_{\widehat{\mathring{\mathcal{K}}}^*\cap\widehat{F}_d}
    \del[0]{
      k_{p_0\dots p_d}^*
      -
      \widetilde{k}_{p_0\dots p_d}
    }
    \widehat{t}^{p_0+q_0}
    \widehat{x}_1^{p_1+q_1}
    \cdots
    \widehat{x}_{d-1}^{p_{d-1}+q_{d-1}}
    \dif\widehat{s}
    \\
    &\qquad\qquad\qquad+
    \int_{\widehat{\mathring{\mathcal{K}}}_*\cap\widehat{F}_d}
    \del[0]{
      k_{p_0\dots p_d,*}
      -
      \widetilde{k}_{p_0\dots p_d}
    }
    \widehat{t}^{p_0+q_0}
    \widehat{x}_1^{p_1+q_1}
    \cdots
    \widehat{x}_{d-1}^{p_{d-1}+q_{d-1}}
    \dif\widehat{s}
    \Big)=0,
  \end{split}
\end{equation}
for any $0\le q_0\le p_t$ and $0\le q_i\le p_s$, $1\le i\le
d-1$. Observing that
\begin{equation*}
  \begin{split}
    \widehat{\mu}_{\mathfrak{h},p_d}
    &:=
    \begin{cases}
      \sum_{0\leq p_0\leq p_t,0\leq p_i\leq p_s,1\le i\le d-1}
      k_{p_0\dots p_d}^*
      \widehat{t}^{p_0}\widehat{x}_1^{p_1}\widehat{x}_2^{p_2}\cdots
      \widehat{x}_{d-1}^{p_{d-1}}
      &
      \text{ on }
      \widehat{\mathring{\mathcal{K}}}^*\cap\widehat{F}_d,
      \\
      \sum_{0\le p_0\le p_t,0\leq p_i\leq p_s,1\le i\le d-1}
      k_{p_0\dots p_d,*}
      \widehat{t}^{p_0}\widehat{x}_1^{p_1}\widehat{x}_2^{p_2}\cdots
      \widehat{x}_{d-1}^{p_{d-1}}
      &
      \text{ on }
      \widehat{\mathring{\mathcal{K}}}_*\cap\widehat{F}_d,
    \end{cases}
    \\
    \widehat{\lambda}_{\mathfrak{h},p_d}
    &:=
    \sum_{0\le p_0\le p_t,0\leq p_i\leq p_s,1\le i\le d-1}
    \widetilde{k}_{p_0\dots p_d}
    \widehat{t}^{p_0}\widehat{x}_1^{p_1}\widehat{x}_2^{p_2}\cdots
    \widehat{x}_{d-1}^{p_{d-1}}
    \text{ on }
    \widehat{F}_d,
  \end{split}
\end{equation*}
we conclude from \cref{eq:subgridfacetproj} that
$\widehat{i}_h^{\mathcal{F}}\widehat{\mu}_{\mathfrak{h},p_d}=\widehat{\lambda}_{\mathfrak{h},p_d}$.
Further, observing that
$\widehat{v}_{\mathfrak{h}}|_{\widehat{F}_d} = \sum_{p_d}
\widehat{\mu}_{\mathfrak{h},p_d}$, that
$\del[0]{\widehat{i}_h^\mathcal{K}\widehat{v}_\mathfrak{h}}|_{\widehat{F}_d}=\sum_{p_d}\widehat{\lambda}_{\mathfrak{h},p_d}$,
and that projection is linear, we conclude that
$\del[0]{\widehat{i}_h^\mathcal{K}\widehat{v}_\mathfrak{h}}|_{\widehat{F}_d}=
\widehat{i}_h^\mathcal{F}\del[0]{\widehat{v}_\mathfrak{h}|_{\widehat{F}_d}}$.

\section*{Acknowledgments}

This research was enabled in part by support provided by Simon
Fraser University
(\url{https://www.sfu.ca/research/supercomputer-cedar}), Compute
Ontario (\url{https://www.computeontario.ca/}), Calcul Qu\'ebec
(\url{https://www.calculquebec.ca/})  and the Digital Research
Alliance of Canada (\url{https://alliancecan.ca}).

\bibliographystyle{abbrvnat}
\bibliography{references}
\end{document}